\documentclass{amsart}

\usepackage{forest}
\usepackage{nicefrac}

\usepackage{subcaption}
\usepackage{graphicx} 

\usetikzlibrary{shapes.geometric,decorations.markings}
\usetikzlibrary{decorations.pathreplacing}
\usetikzlibrary{fit}
\usetikzlibrary{automata}
\usetikzlibrary{positioning}

\usepackage{chngcntr}
\counterwithin{figure}{section}

\usepackage[utf8]{inputenc}
\usepackage[T1]{fontenc}
\usepackage{textcomp}
\usepackage{amsfonts}
\usepackage{amsthm}
\usepackage{amssymb}
\usepackage{pst-node}
\usepackage[all]{xy}
\usepackage{mathtools}
\usepackage{chngcntr}
\usepackage{thmtools}
\usepackage{mathrsfs}
\usepackage{tcolorbox}
\usepackage{listings}
\usepackage{hyperref}
\usepackage{color}
\usepackage{float}
\usepackage{tikz, tikz-cd}
\usepackage{enumerate}
\usepackage[nameinlink, capitalize, noabbrev]{cleveref}
\usepackage{natbib}
\usepackage{caption}
\usetikzlibrary{intersections}
\usepackage{pgfplots}
\pgfplotsset{width=10cm,compat=1.9}
\usepgfplotslibrary{external}
\tikzset{mytext/.style={font=\small, text=black}}
\tikzset{main node/.style={circle,fill=lime!30,draw,minimum size=0.5cm,inner sep=0pt},}

\hypersetup{ colorlinks=true,linkcolor=teal,    citecolor=magenta, }

\usepackage{pst-node}
\usepackage{pst-tree}

\def\BALL[#1](#2){\rput[t](#2){}%
        \pscircle[fillstyle=solid,fillcolor=#1!40](#2){5pt}}
        
\psset{treesep=1.3,levelsep=1.5}

\hypersetup{
    colorlinks=true,
    linkcolor=teal,
    citecolor=magenta,
    }

\numberwithin{equation}{section}

\newtheorem{Theorem}{Theorem}
\newtheorem{Conjecture}[Theorem]{Conjecture}
\newtheorem{Corollary}[Theorem]{Corollary}
\newtheorem{Definition}[Theorem]{Definition}
\newtheorem{proposition}{Proposition}[section]
\newtheorem{lemma}[proposition]{Lemma}
\newtheorem{corollary}[proposition]{Corollary}
\newtheorem{theorem}[proposition]{Theorem}

\newtheorem{example}[proposition]{Example}

\newtheorem{convention}[proposition]{Convention}
\theoremstyle{definition}
\newtheorem{remark}[proposition]{Remark}
\newtheorem{definition}[proposition]{Definition}

\def\NN{\mathbb N}
\def\ZZ{\mathbb Z}
\def\RR{\mathbb R}
\def\CC{\mathbb C}

\def\PP{\mathbb P}
\def\K-{\overline K}
\def\FF{\mathbb F}
\def\Ksep{K^{\textrm{sep}}}
\def\OK{\mathcal{O}_K}

\DeclareMathOperator{\st}{st}
\DeclareMathOperator{\FPP}{FPP}

\DeclareMathOperator{\Gal}{Gal}
\DeclareMathOperator{\St}{St}
\DeclareMathOperator{\Aut}{Aut}
\DeclareMathOperator{\IMG}{IMG}

\DeclareMathOperator{\Sym}{Sym}

\DeclareMathOperator{\Per}{Per}
\DeclareMathOperator{\IIm}{Im}
\DeclareMathOperator{\id}{id}
\DeclareMathOperator{\lcm}{lcm}
\DeclareMathOperator{\Dom}{Dom}

\DeclareMathOperator{\Par}{Par}

\def \<#1>{{\left\langle{#1}\right\rangle}}
\def\abs#1{\left\vert{#1}\right\vert}
\def\set#1{\left\{#1\right\}}

\title[]{The inverse Galois problem of iterated Galois groups and their fixed-point proportion}
\author{Santiago Radi}
\address{Santiago Radi: Department of Mathematics, Texas A\&M University, 77843 College Station, U.S.A.}
\email{santiradi19@gmail.com}
\thanks{The author is supported by Grigorchuk's Simons Foundation Grant MP-TSM-00002045 and the department of Mathematics of Texas A\&M University and Universidad de la República.}
\date{August 2026}

\begin{document}

\begin{abstract}
In 1985, Odoni initiated the study of arboreal representations and the fixed-point proportion, motivated by prime density problems in arithmetic dynamics. Since then, many questions regarding the connection between the dynamics of rational functions and the Galois groups associated to their dynamics (iterated Galois groups) have been posed.

In this article, we give several contributions to this connection with the introduction of the concept of virtually mixing groups. This new concept is related to the capability of self-replication of the iterated Galois groups, and when it fails to be completely self-replicated, the failure is always by a finite index normal subgroup. It turns out that the quotient by this subgroup contains valuable information of the properties of the rational function.

First, we investigate the inverse Galois problem, showing that iterated Galois groups of rational functions are always virtually mixing and the failure (the quotient by this finite-index subgroup) is related to geometric properties of the rational function. As a consequence, we prove that the rational function is induced by an endomorphism of an algebraic curve (dynamical pullback) if and only if this failure is non-trivial.

Finally, we solve the main open problem of the fixed-point proportion of geometric iterated Galois groups of rational functions, by showing that the fixed-point proportion is zero when the map is not a dynamical pullback. This result has direct applications to prime density problems, proportion of periodic points in the reduction of maps and proportion of periodic points over finite fields.

The proofs in this article rely on a combination of techniques and results from group theory, ergodic theory, probability, number theory, complex dynamics, energy transport in graphs, arithmetic dynamics and algebraic geometry.
\end{abstract}

\maketitle

\section{Introduction}
\label{section: Introduction}

Let $L$ be a separably closed field, $K$ a subfield of $L$ and $a$ an element of $L$. Fix $K(a)^{\mathrm{sep}}$ a separable closure of $K(a)$ in $L$. Let $f$ be a \textit{tamely ramified} rational function in $K(x)$ of degree $d \geq 2$, namely, the characteristic of $K$ does not divide the local degree of $f$ at any point of $\PP^1(\K-)$. We write $f^n$ for the $n$th composition of $f$ with itself and denote $f^{-n}(a)$ the set of preimages of $a$ under $f^n$. Construct 
$$T_a:=\bigsqcup_{n\ge 0}f^{-n}(a)$$ 
the rooted tree of preimages of $a$ where a vertex $w$ in level $n+1$ is connected to a vertex $v$ in level $n$ if $f(w) = v$. This induces a natural action of $\Gal(K(a)^{\textrm{sep}}/K(a))$ on $T_a$ preserving adjacency, which gives rise to the so-called \textit{arboreal representation} 
$$\rho_{\textrm{arb}}: \mathrm{Gal}(K(a)^{\mathrm{sep}}/K(a)) \rightarrow \Aut(T_a).$$ 

The image of this representation, denoted $G_\infty(K,f,a)$, is known as the \textit{iterated Galois group} of $f$. The main cases of interest are when $a \in K$ or $a = t$ is transcendental over $K$. In this latter case, one can prove that $T_t$ is a $d$-regular rooted tree. The group $G_\infty(K,f,t)$ is called the \textit{arithmetic iterated Galois group} or \textit{arithmetic iterated monodromy group} of $f$, and in the case that $K$ is separably closed, the group is called \textit{geometric iterated Galois group} or \textit{geometric iterated monodromy group}.

Due to its relevance to problems in arithmetic dynamics and complex dynamics, understanding properties of iterated Galois groups in terms of the rational function $f$ has been of main interest in the area. In this direction, there are conjectures involving when iterated Galois groups are branch, abelian, have positive Hausdorff dimension or have finite index in $\Aut(T)$, and how these properties relate to the dynamical properties of $f$ \cite{AndrewsPetsche2020, Fariña2026FPP, LeungPetsche2025, Pink2013nPCF, Pink2013PCF, Radi2025FPP}. 

On the other hand, one can study the inverse Galois problem of this family, namely, understanding which subgroups of $\Aut(T)$ can be realizable as arithmetic iterated Galois groups. For example, it is well-known that arithmetic iterated Galois groups are closed in $\Aut(T)$ and fractal \cite[Theorem 7]{Adams2025}, which in particular implies that these groups are self-similar and level-transitive (see \cref{section: groups acting on trees} for undefined terms here). 

In \cite{BostonJones2006}, Boston and Jones conjectured that for every irreducible quadratic polynomial with coefficients in a number field $K$ the group $G_\infty(K,f,t)$ is \textit{densely settled}, this is, there is a dense set of elements $g$ in $G_\infty(K,f,t)$ such that the proportion of vertices in level $n$ that are in stable cycles of $g$ approaches to $1$ when $n \rightarrow +\infty$. Although the conjecture has not been confirmed for every irreducible quadratic polynomial, it is known when the critical point of $f$ in $K$ is periodic or when it is preperiodic with orbit length $2$ \cite{CortezLukina2022,EjderKocak2026}.

In \cite{FariñaRadi2026FPP}, Fariña-Asategui and the author introduced the notion of \textit{mixing groups}. Recall that if $g$ is an element in $\Aut(T)$ and $v$ is a vertex in $T$, the section of $g$ at $v$, denoted $g|_v$, is another element in $\Aut(T)$ that encodes the action of $g$ below the subtree rooted at $v$. If $G$ is a subgroup of $\Aut(T)$ and $H$ is a subgroup of $G$,
\begin{align*}
H_v := \set{h|_v \in \Aut(T): h \in \st_G(v)}
\end{align*}
is the collection of all the sections of elements in $H$ that fix $v$. In the case that $G$ is self-similar, it turns out that $H_v$ is a subgroup of $G$.

A fractal group $G$ of $\Aut(T)$ is \textit{mixing} if there exists $N \geq 0$ such that for every $n \geq 1$ and every vertex $v \in T$ whose distance to the root is at least $n + N$, we have 
\begin{align*}
\St_G(n)_v =  G,  
\end{align*}
where here $\St_G(n)$ is the stabilizer of level $n$. 

This definition corresponds to a generalization of super strongly fractal groups, well-studied in groups acting on trees (see \cite{Fariña2025ergodic, FariñaRadi2025RandomSubgroups, UriaAlbizuri2016}). In \cite[Theorem 6]{FariñaRadi2026FPP}, authors prove that every geometric iterated Galois group of a polynomial that is not linearly conjugate to a Chebyshev polynomial is mixing. 

The fact that maps linearly conjugate to Chebyshev polynomials are excluded can be understood from the point of view of complex dynamics. If $f$ is a post-critically finite rational function over $\CC$, the Euler characteristic of the Thurston orbifold of $f$ determines exceptional dynamical properties of $f$. For example, the Euler characteristic is zero if and only if the Julia set is the unit circle, a line segment or the whole Riemann sphere \cite[Theorem 14.6]{Milnor2006}. This is opposite to the typical well-known cases where the Julia sets are convoluted fractal sets. The Thurston orbifold and its Euler characteristic can be generalized for any tamely ramified post-critically finite rational function defined over any field (see \cref{definition: Thurston orbifold}) and as in the case of complex numbers, it helps to detect an exceptional dynamical behavior. 

By Riemann-Hurwitz formula (see \cref{Theorem: Riemann Hurwitz formula}) the Euler characteristic of a tamely ramified post-critically finite rational function is either negative or zero. When the Euler characteristic is zero, the map is said to have \textit{euclidean orbifold} or be \textit{exceptional}. In the case of complex numbers, exceptional maps are completely classified by the work of Douady and Hubbard, proving that if $f$ is exceptional then $f$ is linearly conjugate over $\CC$ to a power map $x^{\pm d}$, a Chebyshev polynomial $\pm T_d$ or to a Lattés map \cite{DoaudyHubbard1993}. 

A minor modification in the definition of mixing groups allows us to discover a property that is satisfied by every geometric iterated Galois group. 

\begin{Definition}
We say that a fractal group $G$ in $\Aut(T)$ is \textit{virtually mixing} if there exists $N \geq 0$ and 
$H \leq_f G$ such that for every $n \geq 1$ and every vertex $v \in T$ whose distance to the root is at least $n + N$, we have 
\begin{align*}
H \leq \St_G(n)_v.  
\end{align*}
We call $N$ a \textit{delay constant} of $G$ and $H$ a \textit{mixing subgroup} of $G$.
\end{Definition}

The choice of the name `virtually mixing' comes from the fact that the condition of mixing is now satisfied by a subgroup of finite index and not necessarily by the whole group $G$. 

Our first main result is the following:

\begin{Theorem}
Let $f$ be a tamely ramified rational function of degree $d \geq 2$ defined over a field $K$. Furthermore, if $f$ is exceptional, assume that $K$ has characteristic zero and its cardinality is at most $2^{\aleph_0}$. Let $G = G_\infty(K^{\textrm{sep}},f,t)$ be its geometric iterated Galois group and assume that for every $n \geq 1$ the subgroup $\St_G(n)$ acts transitively below every vertex at level $n$. Then

\begin{enumerate}[\normalfont(1)]
\item $G$ is virtually mixing and the delay constant $N$ only depends on the ramification portrait of $f$. 

\item There exists a mixing subgroup $G_m$ such that $G_m$ is closed and normal in $G$ and $G/G_m$ is a finite Von Dyck group.

\item The subgroup $G_m$ is explicitly given in terms of generators of the inertia subgroups of each post-critical point of $f$.
\end{enumerate}
\label{theorem: every rational is virtually mixing}
\end{Theorem}

For short, and for reasons justified by \cref{proposition: martingale characterization}, we say that $G$ is a \textit{martingale group} if for every $n \geq 1$ the subgroup $\St_G(n)$ acts transitively below every vertex at level $n$. It is worth mentioning that the hypothesis in \cref{theorem: every rational is virtually mixing} are natural in this area although they are not satisfied by every rational function. For instance, it is well-known that iterated Galois groups are martingale groups when at least one of the following conditions occur:

\begin{enumerate}[\normalfont(1)]
\item $f$ is a polynomial,
\item the degree of $f$ is prime,
\item the action of $G_\infty(K,f,t)$ on the first level is double transitive.
\end{enumerate}
(see \cite[Section 5]{BridyJones2022} and \cite[Theorem 4.2]{Jones2012}). Nevertheless, examples of iterated Galois groups of rational functions that are not martingale groups have also been constructed (see \cite{HeZhu2026}). 

The proof of \cref{theorem: every rational is virtually mixing} will be done in \cref{section: iterated Galois groups are virtually mixing} and it will be split in two parts. The first part concerns all rational functions that are not exceptional. For this first case, we will use a new argument based on energy transportation in graph theory. The idea is to see the ramification portrait of a rational function $f$ as a weighted directed graph and find the optimal paths that lead us to special elements (see \cref{section: energy of paths} for more details). The second part only deals with exceptional maps and it is based on the work of Douady and Hubbard previously mentioned. The hypothesis imposed over the field $K$ of having cardinality $2^{\aleph_0}$ is not really restrictive for the typical applications of iterated Galois groups, since this condition is for example satisfied by number fields and the complex numbers.

Although \cref{theorem: every rational is virtually mixing} refers only to the geometric iterated Galois group, one can still prove that the arithmetic iterated Galois group is virtually mixing in many cases. Recall the exact sequence
\begin{align*}
1 \rightarrow G_\infty(\Ksep, f, t) \rightarrow G_\infty(K, f, t) \rightarrow \Gal(\Ksep \cap K_\infty(f,t)/K) \rightarrow 1,    
\end{align*}
that relates the geometric and the arithmetic iterated Galois groups of a rational function (see \cref{proposition: short exact sequence arithmetic and geometric Galois groups} for a proof). Here $K_\infty(f,t) := K \left( \bigcup_{n \geq 0} f^{-n}(t) \right)$ . 

The extension $\Ksep \cap K_\infty(f,t)/K$ is understood in specific cases of $f$ (see for example \cite{AdamsHyde2025, BenedettoGhiocaJuulTucker2025periodic, BenedettoGhiocaJuulTucker2025preperiodic, EjderGolaskaKara2026, Pink2013nPCF, Pink2013PCF} but it is still mysterious in general. In \cite{Pink2013PCF}, Pink provides examples of quadratic polynomials where $\Ksep \cap K_\infty(f,t)/K$ can be finite or infinite depending on the length of the orbit of the critical point in $K$. In \cite[Theorem 2]{Fariña2026FPP}, Fariña-Asategui shows that if $G_\infty(K,f,t)$ (or equivalently $G_\infty(\Ksep,f,t)$) is branch then $\Ksep \cap K_\infty(f,t)/K$ is a finite extension.

As a corollary of \cref{theorem: every rational is virtually mixing}, we obtain the following:

\begin{Corollary}
Let $f$ be a tamely ramified rational function of degree $d \geq 2$ defined over a field $K$. Furthermore, if $f$ is exceptional, assume that $K$ has characteristic zero and its cardinality is at most $2^{\aleph_0}$. Let $G = G_\infty(K,f,t)$ be its arithmetic iterated Galois group. Assume that $G$ is a martingale group and the extension $\Ksep \cap K_\infty(f,t)/K$ is finite. Then

\begin{enumerate}[\normalfont(1)]
\item $G$ is virtually mixing and the delay constant $N$ only depends on the ramification portrait of $f$ (same constant as in \cref{theorem: every rational is virtually mixing}). 

\item There exists a mixing subgroup $G_m$ such that $G_m$ is closed and normal in $G$ and $G/G_m$ satisfies the exact sequence
\begin{align*}
1 \rightarrow V \rightarrow G/G_m \rightarrow \Gal(\Ksep \cap K_\infty(f,t)/K) \rightarrow 1,    
\end{align*}
where $V$ is a finite Von Dyck group.
\end{enumerate}
\label{Corollary: arithmetic IGG is virtually mixing}
\end{Corollary}

\textit{Von Dyck groups} or \textit{triangle groups} are groups given by the presentation
\begin{align*}
\Delta(n_1,n_2,n_3) := \<a_1,a_2,a_3 \,\,| \,\, a_i^{n_i}, \, a_1 a_2 a_3>.
\end{align*}
It is well-known that a Von Dyck group is finite if and only if 
$$\frac{1}{n_1} + \frac{1}{n_2} + \frac{1}{n_3} > 1.$$

This diophantine inequality has finitely many solutions and the groups obtained in each case are well-known (see \cref{theorem: classification Von Dyck groups}). 

The fact that Von Dyck groups appear in \cref{theorem: every rational is virtually mixing} is not a coincidence. Von Dyck groups are exactly the finite subgroups of $\Gal(\K-(x)/\K-)$, which corresponds to the function field of $\PP^1(\K-)$, evincing a relation between mixing subgroups and algebraic geometry. 

Modifying slightly the notation in \cite{Adams2025}, we say that a rational function $f$ is a \textit{dynamical pullback} if there exists an irreducible smooth projective curve $C$, a Galois cover $\pi: C \rightarrow \PP^1(\K-)$ of degree greater than $1$ and an endomorphism $F: C \rightarrow C$ such that the diagram
\begin{equation*}
\centering
\begin{tikzcd}
C \arrow{rr}{F} \arrow{dd}{\pi} & & C \arrow{dd}{\pi} \\
 & & \\
\PP^1(\K-) \arrow{rr}{f} & & \PP^1(\K-)
\end{tikzcd}
\end{equation*}
is a pullback and $K(\pi^{-1}(t))$ is contained in $K_\infty(f,t)$. By the Riemann-Hurwitz formula, the genus of $C$ can be either $0$ (and so $C$ is isomorphic to $\PP^1(\K-)$) or $1$ (and so $C$ is an elliptic curve). In the case that the genus is zero, then the function field of $\PP^1(\K-)$ embeds into the function field of $C$, and $\Gal(\K-(C)/K(\PP^1(K))$ is a finite subgroup of $\Gal(\K-(x)/\K-)$, which corresponds to a finite Von Dyck group. Well-known examples of dynamical pullbacks are Lattés maps and Chebyshev polynomials. 

As it was mentioned previously, this relation is not a coincidence and dynamical pullbacks can be characterized in terms of the maximal mixing subgroup of $G_\infty(\Ksep,f,t)$. If $G$ is a virtually mixing group with delay constant $N$, we define the \textit{maximal mixing subgroup} as
\begin{align*}
G_M := \bigcap_{n \geq 1} \bigcap_{v \in \mathcal{L}_{n+N}} \St_G(n)_v,
\end{align*}
where $\mathcal{L}_{n+N}$ is the set of vertices at level $n+N$ of the tree $T$. One can prove (see \cref{proposition: properties GM}) that $G_M$ is a mixing subgroup of $G$ and any other mixing subgroup of $G$ is contained in $G_M$. The characterization is the following:

\begin{Theorem}
Let $f$ be a tamely ramified rational function of degree $d \geq 2$ defined over a field $K$ such that $G_\infty(\Ksep,f,t)$ is a martingale group. Then, $f$ is a dynamical pullback if and only if the maximal mixing subgroup of $G_\infty(\Ksep,f,t)$ is not $G_\infty(\Ksep,f,t)$.
\label{theorem: dynamical pullback maximal mixing subgroup}
\end{Theorem}

The nature of $G_M$ is theoretical and its definition makes it hard to explicitly calculate it in practice. However, the subgroup $G_m$ mentioned in \cref{theorem: every rational is virtually mixing} is explicit (see \cref{definition: Gm}). Fortunately, the subgroups coincide in certain cases.

If $G$ is a self-similar group and $H \leq G$, we say that $H$ is \textit{project invariant} if $H_v \leq H$ for all $v \in T$.

\begin{Theorem}
Let $f$ be a tamely ramified rational function of degree $d \geq 2$ defined over a field $K$ such that $G = G_\infty(\Ksep,f,t)$ is a martingale group. Let $G_M$ be the maximal mixing subgroup of $G$ and $G_m$ the subgroup from \cref{theorem: every rational is virtually mixing}. If $G_m$ is project invariant, then $G_m = G_M$.  
\label{theorem: Gm is GM}
\end{Theorem}

Finally, we show applications to concrete problems in number theory. If $G$ is a closed subgroup of $\Aut(T)$, then $G$ has a unique normalized Haar measure $\mu_G$. The \textit{fixed-point proportion} of $G$ is defined as 
\begin{align*}
\FPP(G) := \mu_G(\set{g \in G: \text{$g$ fixes at least one vertex at every level of $T$}}).
\end{align*}
The set in the definition is a Borel subset of $G$ and therefore the fixed-point proportion of $G$ is always well-defined (see \cite[Lemma 2.9]{Radi2025FPP} for an alternative proof).

There are three interesting problems where the fixed-point proportion of an iterated Galois group is involved. 

\vspace{0.75em}

\textbf{Application 1} The first application is related to the density of prime divisors of a sequence obtained from the iteration of a rational function \cite{JonesManes2012}. 
Let $K$ be a number field with ring of integers $\OK$. Let  $M_K^0$ be the set of prime ideals in $\OK$, take $a_0 \in K$ and take a rational function $f\in K(x)$ of degree $d \geq 2$. Assume that the numerator of $f^n$ is irreducible for all $n \geq 0$ and that exists $n_0 \geq 1$ such that $f^n(\infty) \neq 0$ for all $n \geq n_0$. Write 
$$P_f(a_0) := \{\mathfrak{p} \in M_K^0: v_\mathfrak{p}(f^n(a_0)) > 0  \text{ for some } n \geq 0, \text{ with } f^n(a_0) \neq 0, \infty \},$$ 
i.e. $P_f(a_0)$ is the set of primes dividing some non-zero iterate $f^n(a_0)$. If $\mathcal{D}$ denotes the Dirichlet density, then
\begin{align*}
\mathcal{D}(P_f(a_0)) \leq \mathrm{FPP}(G_\infty(K,f,0)).
\end{align*}

\textbf{Application 2} The second application is related to the proportion of periodic points in reduction of rational functions modulo prime ideals \cite{BridyJones2022, Juul2014, Radi2026}. Let $K$ be a number field and $f$ a rational function of degree $d \geq 2$ with coefficients in $K$. If $\mathfrak{p}$ is a prime ideal in $\OK$, let $\Per(f_\mathfrak{p}, \PP^1(\FF_\mathfrak{p}))$ denote the set of periodic points of the reduction of $f$ modulo $\mathfrak{p}$ in the residue field. Then 
\begin{align*}
\Per(f,K) := \lim_{N(\mathfrak{p}) \rightarrow \infty} \frac{\# \Per(f_\mathfrak{p}, \PP^1(\FF_\mathfrak{p}))}{N(\mathfrak{p})+1} \leq \FPP(G_\infty(K,f,t)),
\end{align*}
and
\begin{align*}
\Per_{\inf}(f,K) := \liminf_{N(\mathfrak{p}) \rightarrow \infty} \frac{\# \Per(f_\mathfrak{p}, \PP^1(\FF_\mathfrak{p}))}{N(\mathfrak{p})+1} \leq \FPP(G_\infty(\K-,f,t)),
\end{align*}
where $t$ is transcendental over $\K-$.

\vspace{0.75em}

\textbf{Application 3} The third application is also related to the proportion of periodic points in finite fields but this time we fix the finite field and take extensions of it \cite{BridyJones2022}. Let $\mathbb{F}_q$ be a finite field of characteristic $p$ and $f$ a rational function with coefficients in $\mathbb{F}_q$ and degree $d$ such that $2 \leq d < p$. If $n \geq 1$, let $\Per(f, \PP^1(\FF_{q^n}))$ denote the set of periodic points of $f$ in $\PP^1(\FF_{q^n})$. Then 
\begin{align*}
\Per_{\inf}(f,\FF_q) := \liminf_{n \rightarrow \infty} \frac{\# \Per(f, \PP^1(\FF_{q^n}))}{q^n+1} \leq \FPP(G_\infty(\overline{\mathbb{F}_q},f,t)),
\end{align*}
where $t$ is transcendental over $\overline{\mathbb{F}_q}$.

These applications strengthen the interest of classifying the fixed-point proportion of iterated Galois groups for every rational function, specially of understanding for which rational functions the proportion is zero.  

The first observation is a very simple property of the fixed-point proportion: if $H \leq_f G$ and $\FPP(G) = 0$, then $\FPP(H) = 0$ (see \cite[Lemma 2.10]{Radi2025FPP} for a proof). Therefore, if $a \in K$, in order to prove that $\FPP(G_\infty(K,f,a))$ is zero, it is enough to find a finite-index extension of $G_\infty(K,f,a)$ with null fixed-point proportion. The \textit{specialization problem} states that if $a$ is not a post-critical point of $f$ and $t$ is transcendental over $K$, it is expected that $G_\infty(K,f,a) \leq_f G_\infty(K,f,t)$. This finite index condition has in fact been verified in several cases (see for example \cite{BenedettoGhiocaJuulTucker2025specialization,  BridyDoyleGhiocaTucker2021, BridyDoyleGhiocaTucker2021second, BridyTucker2019, JonesManes2012, Odoni1985}) and has been shown in \cite{Tucker2013ABC} that is related to ABC and Vojta's conjectures. For this reason, it is logic to focus on a classification of the fixed-point proportion in arithmetic iterated Galois groups. 

In \cref{Corollary: arithmetic IGG is virtually mixing}, we proved that in many cases the arithmetic iterated Galois groups are virtually mixing. The following result gives a tool to prove that the fixed-point proportion of a virtually mixing group is zero: 

\begin{Theorem}
\label{theorem: FPP virtually mixing groups}
Let $G \leq \Aut(T)$ be a martingale virtually mixing group with mixing subgroup $H$. If at every coset $S$ of $G/H$ there exists an element $s \in S$ such that $s$ fixes at least one vertex at every level of the tree, then $$\FPP(G) = 0.$$
\end{Theorem}

\cref{theorem: FPP virtually mixing groups} is a combination of the techniques presented by the authors in \cite{FariñaJonesRadi2026corrigendum} and \cite{FariñaRadi2026FPP}, which gives rise to a tool that is applicable to a wider family of rational functions. Combining Theorems \ref{theorem: every rational is virtually mixing}, \ref{theorem: dynamical pullback maximal mixing subgroup} and \ref{theorem: FPP virtually mixing groups}, we obtain the following result:

\begin{Theorem}
\label{theorem: not dynamical pullbacks and FPP}
Let $f$ be a tamely ramified rational function of degree $d \geq 2$ defined over a field $K$ such that $G = G_\infty(\Ksep,f,t)$ is a martingale group. If $f$ is not a dynamical pullback, then $$\FPP(G_\infty(\Ksep,f,t)) = 0.$$
\end{Theorem}

\cref{theorem: not dynamical pullbacks and FPP} answers the longstanding open problem of calculating the fixed-point proportion of geometric iterated Galois groups for rational functions not arising from algebraic curves and confirms the belief accepted in the area that the only possible iterated Galois groups with positive fixed-point proportion may only come from maps induced by algebraic curves. This problem has been of interest since Odoni's first article in 1985, where he used the fixed-point proportion of iterated Galois groups to prove that the Dirichlet density of prime numbers dividing at least one term in the Sylvester sequence is zero \cite{Odoni1980}. \cref{theorem: not dynamical pullbacks and FPP} also extends the result of Fariña-Asategui and the author in \cite{FariñaRadi2026FPP}, where the fixed-point proportion of geometric iterated Galois groups of polynomials was classified.

Exceptional rational maps are dynamical pullbacks but the converse is not true. Thus, more can still be proved:

\begin{Conjecture}[{\cite[Conjecture 7]{FariñaRadi2026FPP}}]
Let $f$ be a tamely ramified rational function of degree $d \geq 2$ defined over a field $K$ such that $G_\infty(\Ksep,f,t)$ is a martingale group. Then $\FPP(G_\infty(\Ksep,f,t)) > 0$ if and only if $f$ is exceptional and not linearly conjugate to $x^{\pm d}$.
\end{Conjecture}

Apart from the case $x^{\pm d}$, the only case where the fixed-point proportion of a exceptional map is known is when $f$ is linearly conjugate to $\pm T_d$, where $T_d$ is the Chebyshev polynomial of degree $d$. This was computed by Jones in the case of $K = \CC$ \cite{Jones2012} and by Fariña-Asategui and the author in the case of arbitrary fields \cite{FariñaRadi2026FPP}. A classification of the possible values of the fixed-point proportion for every exceptional rational function is not known yet. 

On the other hand, there are no examples of the fixed-point proportion of geometric iterated Galois groups for rational maps that are dynamical pullbacks but are not exceptional. This case seems to require a more explicit description of $G_\infty(\Ksep,f,t)$. This is related to the problem of \textit{semirigidity of iterated Galois groups}, where the goal is to express $G_\infty(\Ksep,f,t)$ as the closure of a subgroup generated by elements given by an explicit wreath recursion. Semirigidity is one of the hardest problems in the area and progress has been made in specific cases such as quadratic polynomials \cite{Pink2013nPCF, Pink2013PCF}, unicritical polynomials \cite{AdamsHyde2025}, quadratic rational functions with periodic critical points \cite{Ejder2026}, a family of rational functions of degree $3$ \cite{HlushchankaLukinaWardell2025} and for other concrete quadratic rational functions \cite{EjderGolaskaKara2026, EjderKaraOzman2024}. A remarkable observation is that the proof of \cref{theorem: not dynamical pullbacks and FPP} does not require an explicit description of $G_\infty(\Ksep,f,t)$.

\cref{theorem: not dynamical pullbacks and FPP} serves directly to Application 2. This extends the results of the author in \cite{Radi2026}, where a whole classification of the values of $\Per_{\inf}(f,K)$ for any polynomial $f$ and any number field $K$ were given. 

\begin{Corollary}
Let $f$ be a rational function of degree $d \geq 2$ defined over a number field $K$ such that $G_\infty(\K-,f,t)$ is a martingale group. If $f$ is not a dynamical pullback, then 
\begin{align*}
\liminf_{N(\mathfrak{p}) \rightarrow \infty} \frac{\# \Per(f_\mathfrak{p}, \PP^1(\FF_\mathfrak{p}))}{N(\mathfrak{p})+1} = 0.
\end{align*} 
\end{Corollary}

It also serves directly to Application 3.

\begin{Corollary}
Let $f$ be a tamely ramified rational function of degree $d \geq 2$ defined over a finite field $\mathbb{F}_q$ of characteristic $p$ such that $2 \leq d < p$. If $G_\infty(\FF_q,f,t)$ is a martingale group and $f$ is not a dynamical pullback, then 
\begin{align*}
\liminf_{n \rightarrow \infty} \frac{\# \Per(f, \PP^1(\FF_{q^n}))}{q^n+1} = 0.
\end{align*} 
\end{Corollary}

\subsection*{Organization}
The article will be separated in seven sections. \cref{section: Introduction} is this introduction. In \cref{section: groups acting on trees} we state the notation and introduce the concept of virtually mixing groups. In \cref{section: iterated Galois groups}, we review the main properties of geometric iterated Galois groups. In \cref{section: energy of paths}, we develop the strategy of transportation cost in graphs. In particular, we prove in \cref{theorem: Gm is a mixing subgroup} that the delay constant only depends on the ramification portrait. In \cref{section: iterated Galois groups are virtually mixing} we prove that the iterated Galois group of rational functions are virtually mixing, which correspond to \cref{theorem: every rational is virtually mixing} and \cref{Corollary: arithmetic IGG is virtually mixing}. In \cref{section: dynamical pullbacks} we prove \cref{theorem: dynamical pullback maximal mixing subgroup} and \cref{theorem: Gm is GM} and finally in \cref{section: Fixed-point proportion of virtually mixing groups} we prove \cref{theorem: FPP virtually mixing groups} and \cref{theorem: not dynamical pullbacks and FPP}.

\section{Groups acting on trees}
\label{section: groups acting on trees}

\subsection{Basic definitions}

A \textit{$d$-regular rooted tree} is a pointed graph $(\mathscr{T}, \emptyset)$, where $\mathscr{T}$ is a tree and $\emptyset$ is a vertex of $\mathscr{T}$ such that every vertex different to $\emptyset$ has degree $d+1$ and $\emptyset$ has degree $d$. The vertex $\emptyset$ is called the \textit{root}. As it is typically done in geometric group theory, we will depict the tree with the root on top and then the rest of the vertices hanging down as a Christmas tree. Given a vertex $v$, we denote $\abs{v}$ to the distance in the tree between $v$ and $\emptyset$. For any $n \geq 0$, the \textit{$n$th level} of the tree $\mathscr{T}$ is the set 
$$\mathcal{L}_n := \set{v \in \mathscr{T}: \abs{v} = n}$$ and given $0 \leq m \leq \infty$, the truncated rooted tree is the pointed subgraph $(\mathscr{T}^m, \emptyset)$ where 
\begin{align*}
\mathscr{T}^m := \set{v \in \mathscr{T}: \abs{v} \leq m}.
\end{align*}
Note that $\mathscr{T}^\infty = \mathscr{T}$. We write $\partial \mathscr{T}$ the set of infinite simple paths of $\mathscr{T}$ starting from the root. The elements of $\partial \mathscr{T}$ we will call \textit{ends}. 

We denote $\Aut(\mathscr{T})$ the group of graph automorphisms of $\mathscr{T}$, acting on $\mathscr{T}$ on the left. Note that as the root has smaller degree than the other vertices and automorphisms of a graph must preserve adjacency, the action of $\Aut(\mathscr{T})$ on $\mathscr{T}$ can be separated in individual actions $\Aut(\mathscr{T}) \curvearrowright \mathcal{L}_n$ for every $n \geq 0$. This implies that for every $m \geq n \geq 0$ we have actions $\Aut(\mathscr{T}) \curvearrowright \mathscr{T}^n$ which induce surjective group homomorphisms 
\begin{align*}
\pi_n: \Aut(\mathscr{T}) \rightarrow \Aut(\mathscr{T}^n) \quad \text{ and } \quad \pi_{m,n}: \Aut(\mathscr{T}^m) \rightarrow \Aut(\mathscr{T}^n)
\end{align*}
obtained by restricting the action of the elements to the first $n$ levels. 

Given a vertex $v \in \mathscr{T}$, we write $\mathrm{st}(v)$ for the \textit{stabilizer of the vertex $v$} in the natural action of $\Aut(\mathscr{T})$ on $\mathscr{T}$. Given a natural number $n \geq 1$, the \textit{stabilizer of level $n$} is defined as $$\mathrm{St}(n) = \bigcap_{v \in \mathcal{L}_n} \mathrm{st}(v).$$ 

By adjacency, this implies that elements in $\St(n)$ do not act on the first $n$ levels, which implies that $\ker(\pi_n) = \St(n)$. In particular $\St(n)$ is a normal subgroup of finite index in $\Aut(\mathscr{T})$. 

Another way to construct $\Aut(\mathscr{T})$ is using an inverse limit. If $m \geq n \geq 0$, the transition maps are the maps $\pi_{m,n}$ and the projections are the maps $\pi_n$. This allows us to equip $\Aut(\mathscr{T})$ with the so-called \textit{congruence topology}, where $\set{\St(n)}_{n \geq 0}$ is a basis of neighborhoods of the identity. With this topology $\Aut(\mathscr{T})$ is a profinite group, which in particular implies that is Hausdorff and compact. 

A particular $d$-regular rooted tree that will be used is the \textit{canonical} $d$-regular rooted tree. Let $X := \set{1, \dots, d}$ and let $X^\NN$ be the free monoid on $X$. Use $\emptyset_X$ to denote the empty word in $X^\NN$. We construct the canonical $d$-regular rooted tree as follows:
\begin{enumerate}[\normalfont(i)]
\item the root is $\emptyset_X$.
\item The vertices of $T$ are the elements in $X^\NN$.
\item If $v = x_1 \dots x_n$ is a word of length $n$ and $w = y_1 \dots y_ny_{n+1}$ is a word of length $n+1$, then $v$ and $w$ are adjacent if and only if $x_i = y_i$ for all $1 \leq i \leq n$. 
\end{enumerate}
Level by level, the vertices will be organized from left to right in lexicographical order. We will reserve the letter $T$ to denote this tree.

If $v$ is a vertex of $T$, we denote $T_v$ the subtree of $T$ whose vertices are the vertices of the form $v w$ with $w \in X^\NN$. Note that $(T_v, v)$ is also a $d$-regular rooted tree isomorphic to $T$ via the shift map 
\begin{align*}
\sigma_v: T_v &\rightarrow T \\
vw & \mapsto w.
\end{align*}

If $g \in \Aut(T)$ and $v \in T$, by preservation of adjacency, we have $g: T_v \rightarrow T_{g(v)}$. Using the shift maps we define the \textit{section} of $g$ at $v$ by 
\begin{align*}
g|_v := \sigma_{g(v)} \cdot g \cdot \sigma_v^{-1} \in \Aut(T).
\end{align*}
The element $g|_v$ represents the action of $g$ over all the vertices below $v$.

It is not hard to check the following properties:
\begin{align}
(gh)|_v = g|_{h(v)} \cdot h|_v, \quad (g^{-1})|_v = (g|_{g(v)})^{-1} \quad  \text{ and } \quad g|_{vw} = (g|_v)|_w.
\label{equation: concatenation of sections}
\end{align}

Given $v \in T$, we define the map $\varphi_v: \st(v) \rightarrow \Aut(T)$ by $\varphi_v(g) = g|_v$. By \cref{equation: concatenation of sections} $\varphi_v$ is a continuous group homomorphism.

If $m \geq 1$, we define $g|_v^m := \pi_m(g|_v)$, which represents the action of $g$ on the first $m$ levels below the vertex $v$. In particular $g|_v^1 \in \Sym(X)$ is called the \textit{label} of $g$ at $v$. The collection $\set{g|_v^1}_{v \in T}$ of labels is called the \text{portrait} of $g$. 

We can give a description of $\Aut(T)$ as an iterated wreath product. Given $m \geq 1$, we define 
\begin{align*}
\psi_m: \Aut(T) & \rightarrow \Aut(T^m) \ltimes \left(  \Aut(T) \times \overset{d^m}{\ldots} \times \Aut(T) \right) \\
g & \mapsto  \pi_m(g) (g|_{v_1}, \dots, g|_{v_{d^m}}).
\end{align*}

This map is an isomorphism which allows us to represent the elements of $\Aut(T)$ by recursion. By a slightly abuse of notation we will write $g = \pi_m(g) (g|_{v_1}, \dots, g|_{v_{d^m}})$ without writing $\psi_m$ explicitly. The two cases we will often use are $m = 1$, where in this case $\pi_1(g) \in \Aut(T^1) \simeq \Sym(d)$, and if $g \in \St(n)$ then $\pi_n(g) = 1$ and we write $g = (g|_{v_1}, \dots, g|_{v_{d^n}})_n$.

If $(\mathscr{T}, \emptyset)$ is a $d$-regular rooted tree, a \textit{labeling} is an isomorphism $\phi: \mathscr{T} \rightarrow T$. Clearly, we must have $\phi(\emptyset) = \emptyset_X$. A labeling $\phi$ induces a group isomorphism $\rho_\phi: \Aut(\mathscr{T}) \rightarrow \Aut(T)$ given by 
$$(\rho_\phi(g))(v) := \phi(g(\phi^{-1}(v)))$$ 
for any $v \in T$.

Let $(\mathscr{T}, \emptyset)$ be any $d$-regular rooted tree and let $G$ be a subgroup of $\Aut(\mathscr{T})$. We define vertex stabilizers and level stabilizers in $G$ by restricting the ones of $\Aut(\mathscr{T})$, namely, $\st_G(v) := G \cap \st(v)$ and $\St_G(n) := \St(n) \cap G$ for all $v \in \mathscr{T}$ and $n \geq 1$. By inclusion, we have actions $G \curvearrowright \mathcal{L}_n$. We say that $G$ is \textit{level-transitive} if $G \curvearrowright \mathcal{L}_n$ is transitive for all $n \geq 1$. If $v$ is a vertex at level $n$, then we have a well-defined action $\St_G(n) \curvearrowright T_v$. We say that $G$ is a \textit{martingale group} if the action $\St_G(n) \curvearrowright T_v$ is level-transitive for all $n \geq 1$. Notice that this equivalent to $\St_G(n) \curvearrowright (T_v)^1$ acting transitively for all $n \geq 1$.

If $G$ is a subgroup of $\Aut(T)$, we say that $G$ is \textit{self-similar} if $g|_v \in G$ for all $v \in T$. Notice that in this case the map $\varphi_v: \st_G(v) \rightarrow G$ is well-defined for all $v \in T$. The \textit{projection of $G$ at $v$} is defined as
\begin{align*}
G_v := \varphi_v(\st_G(v)).    
\end{align*}

We say that a self-similar group $G \leq \Aut(T)$ is \textit{fractal} if $G$ is level-transitive and for all $v\in T$ we have $$G_v=G.$$ By induction, it is not hard to prove that for a group to be fractal, it is necessary to check the condition $G_v = G$ only in one vertex of $\mathcal{L}_1$.

\begin{lemma}
\label{lemma: fractal groups property}
If $G \leq \Aut(T)$ is fractal group, we have an element $g \in G$, a number $n \geq 1$ and two vertices $v,w \in \mathcal{L}_n$, then there exists $\widetilde{g} \in G$ such that $\widetilde{g}(v) = w$ and $\widetilde{g}|_v = g$.
\end{lemma}

\begin{proof}
As $G$ is level-transitive there exists $h_1 \in G$ such that $h_1(v) = w$. As $G$ is fractal there exists $h_2 \in \st_G(v)$ such that $h_2|_v = (h_1|_v)^{-1} \cdot g$. Then $\widetilde{g} = h_1 h_2$.
\end{proof}

If $G$ is a subgroup of $\Aut(\mathscr{T})$, we define self-similarity of $G$ in terms of the labeling $\phi$. Concretely, if $\phi: \mathscr{T} \rightarrow T$ is a labeling, we say that the pair $(G, \phi)$ is \textit{self-similar} if the group $\rho_\phi(G)$ is self-similar. Analogously, a pair $(G, \phi)$ is fractal if $\rho_\phi(G)$ is fractal.

If $(\mathscr{T},\emptyset)$ is a $d$-regular rooted tree and $G \leq \Aut(\mathscr{T})$ is closed in the congruence topology, then $G$ is also profinite and consequently compact. If $\mathscr{T} = T$, then $\varphi_v: G \rightarrow \Aut(T)$ is continuous.  

\begin{definition}
Let $G$ be a closed subgroup of $\Aut(T)$ and $g_1,g_2 \in G$. We say that $g_1 \sim g_2$ if $g_1$ is conjugate in $G$ to an element $g$ such that $\overline{\<g>} = \overline{\<g_2>}$, namely, they generate the same closed subgroup.
\label{definition: generate same closed subgroup}
\end{definition}

\subsection{Virtually mixing groups}

In this subsection we state some lemmas that will be used later to prove that iterated Galois groups are virtually mixing. Throughout this section, we assume that $T$ represents the canonical $d$-regular rooted tree.

We start recalling the definition of virtually mixing groups:

\begin{definition}
We say that a fractal group $G$ in $\Aut(T)$ is \textit{virtually mixing} if there exists $N \geq 0$ and 
$H \leq_f G$ such that for every $n \geq 1$ and every vertex $v \in T$ such that $\abs{v} \geq n + N$, we have 
\begin{align*}
H \leq \St_G(n)_v.  
\end{align*}
We call $N$ a \textit{delay constant} of $G$ and $H$ a \textit{mixing subgroup} of $G$.

In the case that $H$ can be taken equal to $G$, we say that $G$ is \textit{mixing}.
\label{definition: virtually mixing not intro}
\end{definition}

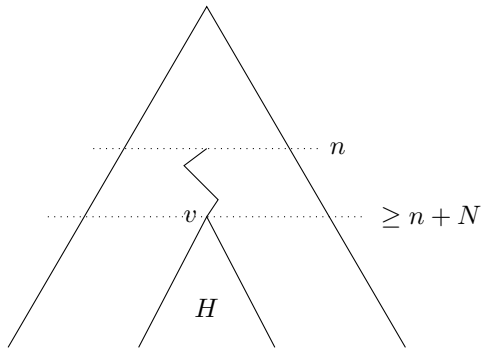
\begin{figure}[t]
\centering
\begin{tikzpicture}[scale=0.75]

\draw (-3.5,-6) -- (0,0) -- (3.5,-6);

\draw[dotted] (-2,-2.5) -- (2,-2.5);
\node[right] at (2,-2.5) {$n$};

\draw (0,-2.5) -- (-0.4,-2.8) -- (0.2,-3.4) -- (0,-3.7);

\draw[dotted] (-2.8,-3.7) -- (2.8,-3.7);
\node[right] at (2.9,-3.7) {$\geq n+N$};

\node[left] at (0,-3.7) {$v$};

\draw (-1.2,-6) -- (0,-3.7) -- (1.2,-6);
\node at (0,-5.3) {$H$};
\end{tikzpicture}
\caption{Representation of virually mixing on the tree.}
\label{figure: mixing group definition}
\end{figure}

\cref{figure: mixing group definition} depicts the definition of virtually mixing. Roughly speaking, virtually mixing means that the level of self-replication of the group allows us to see a finite-index subgroup acting below the vertex $v$. This is more restrictive than fractal as for every element $g$ in $H$, there is an element $\widetilde{g}$ whose section at $v$ is $g$, but the element $\widetilde{g}$ fixes not only $v$ but all the vertices at level $n$. This latter condition is not required in the definition of fractal.

The first observation is that \cref{definition: virtually mixing not intro} can be simplified to check only one infinitely many vertices. 

\begin{lemma}
Let $G$ be a fractal group in $\Aut(T)$. Then $G$ is virtually mixing if and only if there exists $N \geq 0$, a set of vertices $\set{v_i}_{i \geq 0} \subseteq T$ with $\abs{v_i} \xrightarrow[i \rightarrow +\infty]{} +\infty$ and $H \leq_f G$ such that $H \leq \St_G(\abs{v_i}-N)_{v_i}$.
\label{lemma: reduction definition virtually mixing}
\end{lemma}

\begin{proof}
The direct implication is obvious. For the converse, it is enough to prove the following three observations:

\begin{enumerate}[(1)]
\item For any $v \in T$ with $\abs{v} = n+k+N$ and $k \geq 0$ we have $\St_G(n+k)_v \leq \St_G(n)_v$,
\item For any $v \in T$ such that $\abs{v} = n+N$ we have $\St_G(n+1)_{xv} \leq \St_G(n)_v$,
\item For any pair of vertices $v,w$ such that $\abs{v} = \abs{w} = n+N$ we have $\St_G(n)_v = \St_G(n)_w$
\end{enumerate}

(1) follows from the fact that $\St_G(n+k) \leq \St_G(n)$. For (2), let $g \in \St_G(n+1)_{xv}$. Then, there exists $\widetilde{g} \in \St_G(n+1) \cap \st_G(xv)$ such that $\widetilde{g}|_{xv} = g$. Then $\widetilde{g}|_x \in \St_G(n) \cap \st_G(v)$ and its section at $v$ is $g$. For (3), as $G$ is fractal, by \cref{lemma: fractal groups property}, there exists $\widetilde{g} \in G$ such that $\widetilde{g}(v) = w$ and $\widetilde{g}|_v = 1$. Conjugating by $\widetilde{g}$, we obtain 
\begin{equation*}
\St_G(n)_w = (\widetilde{g}|_v) \, \St_G(n)_v \,  (\widetilde{g}|_v)^{-1} = \St_G(n)_v.
\end{equation*}

So now the proof goes as follows: let $n \geq 0$ and $v \in T$ such that $\abs{v} \geq n+N$. Write $\abs{v} = k+n+N$ with $k \geq 0$. As $\abs{v_i} \xrightarrow[i \rightarrow +\infty]{} +\infty$, there exists $i_0$ such that $\abs{v_{i_0}} \geq \abs{v}$. Applying (2) $\abs{v_{i_0}} - \abs{v}$ times, there exists $w \in \mathcal{L}_{\abs{v}}$ such that
\begin{align*}
H \leq \St_G(\abs{v_{i_0}}-N)_{v_{i_0}} \leq \St_G(\abs{v}-N)_w.
\end{align*}
By (3), 
\begin{align*}
H \leq \St_G(\abs{v}-N)_w = \St_G(\abs{v}-N)_v.
\end{align*}
Finally, as $\abs{v}-N = n+k$, by (1), 
\begin{align*}
H \leq \St_G(\abs{v}-N)_v \leq \St_G(n)_v,
\end{align*}
proving that $G$ is virtually mixing.
\end{proof}

The second observation is that we may assume that the mixing subgroup is normal and closed.

\begin{lemma}
Let $G$ be a fractal closed group in $\Aut(T)$, two numbers $N,n \geq 0$ and a vertex $v \in T$ such that $\abs{v} = n + N$. If $H \leq G$ such that $H \leq \St_G(n)_v$, then $$\overline{H}^G \leq \St_G(n)_v,$$ where $\overline{H}$ is the closure of $H$ in $G$ and $\overline{H}^G$ is the normal closure of $\overline{H}$ in $G$.
\label{lemma: virtually mixing closed normal closure}
\end{lemma}

\begin{proof}
Let $h \in \overline{H}$ and $\set{h_k}_{k \geq 1} \subseteq H$ a sequence converging to $h$. Let $\set{\widetilde{h_k}}_{k \geq 1}$ be a sequence in $\St_G(n) \cap \st_G(v)$ such that $\widetilde{h_k}|_v = h_k$ for every $k \geq 1$. As $G$ is compact and $\St_G(n) \cap \st_G(v)$ is closed, there exists
a subsequence $\set{\widetilde{h_{k_m}}}_{k \geq 1}$ converging to an element $\widetilde{h} \in \St_G(n) \cap \st_G(v)$. Then, using the continuity of $\varphi_v$ we conclude
\begin{align*}
\widetilde{h}|_v = \lim_{m \rightarrow +\infty} \widetilde{h_{k_m}}|_v = \lim_{m \rightarrow +\infty} h_{k_m} = h.
\end{align*}

Let $g \in G$ and $h \in \overline{H}$. Then there exist $\widetilde{h} \in \St_G(n) \cap \st_G(v)$ such that $\widetilde{h}|_v = h$. By fractality of $G$, there exists $\widetilde{g} \in \st_G(v)$ such that $\widetilde{g}|_v = g$. Then $\widetilde{g} \, \widetilde{h} \, \widetilde{g}^{-1} \in \St_G(n) \cap \st_G(v)$ and 
\begin{equation*}
(\widetilde{g} \, \widetilde{h} \, \widetilde{g}^{-1})|_v = ghg^{-1}. \qedhere
\end{equation*}
\end{proof}

\cref{lemma: virtually mixing closed normal closure} allows us to only worry about finding abstract subgroups of $G$ that can appear in $\St_G(n)_v$. In the context of iterated Galois groups, we will express these subgroups as generated by specific elements that will be related with the dynamics of the rational functions. To find these specific elements, we use the following commutator trick, developed in \cite[Lemma 2.3]{FariñaRadi2026FPP}:

\begin{lemma}[Commutator trick]
\label{lemma: commutator trick}
Let $G\le \mathrm{Aut}(T)$ be a fractal martingale group. Let $h$ be an element in $G$ and assume that there exists a natural number $N\ge 1$ (depending on $h$) and an element $g\in G$ satisfying that: 
\begin{enumerate}[\normalfont(i)]
\item there are vertices $u,w\in T$ with $\abs{u} \leq \abs{w} \leq N$ fixed by $g$;
\item we have $g|_u=1$ and $g|_w=h$.
\end{enumerate}
Then, for any $n\ge1$ and any $v\in T$ such that $|v|\ge n+N$ we have
$$h \in \mathrm{St}_G(n)_v.$$
\end{lemma}

\begin{remark}
Combining the last three lemmas we observe that if $H$ is a closed normal subgroup of $G$ such that $H \leq \St_G(n)_v$ (where $n$ and $v$ are as in \cref{lemma: commutator trick}), we have $h_1,h_2 \in G$ such that $h_1 \sim h_2$ (as in \cref{definition: generate same closed subgroup}) and $h_2 \in H$, then also $h_1 \in H$.

Indeed, if $h_1 \sim h_2$, then $h_1$ is conjugate to an element $h_1'$ such that $\overline{\<h_1'>} = \overline{\<h_2>}$. Since $h_2 \in H$ and $H$ is closed then $h_1' \in H$. As $H$ is normal and $h_1'$ is conjugate to $h_1$, then $h_1 \in H$.
\label{remark: equivalent elements are in mixing subgroup}
\end{remark}

Despite simple, \cref{remark: equivalent elements are in mixing subgroup} will be specially useful to deal with the lack of knowledge over the elements in iterated Galois groups.

\begin{definition}
Let $G \leq \Aut(T)$ be a self-similar group and $H$ a subgroup of $G$. We say that $H$ is \textit{project invariant} if for every vertex $v \in T$ we have $H_v \leq H$.
\end{definition}

The concept of project invariant was introduced by Adams in \cite{Adams2025}. Initially, Adams named this property self-similar, but this concept is not equivalent to self-similarity (see \cite[Remark 2.4]{FariñaRadi2026FPP}). For this reason, Fariña-Asategui and the author decided to rename it as project invariant. Project invariant subgroups in the context of iterated Galois groups are closely related to dynamical pullbacks and they play an important role in many of the main theorems in this article. 
Note that in order to check that a subgroup is project invariant, it is enough to check the condition in the vertices of the first level.

\begin{definition}
Let $G$ be a fractal group of $\Aut(T)$ and $N \geq 0$. We define the \textit{maximal mixing subgroup} of $G$ with delay constant $N$ as 
\begin{align*}
G_M := \bigcap_{n \geq 0} \bigcap_{v \in \mathcal{L}_{n+N}} \St_G(n)_v.
\end{align*}
\label{definition: maximal mixing subgroup}
\end{definition}

The subgroup $G_M$ represents the maximal amount of elements that satisfy the self-replication condition mentioned in the definition of virtually mixing groups. As defined in \cref{definition: maximal mixing subgroup}, the subgroup may not have finite index in $G$. Thus, an alternative definition for virtually mixing is that $G$ is virtually mixing if and only if there exists $N \geq 0$ such that $G_M$ has finite index.

Using the same arguments exposed in \cref{lemma: reduction definition virtually mixing}, we have that
\begin{align}
G_M = \bigcap_{v \in S} \St_G(n)_v,
\label{equation: simplification definition GM}
\end{align}
where $S$ is any subset of vertices such that their lengths go to infinity.

In the following, we prove some important properties of $G_M$. In particular we justify why it is called the maximal mixing subgroup:

\begin{proposition}
\label{proposition: properties GM}
Let $G$ be a fractal closed group of $\Aut(T)$ and $N \geq 0$. Then,

\begin{enumerate}[\normalfont(1)]
\item $G_M$ is closed and normal in $G$.
\item If $H \leq G$ satisfies that $H \leq \St_G(n)_v$ for every $v \in T$ such that $\abs{v} \geq n+N$, then $H \leq G_M$.
\item $G_M$ is project invariant.
\end{enumerate}
\end{proposition}

\begin{proof}
For (1), the subgroup $\St_G(n) \cap \st_G(v)$ is compact in $G$ and $\varphi_v$ is continuous. Therefore $\St_G(n)_v$ is closed. As intersection of closed sets is closed, then $G_M$ is closed. Similarly, the subgroup $\St_G(n) \cap \st_G(v)$ is normal in $\st_G(v)$ and $\varphi_v$ is a surjective group homomorphism since $G$ is fractal. Therefore $\St_G(n)_v$ is normal in $G$ and intersection of normal subgroups is normal.

(2) follows from \cref{lemma: reduction definition virtually mixing}.

For (3), let $x \in \mathcal{L}_1$ and $h \in (G_M)_x$. Let $g \in \st_{G_M}(x)$ such that $g|_x = h$. By definition of $G_M$, given $n \geq 2$ and $v \in \mathcal{L}_{n+N}$, there exists $\widetilde{g}_n \in \St_G(n) \cap \st_G(v)$ such that $\widetilde{g}_n|_v = g$. Write $v = v_1w$ where $\abs{v_1} = 1$. Then $\abs{wx} = n+N$, the element $\widetilde{h}_n := \widetilde{g}_n|_{v_1} \in \St_G(n-1) \cap \st_G(wx)$ and 
\begin{align*}
\widetilde{h}_n|_{wx} = \widetilde{g}_n|_{vx} = g|_x = h.
\end{align*}
Therefore $h \in G_M$ and $(G_M)_x \leq G_M$ for any $x \in \mathcal{L}_1$.
\end{proof}

\section{Iterated Galois groups}
\label{section: iterated Galois groups}

In this section we recall some results about iterated Galois groups that will be used later.

\subsection{Definition and generators}

Let $K$ be a field, $t$ transcendental over $K$ and fix $K(t)^{\textrm{sep}}$ a separable closure of $K(t)$. Note that $K(t)$ is the field of fractions of $K[t]$. Therefore $K(t)^{\textrm{sep}}$ is a ring extension of $K[t]$. Let us call $K[t]^{\textrm{sep}}$ the integral closure of $K[t]$ in $K(t)^{\textrm{sep}}$. Then $K(t)^{\textrm{sep}}$ is the fraction field of $K[t]^{\textrm{sep}}$. 

Let $f \in K(x)$ be a rational function of degree $d \geq 2$. If $z \in \PP^1(\K-)$, we denote $e_f(z)$ the \textit{local degree} of $f$ at $z$. The set of \textit{critical points} of $f$, denoted $C_f$, corresponds to the points in $\PP^1(\K-)$ whose local degree is greater than $1$. A critical point $c$ is said to be \textit{totally ramified} if $e_f(c) = d$. We say that $f$ is \textit{tamely ramified} over $K$ if the characteristic of $K$ does not divide the local degree of $f$ at any point of $\PP^1(\K-)$. 

Given $f$ tamely ramified over $K$, we construct the $d$-regular rooted tree $$T_t := \bigsqcup_{n \geq 0} f^{-n}(t)$$ as explained in \cref{section: Introduction}. As it is well-known, the pair $(T_t,t)$ is a $d$-regular rooted tree (see \cite[Lemma 2.12]{Radi2025FPP} for a proof) and we have a group homomorphism $$\rho_{\textrm{arb}}: \Gal(K(t)^{\mathrm{sep}}/K(t)) \rightarrow \Aut(T_t),$$ induced by the natural action of the absolute Galois group of $K(t)$ at each level of the tree. 

The image of $\rho_{\textrm{arb}}$, denoted $G_\infty(K,f,t)$ in this article, is called the \textit{arithmetic iterated Galois group} of $f$ over $K$. In the case that $K$ is separably closed, the group $G_\infty(K,f,t)$ is called the \textit{geometric iterated Galois group} of $f$. It is well-known that arithmetic and geometric iterated Galois groups are closed in $\Aut(T_t)$. 

The following result was known before in the literature, but we cite the recent proof of Adams due to its clever algebraic argument: 

\begin{proposition}[{\cite[Theorem 7]{Adams2025}}]
Let $K$ be a field, $t$ transcendental over $K$ and $f \in K(x)$ a tamely ramified rational function with degree $d \geq 2$. Then, there exists a labeling $\phi: T_t \rightarrow T$ such that the pair $(G_\infty(K,f,t), \phi)$ is fractal.
\label{proposition: iterated galois groups are fractal}
\end{proposition}

Taking $\phi$ a labeling as in \cref{proposition: iterated galois groups are fractal} and making a slightly abuse of notation, we will denote $G_\infty(K,f,t)$ to $\rho_\phi(G_\infty(K,f,t))$ and assume that $G_\infty(K,f,t)$ is a fractal subgroup of $\Aut(T)$.   

The set of \textit{post-critical points} of $f$ is the set
\begin{align*}
P_f := \bigcup_{n \geq 1} f^n(C_f).
\end{align*}
We say that $f$ is \textit{post-critically finite} if $P_f$ is finite.

Using the first isomorphism theorem we have another useful description of the iterated Galois groups. Indeed,
\begin{align*}
G_\infty(K,f,t) = \IIm(\rho_{\textrm{arb}}) \cong \frac{\Gal(K(t)^{\mathrm{sep}}/K(t))}{\ker(\rho_{\textrm{arb}})} \cong \Gal(K_\infty(f,t)/K(t)),
\end{align*}
where $$K_\infty(f,t) := K \left(\bigcup_{n \geq 0} f^{-n}(t) \right)$$ is the field obtained by adjoining all the preimages of $t$ under $f$. We write $$\mathcal{O}_\infty := K_\infty(f,t) \cap K[t]^{\textrm{sep}}$$ the ring of integers of $K_\infty(f,t)$.

Using Grothendieck's exposés, one can relate Galois theory and algebraic topology with étale fundamental groups and give a nice description of the generators of $G_\infty(\Ksep,f,t)$ in terms of homotopy classes of loops around $\PP^1(K)\setminus P_f$ (see \cite{Grothendieck1963expose1to5, Grothendieck1963expose6to11, Szamuely2009} for more details). Moreover, one can also understand the action of these generators on the first level and obtain some information about their sections:

\begin{proposition}[{\cite[Proposition 5.1]{FariñaRadi2026FPP}}]
Let $K$ be a field, $t$ transcendental over $K$ and $f \in K(x)$ a tamely ramified rational function with degree $d \geq 2$. Then
\begin{enumerate}[\normalfont(1)]
\label{proposition: action inertia generators}
\item $G = G_\infty(\Ksep,f,t)$ is topologically generated by $\{g_p: p \in P_f\}$, where each $g_p$ corresponds to a topological generator of an inertia subgroup $I_\mathfrak{P}$ for some prime~$\mathfrak{P}$ in $\mathcal{O}_\infty$ above $(t-p)$. Moreover, there exists an ordering of the topological generators such that
$$\prod_{p \in P_f} g_p  = 1.$$

\item Let $t_i \in f^{-1}(t)$. We have $\mathfrak{P} \cap K(t_i) = (t_i-q)$ for some $q \in f^{-1}(p)$, and the $g_p$-orbit of $t_i$ is a cycle of length $e_f(q)$.

\item If $t_i$ and $q$ are as before, then there exists $x \in \mathcal{L}_1$ such that
\begin{align*}
g_p^{e_f(q)} \in \st_G(x) \quad \text{  and  } \quad \left( g_p^{e_f(q)} \right)|_x \sim g_q.    
\end{align*}
\end{enumerate}
\end{proposition}

\begin{remark}
\label{remark: sections of generators}
Let $p \in P_f$ and $g_p$ a generator of the inertia subgroup of $I_\mathfrak{P}$ for some prime ideal $\mathfrak{P}$ in $\mathcal{O}_\infty$ dividing $(t-p)$. 
Applying \cref{proposition: action inertia generators} to $f^n$ and using the fact that $G_\infty(K,f^n,t) = G_\infty(K,f,t)$, we find a description of the sections of $g_p$ at level $n$. Indeed, let $s \in f^{-n}(t)$. Then $\mathfrak{P} \cap K(s) = (s-q)$ for some $q \in f^{-n}(p)$. 
\begin{enumerate}[\normalfont(1)]
\item If $q \notin P_f$, then the inertia subgroup of $I_\mathcal{Q}$ is trivial for any prime ideal $\mathcal{Q}$ in $\mathcal{O}_\infty$ above $(t-q)$ and therefore there exists a vertex $v \in \mathcal{L}_n$ such that
\begin{align*}
g_p^{e_{f^n}(q)} \in \st_G(v) \quad \text{  and  } \quad \left( g_p^{e_{f^n}(q)} \right)|_v = 1.    
\end{align*}

\item If $q \in P_f$, then there exists a vertex $v \in \mathcal{L}_n$ such that
\begin{align*}
g_p^{e_{f^n}(q)} \in \st_G(v) \quad \text{  and  } \quad \left( g_p^{e_{f^n}(q)} \right)|_v \sim g_q.    
\end{align*}
\end{enumerate}
\end{remark}

\subsection{Exceptional set and Thurston orbifolds}

The following well-known formula relates the degree of $f$ with the local degrees of the points in $\PP^1(\K-)$:

\begin{theorem}[{Riemann-Hurwitz formula}]
\label{Theorem: Riemann Hurwitz formula}
Let $K$ be a field, $t$ transcendental over $K$ and $f \in K(x)$ a tamely ramified rational function with degree $d \geq 2$. Then,
\begin{align*}
2d-2 = \sum_{z \in \PP^1(\K-)} (e_f(z) -1).
\end{align*}
\end{theorem}

\begin{definition}
\label{definition: Deltaf}
Following the notation in \cite{FariñaRadi2026FPP}, let us define 
\begin{align*}
\Delta_f := \set{p \in P_f: f^{-1}(p) \subseteq C_f \cup P_f}.
\end{align*}
\end{definition}

This set plays an important role in this article and will help us to distinguish the problematic post-critical points. Riemann-Hurwitz formula imposes important restriction in its size:

\begin{theorem}[{\cite[Lemma 5.4]{FariñaRadi2026FPP}}]
\label{theorem: bound Deltaf}
Let $K$ be a field, $t$ transcendental over $K$ and $f \in K(x)$ a tamely ramified rational function with degree $d \geq 2$. Then, $$\# \Delta_f \leq 4.$$
Moreover, if $f(\Delta_f) \subseteq \Delta_f$ and $\# \Delta_f = 4$ then
\begin{enumerate}[\normalfont(i)]
\item every critical point is of degree 2;
\item $\Delta_f = P_f$ and
\item $f^{-1}(\Delta_f) = C_f \sqcup P_f$.
\end{enumerate}
\end{theorem}

In the context of complex dynamics, an important quantity is the Euler characteristic of the orbifold of a rational function. This number discriminates rational maps with exceptional dynamical properties such as `non-fractal' Julia set. This notion can be extended to rational functions over arbitrary separably closed fields.

\begin{definition}
\label{definition: Thurston orbifold}
Let $K$ be a field and let $f$ be a
post-critically finite tamely ramified rational function defined over $K$. Given a point $z \in \PP^1(\K-)$, we  define the \textit{Thurston orbifold} of $f$ as the pair $(\mathbb{P}(K)^1,\nu_f)$, where the function $\nu_f:\PP^1(\K-) \to \NN \cup \set{\infty}$ is defined as
$$\nu_f(z):=\mathrm{lcm}\{\nu_f(w)e_f(w) : w\in f^{-n}(z) \text{ for some }n\ge 1\}.$$
Notice that if $z$ is in a periodic cycle which contains a critical point, then $\nu_f(z) = \infty$. This is in fact and if and only if. 

The finite tuple $$(\nu_f(z))_{z\in P_f}$$ is the \textit{signature} of the orbifold. 

The \textit{Euler characteristic} $\chi_f$ of the orbifold $(\PP^1(\K-),\nu_f)$ is calculated as
$$\chi_f := 2-\sum_{z\in P_f}\left(1-\frac{1}{\nu_f(z)}\right).$$
\end{definition}

As a consequence of Riemann-Hurwitz formula, the Euler characteristic $\chi_f$ is always non-positive (see \cite[Proposition 2.12]{BonkMeyer2017} for a proof). We say that $f$ is \textit{hyperbolic} if $\chi_f < 0$ and \textit{euclidean} if $\chi_f = 0$.

Rational functions with euclidean orbifold are the ones that are exceptional in complex dynamics and their signatures are very easy to calculate from the definition of $\chi_f$:
\begin{align*}
(2,2,2,2), \quad (2,4,4), \quad (2,3,6), \quad (3,3,3), \quad (2,2,\infty) \quad \text{and} \quad (\infty,\infty).
\end{align*}

We finally give a straightforward relation between the orbifold of $f$ and the order of the generators of $G_\infty(K,f,t)$ given in \cref{proposition: action inertia generators}.

\begin{proposition}
\label{proposition: order inertia generators}
Let $K$ be a field, $t$ transcendental over $K$ and $f \in K(x)$ a tamely ramified rational function with degree $d \geq 2$. If $p \in P_f$ and $g_p$ is one of the generators of $G_\infty(\Ksep,f,t)$ given in \cref{proposition: action inertia generators}, then the order of $g_p$ is $v_f(p)$.
\end{proposition}

\begin{proof}
As $\Aut(T_t)$ is profinite 
\begin{align*}
\abs{g_p} = \lim_{n \rightarrow +\infty} \abs{\pi_n(g_p)}.
\end{align*}
Applying \cref{proposition: action inertia generators} with $f^n$ we obtain $\abs{\pi_n(g_p)} = \lcm \set{e_{f^n}(q): q \in f^{-n}(p)}$. By multiplicity of the local degree this is 
\begin{align*}
\abs{\pi_n(g_p)} = \lcm \set{v_f(q)e_f(q): q \in f^{-n}(p)}.
\end{align*}
Taking limit in $n$, we conclude $\abs{g_p} = v_f(p)$.
\end{proof}

\section{Energy of paths}
\label{section: energy of paths}

\subsection{Strategy}
\label{section: strategy}

In this section we develop the main idea of the article to prove that every iterated Galois group is virtually mixing. If $G_\infty(\Ksep,f,t)$ is a geometric iterated Galois group, by \cref{proposition: action inertia generators} we have a set of distinguished topological generators, whose sections can be calculated recursively. Although this recursion is distorted by the equivalence $\sim$ defined in \cref{definition: generate same closed subgroup}, by \cref{remark: equivalent elements are in mixing subgroup}, this equivalence does not deprive us of proving that an specific generator is in our candidate of mixing subgroup. 

Our best tool to prove that an element can be `seen' from every $\St_G(n)_v$ is the commutator trick given in \cref{lemma: commutator trick}. This trick only requires to find an element $g \in G_\infty(K,f,t)$ such that $g$ fixes two vertices $u,w \in T$ with $\abs{u} \leq \abs{w} \leq N$, the section $g|_u = 1$ and $g|_w$ is the element of $G$ we want to see from every $\St_G(n)_v$. 

So, combining the previous two paragraphs, the idea is the following: pick $p \in P_f$. If $q = f^m(p)$, then by \cref{remark: equivalent elements are in mixing subgroup} there is a power of $g_q$ that has a section at level $m$ that is equivalent (in the sense of \cref{definition: generate same closed subgroup}) to $g_p$. We can think of this as a path that joins $q$ to $p$, where the edges of this directed graph are constructed by using the preimages of $f$. On the other hand, we may take a different `path' from $q$ to a preimage $r$ that is not post-critical. By \cref{remark: sections of generators}, a power of $g_r$ will be $1$ and therefore there is a power of $g_q$ fixing two vertices and having the desired sections to use the commutator trick.

Although this is the main idea in the section, there are some subtleties to take into consideration. First of all, it may happen that there is no such a power of an element $g_q$ that allows us to see a section $1$ and a section $g_p$ below fixed vertices. In these cases, we will prioritize the section $1$ and therefore we will not `see' $g_p$ but a power of it. This is the main reason why it is possible to prove virtually mixing and not mixing. Second, the delay constant $N$ in \cref{lemma: commutator trick} depends on the element $g_p$ we are trying to `see'. However, in the definition of virtually mixing, the delay constant must only depend on the group. Thus, one of the tasks of this section is to prove that we can find a delay constant $N$ independent of the generators $g_p$. 

To formalize the idea previously exposed, we will use the language of energy transport in graph theory \cite{MedinaTresch2025, Villani2009}. Let $\Gamma$ be a directed graph where the edges are weighted with non-negative real numbers. Suppose we have two points $x,y \in \Gamma$ and we have at least one path from $x$ to $y$. If we think of the weights as the energy of going through that edge, what is the path that minimizes the energy cost? 

Our problem has this nature somehow. By \cref{remark: sections of generators}, we need to take a power of $g_q$ to see a power of $g_p$ and $1$ as sections. The power we need is the product of the local degrees of the points in the path in between $q$ and $p$ and in the path in between $q$ and $r$. Minimizing the energy here means to take the best path that makes that power as small as possible to see the least power of $g_p$.

\subsection{Ramification portrait and energy of paths}
\label{section: ramification portrait and energy of paths}

\begin{definition}
Let $K$ be a field and fix $\K-$ an algebraic closure of $K$. Let $f \in K(x)$ be a rational function. The \textit{ramification portrait} of $f$ is the graph $\Gamma_f$ where the vertices are the points in $\PP^1(\K-)$ and two points $z_1,z_2$ are connected by an edge $z_1 \rightarrow z_2$ if $f(z_1) = z_2$. Moreover, the weight in the edge $z_1 \rightarrow z_2$ is the local degree $e_f(z_1)$.
\end{definition}

To help to describe the ramification portrait of $f$ we will draw points with different symbols based on the following rules: 

\begin{enumerate}[\normalfont(1)]
\item A black dot represents a post-critical point that is not critical.

\item A white dot represents a critical point that is not post-critical.

\item A circled black dot represents a post-critical point that is critical.

\item A dotted circle represents a point that is neither critical nor post-critical.
\end{enumerate}

Apart from the previous rules we will do some simplifications. When we know the weight in the edge is $1$, we will not write the weight in the portrait. Furthermore, only necessary points that are neither critical nor post-critical will be drawn.

\begin{figure}
\centering
\begin{tikzpicture}[ >=Stealth, every node/.style={inner sep=0pt}, arr/.style={ ->, line width=0.9pt, shorten <=5pt, shorten >=5pt }, scale=0.75,]


\coordinate (A) at (0,0);
\coordinate (B) at (2.4,0);
\coordinate (C) at (4.8,0);
\coordinate (D) at (7.2,0);
\coordinate (E) at (9.6,0);

\coordinate (F) at (4.8,-1.5);


\draw[fill=white] (A) circle (0.18);

\draw (B) circle (0.22);
\fill (B) circle (0.11);

\fill (C) circle (0.12);

\fill (D) circle (0.12);

\fill (E) circle (0.12);

\draw[dotted,line width=1pt] (F) circle (0.18);


\node[above=6pt] at (A) {$c$};
\node[above=6pt] at (B) {$p_1$};
\node[above=7pt] at (C) {$p_2$};
\node[above=6pt] at (D) {$p_3$};
\node[above=6pt] at (E) {$p$};
\node[above=7pt] at (F) {$r$};


\draw[arr] (A) -- node[midway,above=5pt] {$2$} (B);
\draw[arr] (B) -- node[midway,above=5pt] {$2$} (C);
\draw[arr] (C)--(D);
\draw[arr] (D)--(E);

\draw[arr] (F)--(D);

\draw[arr]
(E) .. controls +(1.2,0.7) and +(1.2,-0.7) .. (E);
\end{tikzpicture}
\caption{Example of how we will be drawing ramification portraits in this article.}
\label{figure: ramification portrait first example}
\end{figure}
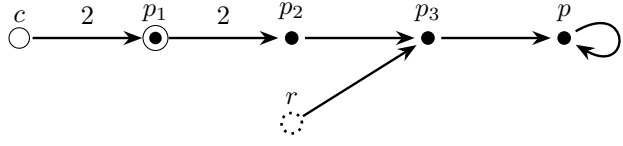

\cref{figure: ramification portrait first example} represents part of the ramification portrait of a rational function of the form $$f(x) = \frac{1+k}{x^2+k},$$ where $k$ is the solution of a polynomial with degree $9$. Here, the critical points are $0$ and $\infty$ and $1$ is the post-critical fixed point of $f$. 

As an element $g_p$ is a section of an element $g_q$ if $p$ is in the backward orbit of $q$, we need to flip the direction of the arrows in the ramification portrait in order to use the idea of energy transport explained in \cref{section: strategy}. So, although we choose to depict the ramification portrait in the direction of the application of $f$, the paths that we will be interested in go against the direction of the arrows. Formally speaking we define $\Gamma_f^{\textrm{opp}}$ the directed weighted graph obtained from $\Gamma_f$ by flipping the direction of the edges. 

\begin{definition}
A \textit{path} in $\Gamma_f^{\textrm{opp}}$ is a function $$\gamma: \set{0,\dots,L} \rightarrow \PP^1(\K-)$$ where $L \in \NN \cup \set{\infty}$ and such that $\set{\gamma(i), \gamma(i+1)}$ is an edge in $\Gamma_f^{\textrm{opp}}$ for all $i \in \set{0,\dots,L-1}$. The \textit{length} of $\gamma$ is $\ell(\gamma) := L$. The \textit{energy} of $\gamma$ is defined as
\begin{align*}
E(\gamma) := \prod_{i = 1}^{\ell(\gamma)} e_f(\gamma(i)).
\end{align*}
\end{definition}

For example in \cref{figure: ramification portrait first example}, if $\gamma$ is a path in $\Gamma_f^{\textrm{opp}}$ starting at $p$ and ending at $c$, then the length of $\gamma$ will depend on how many times we go around the loop. Concretely
$$\ell(\gamma) = \#\text{(laps around $p$)} + 4.$$
However, as $e_f(p) = 1$, the energy is $E(\gamma) = 4$ independently of the length. 

\begin{convention}
To specify a path, we will use the notation
\begin{align*}
\gamma: \gamma(0) \rightarrow \gamma(1) \rightarrow \dots \rightarrow \gamma(L).
\end{align*}
\end{convention}

\begin{definition}
Let $\gamma_1, \dots, \gamma_k$ be a finite collection of paths in $\Gamma_f^{\textrm{opp}}$. The \textit{energy} consumed by the collection is defined as
$$E(\gamma_1,\dots,\gamma_n) := \lcm(E(\gamma_1),\dots,E(\gamma_n)).$$
\end{definition}

We will be mainly interested in the case of two curves since the commutator trick requires to find an element where we have control of two of its sections. Translating the hypothesis of the commutator trick to the language of paths in $\Gamma_f^{\textrm{opp}}$, we make the following definition:

\begin{definition}
Let $K$ be a field and $f \in K(x)$ be a rational function of degree $d \geq 2$. Let $p \in P_f$. Given $\gamma_1,\gamma_2 \in \Gamma_f^{\textrm{opp}}$, we say that $(\gamma_1,\gamma_2)$ is a \textit{good pair} if the following properties are satisfied:

\begin{enumerate}[\normalfont(1)]
\item $\gamma_1(0) = \gamma_2(0)$, namely they have the same initial point. 
\item The ending point of $\gamma_1$ is $p$ and the ending point of $\gamma_2$ is not post-critical.
\item $\ell(\gamma_1) \geq \ell(\gamma_2)$.
\end{enumerate}

We will denote $\Gamma_{f,p}$ the set of good pairs of $p$.
\end{definition}

Note that in good pairs $E(\gamma_2) < \infty$ as $\ell(\gamma_2) \leq \ell(\gamma_1) < \infty$.

The set of good pairs might be empty if the post-critical point does not have any path to a non-post-critical point. This is the case for example if the post-critical point is a totally ramified fixed point, as it happens with $\infty$ when $f$ is a polynomial. In fact, it is possible to classify such a post-critical points:

\begin{lemma}
Let $f$ be a rational function of degree $d \geq 2$ defined over a field $K$. Let $p \in P_f$. Then the set $\Gamma_{f,p}$ of good pairs of $p$ is empty if and only if every point in the set $\set{f^m(p):m \geq 0}$ is totally ramified.
\label{lemma: post ritical point empty good pairs}
\end{lemma}

\begin{proof}
$(\Rightarrow)$ If $p$ has an infinite orbit, as $f$ can only have finitely many critical points, then there exists $m > 0$ such that $f^m(p)$ has a preimage $r$ outside $P_f$. Then the pair $(\gamma_1,\gamma_2)$ where
\begin{align}
\left \{ \begin{array}{r@{\ :\ }l}
\gamma_1 & f^m(p) \rightarrow f^{m-1}(p) \rightarrow \dots \rightarrow p, \\
\gamma_2 & f^m(p) \rightarrow r
\end{array}\right.
\label{equation: good pair totally ramified wandering}
\end{align}
is a good pair as $1 = \ell(\gamma_2) \leq \ell(\gamma_1) = m$. Therefore, we may assume that $p$ is preperiodic. 

Suppose first that $p$ is periodic with period $M$ and there exists $m \geq 0$ such that $f^m(p)$ is not totally ramified. Then $q = f^{m+1}(p)$ has at least two different preimages. Let $r$ be the preimage that is not in the cycle of $p$. If $r$ is not post-critical, we use the good pair in \cref{equation: good pair totally ramified wandering}. If it is post-critical, there exists $k \in \NN$ such that $r = f^k(c)$ for some critical point $c \in C_f \setminus P_f$. So, if we consider $(\gamma_1,\gamma_2)$ such that 
\begin{align*}
\left \{ \begin{array}{r@{\ :\ }l}
\gamma_1 & q = f^{m+1}(p) \rightarrow f^m(p) \rightarrow \dots \rightarrow p, \\
\gamma_2 & q = f^{m+1}(p) \rightarrow r = f^k(c) \rightarrow \dots \rightarrow c,
\end{array}\right.
\end{align*}
we have 
\begin{align*}
\ell(\gamma_2) = k+1 \quad \text{ and } \quad \ell(\gamma_1) = m+1+nM,
\end{align*}
where $n$ represents the amount of laps around the cycle. Choosing $n$ such that $\ell(\gamma_1) \geq \ell(\gamma_2)$, then $(\gamma_1,\gamma_2)$ is a good pair.

If $p$ is strictly preperiodic, then there exists $c \in C_f \setminus P_f$ and $k > 0$ such that $p = f^k(c)$. Let $m_0$ be the smallest natural number such that $q = f^{m_0}(p)$ is in the cycle of the orbit of $p$ and consider $(\gamma_1,\gamma_2)$ such that 
\begin{align*}
\left \{ \begin{array}{r@{\ :\ }l}
\gamma_1 & q = f^{m_0}(p) \rightarrow \dots \rightarrow q \rightarrow f^{m_0-1}(p) \rightarrow \dots \rightarrow p, \\
\gamma_2 & q  = f^{m_0}(p) \rightarrow \dots \rightarrow p = f^k(c) \rightarrow \dots \rightarrow c,
\end{array}\right.
\end{align*}
then 
\begin{align*}
\ell(\gamma_2) = k+m_0 \quad \text{ and } \quad \ell(\gamma_1) = m_0+nM,
\end{align*}
where again $n$ represents the amount of laps around the cycle. Choosing $n$ such that $\ell(\gamma_1) \geq \ell(\gamma_2)$, then $(\gamma_1,\gamma_2)$ is a good pair.

$(\Leftarrow)$ If every point in the orbit of $p$ is totally ramified, then we cannot construct a path $\gamma_2$.
\end{proof}

\begin{remark}
\label{remark: no good pairs clasification}
Note that by \cref{lemma: post ritical point empty good pairs} and Riemann-Hurwitz formula, the orbit of a post-critical point that does not have good pairs cannot have more than $2$ elements. Moreover, as the points in the orbit are totally ramified, then $p$ must be periodic. This leads to two possible cases:

\begin{enumerate}[\normalfont(1)]
\item $p$ is a totally ramified fixed point and therefore $f$ is linearly conjugate to a polynomial. 
\item $p$ is in a $2$-cycle and $f$ is linearly conjugate to $1/x^d$.
\end{enumerate}
\end{remark}

\begin{definition}
Let $K$ be a field and $f \in K(x)$ be a rational function of degree $d \geq 2$. Let $p \in P_f$ and $P = (\gamma_1, \gamma_2)$ a good pair for $p$. The \textit{conditional energy} of $P$ is defined as
$$\mathcal{E}(P) := \frac{E(\gamma_1,\gamma_2)}{E(\gamma_1)} = \frac{E(\gamma_2)}{\gcd(E(\gamma_1),E(\gamma_2))}.$$
\label{definition: conditional energy}
\end{definition}

Note that $\mathcal{E}(P)$ is well-defined as $E(\gamma_2) < \infty$.

\begin{lemma}
Let $f$ be a rational function of degree $d \geq 2$ defined over a field $K$ such that $G = G_\infty(K,f,t)$ is a martingale group. Let $p \in P_f$ and $P = (\gamma_1, \gamma_2)$ a good pair for $p$. Then
\begin{align*}
g_p^{\mathcal{E}(P)} \in \St_G(n)_v
\end{align*}
for all $n \geq 1$ and all $v \in T$ such that $\abs{v} \geq n+\ell(\gamma_1)$.
\label{lemma: see the element with conditional energy}
\end{lemma}

\begin{proof}
Let us set $q$ for the starting point of $\gamma_1$ (and $\gamma_2$), $r$ the ending point of $\gamma_2$, $m_1 = \ell(\gamma_1)$ and $m_2 = \ell(\gamma_2)$. Then, by \cref{remark: sections of generators} there exist a vertex $u \in \mathcal{L}_{m_2}$ and a vertex $w \in \mathcal{L}_{m_1}$ such that 
\begin{align*}
g_q^{e_{f^{m_2}}(r)} \in \st_G(u), \quad & \left(g_q^{e_{f^{m_2}}(r)}\right)|_u = 1, \\ 
g_q^{e_{f^{m_1}}(p)} \in \st_G(w) \quad \text{ and } \quad & \left(g_q^{e_{f^{m_1}}(p)}\right)|_w \sim g_p.
\end{align*}

By the chain rule we have
\begin{align*}
e_{f^{m_2}}(r) &= \prod_{i = 1}^{m_2} e_f(\gamma_2(i)) =  E(\gamma_2) \quad \text{ and } \\
e_{f^{m_1}}(p) &= \prod_{i = 1}^{m_1} e_f(\gamma_1(i)) =  E(\gamma_1).
\end{align*}

Define $$g = g_q^{E(\gamma_1,\gamma_2)}.$$ 

As $E(\gamma_i) \mid E(\gamma_1,\gamma_2)$ for $i = 1,2$ then $g$ fixes $u$ and $w$. Moreover, 
\begin{align*}
g|_u = 1 \quad \text{ and } \quad & g|_w \sim g_p^{\mathcal{E}(P)}.
\end{align*}

By the commutator trick, then $g|_w \in \St_G(n)_v$ for every $n \geq 1$ and $v \in T$ such that $\abs{v} \geq n+m_1$. Applying \cref{remark: equivalent elements are in mixing subgroup} with $H = \overline{\<g|_w>}^G$, we conclude the lemma. 
\end{proof}

\subsection{Uniformizing the delay constant}

Observe that the delay constant in \cref{lemma: see the element with conditional energy} not only depends on the post-critical point $p$ but also on the good pair. This is currently a problem as in the definition of virtually mixing the delay constant must only depend on the group. 

In this subsection, we focus on showing that we can make the delay constant to depend only on the post-critical point. 

\begin{example}
\begin{figure}
\centering
\begin{tikzpicture}[
>=Stealth,
every node/.style={inner sep=0pt},
arr/.style={
    ->,
    line width=0.9pt,
    shorten <=5pt,
    shorten >=5pt
},
scale=0.75
]


\coordinate (A) at (4.8,0);
\coordinate (B) at (7.2,0);
\coordinate (C) at (4.8,-1.5);

\coordinate (A0) at (2.3,0);
\coordinate (C0) at (2.3,-1.5);


\fill (A) circle (0.12);

\draw (B) circle (0.22);
\fill (B) circle (0.11);

\draw[fill=white] (C) circle (0.18);


\node[above=6pt] at (A) {$p$};
\node[above=6pt] at (B) {$q$};
\node[below=7pt] at (C) {$r$};


\node at (A0) {$\cdots$};


\draw[arr] ($(A0)+(0.25,0)$) -- (A);

\draw[arr] (A) -- (B);
\draw[arr] (C) -- node[midway,below=5pt] {$2^n$} (B);

\draw[arr]
(B) .. controls +(1.2,0.7) and +(1.2,-0.7)
.. node[midway,right=5pt] {$2$} (B);

\end{tikzpicture}
\caption{Ramification portrait of \cref{example: optimal gamma1}.}
\label{figure: optimal gamma1}
\end{figure}

Consider the ramification portrait in \cref{figure: optimal gamma1}. Let $\gamma_2$ be the simple path $\gamma_2: q \rightarrow r$ and let $P = (\gamma_1, \gamma_2)$ be a pair in $\Gamma_{f,p}$. Then, the starting point of $\gamma_1$ must be $q$ and the length of $\gamma_1$ is determined by how many times $\gamma_1$ goes around the loop. In particular, we have $$E(\gamma_1) = 2^{\ell(\gamma_1)-1}.$$ The conditional energy of the pair is 
\begin{align*}
\mathcal{E}(\gamma_1,\gamma_2) =  \left \{ \begin{matrix} 
2^{n+1-\ell(\gamma_1)} & \text{if $\ell(\gamma_1) \leq n+1$ }\\ 
1 & \text{if $\ell(\gamma_1) > n+1$.}
\end{matrix}\right.
\end{align*}

Note that the conditional energy reduces more the more we go around the loop, but there is an optimal path $\gamma_1$ that minimizes the conditional energy and the length of $\gamma_1$.
\label{example: optimal gamma1}
\end{example}

What happens in \cref{example: optimal gamma1} can be generalized for any pair:

\begin{lemma}
Let $f$ be a rational function of degree $d \geq 2$ defined over a field $K$. Let $p \in P_f$ such that $\Gamma_{f,p} \neq \emptyset$. Fix a path $\gamma_2$ 
and consider all the good pairs $P \in \Gamma_{f,p}$ with the path $\gamma_2$ in the second coordinate. Then, there exists a unique path $\gamma_1$ such that $\mathcal{E}(\gamma_1,\gamma_2)$ is minimal and among those that minimize $\mathcal{E}(\gamma_1,\gamma_2)$, we have that $\gamma_1$ is the shortest.
\label{lemma: optimal path}
\end{lemma}

\begin{proof}
Let $P = (\gamma_1, \gamma_2)$ and $P' = (\gamma_1', \gamma_2)$ be two good pairs in $\Gamma_{f,p}$ with the same second coordinate. As $f$ is a function and the paths $\gamma_1$ and $\gamma_1'$ start and end at the same points, one of the paths must contain the other (compare with \cref{example: optimal gamma1} going a different amount of times around the loop). This allows us to order the paths $\gamma_1$ by their length. Moreover, we have a unique path for each possible length. Without loss of generality, assume that $\gamma_1'$ is shorter than $\gamma_1$. This implies that $$E(\gamma_1') \mid E(\gamma_1)$$ and consequently $$\mathcal{E}(P) \mid \mathcal{E}(P').$$

As $E(\gamma_2)$ is finite, there is a path $\gamma_1$ that minimizes $\mathcal{E}(P)$ and is the shortest among those that minimize $\mathcal{E}(P)$. 
\end{proof}

\begin{definition}
The unique path $\gamma_1$ in \cref{lemma: optimal path} will be called the \textit{optimal path} with respect to $\gamma_2$. In the case that $\gamma_1$ is taken optimal, we call $(\gamma_1,\gamma_2)$ an optimal pair. 
\end{definition}

\begin{example}
\begin{figure}
\centering
\begin{tikzpicture}[ >=Stealth, every node/.style={inner sep=0pt}, arr/.style={ ->, line width=0.9pt, shorten <=7pt, shorten >=7pt }, scale=0.85,]


\node at (-2.5,0) (A0) {$\cdots$};

\coordinate (p)  at (0,0);
\coordinate (q1) at (2.4,0);
\coordinate (q2) at (5.4,0);

\coordinate (r1) at (0.8,-1.3);
\coordinate (r2) at (-0.8,-2.6);


\draw (p) circle (0.22);
\fill (p) circle (0.11);
\node[above=10pt] at (p) {$p$};

\draw (q1) circle (0.22);
\fill (q1) circle (0.11);
\node[above=10pt] at (q1) {$q_1$};

\draw (q2) circle (0.22);
\fill (q2) circle (0.11);
\node[above=10pt] at (q2) {$q_2$};

\draw[fill=white] (r1) circle (0.18);
\node[left=10pt] at (r1) {$r_1$};

\draw[dotted,line width=1pt] (r2) circle (0.18);
\node[left=10pt] at (r2) {$r_2$};


\draw[arr] ($(A0)+(0.25,0)$) -- (p);

\draw[arr] (p) -- node[midway,above=5pt] {$2$} (q1);

\draw[arr,bend left=25]
(q1) to node[midway,above=5pt] {$d_1$} (q2);

\draw[arr,bend left=25]
(q2) to node[midway,below=5pt] {$d_2$} (q1);

\draw[arr]
(r1) -- node[midway, below=5pt] {$4$} (q1);

\draw[arr]
(r2) -- (r1);
\end{tikzpicture}
\caption{Ramification portrait of \cref{example: minimized pair}.}
\label{figure: example minimized pair}
\end{figure}

Consider the (incomplete) ramification portrait in \cref{figure: example minimized pair}, where $2 \nmid d_1d_2$. 

Let $P = (\gamma_1,\gamma_2)$ and $P' = (\gamma_1',\gamma_2')$ be optimal good pairs of $p$ such that all the paths start at $q_1$ but the path $\gamma_2$ ends at $r_1$ and the path $\gamma_2'$ ends at $r_2$. 

Let $m_1$, $m_1'$, $m_2$ and $m_2'$ be the laps around the loop of $\gamma_1$, $\gamma_1'$, $\gamma_2$ an $\gamma_2'$ respectively. Then
\begin{align*}
\ell(\gamma_1) &= m_1+1, \quad E(\gamma_1) = 2(d_1d_2)^{m_1}, \\
\ell(\gamma_2) &= m_2+1, \quad E(\gamma_2) = 4(d_1d_2)^{m_2}, \\
\ell(\gamma_1') &= m_1'+1, \quad E(\gamma_1') = 2(d_1d_2)^{m_1'}, \\
\ell(\gamma_2') &= m_2'+2, \quad E(\gamma_2') = 4(d_1d_2)^{m_2'}.
\end{align*}
As $P$ and $P'$ are optimal pairs then $m_1 = m_2$ and $m_1' = m_2'+1$ (because as a good pair it must satisfy $\ell(\gamma_1') \geq \ell(\gamma_2')$). Therefore
\begin{align*}
\mathcal{E}(P) = 2 \quad \text{ and } \quad \mathcal{E}(P') = 2.
\end{align*}
Note that the conditional energy is the same independently of the laps around the loop. It is also independent of whether we end at $r_2$ or at $r_1$. However, the length of $\gamma_1$ and $\gamma_1'$ are affected by these unnecessary laps around the loop that $\gamma_2$ and $\gamma_2'$ do, and as it was shown in \cref{lemma: see the element with conditional energy}, the delay constant depends on the length of $\gamma_1$.
\label{example: minimized pair}
\end{example}

What happens in \cref{example: minimized pair} can be generalized for any pair:

\begin{lemma}
\label{lemma: minimized pairs}
Let $f$ be a rational function of degree $d \geq 2$ defined over a field $K$. Let $p \in P_f$ such that $\Gamma_{f,p} \neq \emptyset$. Let $P = (\gamma_1, \gamma_2)$ and $P' = (\gamma_1', \gamma_2')$ be two optimal good pairs in $\Gamma_{f,p}$ such that
\begin{enumerate}[\normalfont(1)]
\item $\gamma_2$ and $\gamma_2'$ have the same starting point and
\item $\gamma_2 \subseteq \gamma_2'$, namely, there exists $i_0 \in \NN$ such that $\gamma_2(i) = \gamma_2'(i+i_0)$ for all $i \in \Dom(\gamma_2)$.
\end{enumerate}
Then $$\mathcal{E}(P) = \mathcal{E}(P').$$
\end{lemma}

\begin{proof}
As $P$ is a good pair, the ending point $r$ of $\gamma_2$ is not in $P_f$. This means that every point in $\gamma_2'$ after $r$ is not in $P_f$. In particular their local degree is $1$ and so the energy of $\gamma_2'$ is unchanged if we omit all the points in the path after $r$.

Let $q$ be the starting point of all the paths. We have two options: 

If $q$ is not in a cycle, then $$E(\gamma_2) = E(\gamma_2')$$ and as $P$ and $P'$ are optimal pairs, then $\mathcal{E}(P) = \mathcal{E}(P')$ by \cref{lemma: optimal path}.

If $q$ is in a cycle, then every point in the cycle is post-critical. Thus, $r$ is not in the cycle. Let $q_1$ be the closest point to $r$ that is in the cycle and let us split $\gamma_2'$ as the concatenation of three different paths:
\begin{enumerate}[\normalfont(1)]
\item $\tau_1$ is the shortest path in $\Gamma_f^{\textrm{opp}}$ from $q_1$ to $q$,
\item $\tau_C$ is the shortest path in $\Gamma_f^{\textrm{opp}}$ from $q_1$ to $q_1$ containing every point in the cycle and
\item $\tau_r$ is the shortest path in $\Gamma_f^{\textrm{opp}}$ from $q_1$ to $r$.
\end{enumerate}
Then, as in \cref{example: minimized pair}, 
\begin{equation*}
\mathcal{E}(P') = \frac{E(\tau_1)E(\tau_r)}{\gcd(E(\tau_1)E(\tau_r), E(\gamma_1'))} = \frac{E(\tau_1)E(\tau_r)}{\gcd(E(\tau_1)E(\tau_r), E(\gamma_1))} = \mathcal{E}(P). \qedhere
\end{equation*}
\end{proof}

Based on \cref{lemma: minimized pairs} and the fact that we want to optimize the length of $\gamma_1$ we give the following definition:

\begin{definition}
\label{definition: minimized pair}
Let $f$ be a rational function of degree $d \geq 2$ defined over a field $K$. Let $p \in P_f$ such that $\Gamma_{f,p} \neq \emptyset$. We say that a pair $P = (\gamma_1,\gamma_2) \in \Gamma_{f,p}$ is \textit{minimized} if 
\begin{enumerate}[\normalfont(1)]
\item $\gamma_2$ has no loops,
\item the only point in the image of $\gamma_2$ that is not in $P_f$ is its ending point,
\item $\gamma_1$ is optimal with respect to $\gamma_2$.
\end{enumerate}
\end{definition}

\begin{proposition}
Let $f$ be a rational function of degree $d \geq 2$ defined over a field $K$. Let $p \in P_f$ such that $\Gamma_{f,p} \neq \emptyset$ and let $q \in \set{f^n(p): n \geq 1}$. Then, there are finitely many minimized pairs in $\Gamma_{f,p}$ starting at $q$.
\label{proposition: finitely many minimized pairs}
\end{proposition}

\begin{proof}
Let $$\mathcal{O}_f^-(q) := \set{z \in \PP^1(\K-): f^n(z) = q \text{ for some $n \geq 0$}}$$ denote the backward orbit of $q$. Note that $\mathcal{O}_f^-(q) \cap P_f$ is a finite set and thus
\begin{align*}
A := f^{-1}(\mathcal{O}_f^-(q) \cap P_f) \setminus P_f  
\end{align*}
is also finite. 

If $P = (\gamma_1,\gamma_2)$ is a minimized pair, then $P$ is determined by $\gamma_2$. This is because $\gamma_1$ is optimal with respect to $\gamma_2$ and such a $\gamma_1$ is unique by \cref{lemma: optimal path}. By (1) in \cref{definition: minimized pair}, the path $\gamma_2$ does not have loops and therefore $\gamma_2$ is uniquely determined by its ending point $r$. Finally, by (2) in \cref{definition: minimized pair} the element $r \in A$ and the set $A$ is finite.
\end{proof}

As a corollary of \cref{proposition: finitely many minimized pairs}, we conclude:

\begin{corollary}
Let $f$ be a rational function of degree $d \geq 2$ defined over a field $K$. Let $p \in P_f$ such that $\Gamma_{f,p} \neq \emptyset$ and $\set{f^n(p): n \geq 1}$ is finite. Then, there are finitely many minimized pairs in $\Gamma_{f,p}$.
\label{corollary: finitely many minimized pairs finite forward orbit}
\end{corollary}

\subsection{The subgroup $G_m$}

In this subsection, we collect all the information obtained in \cref{section: energy of paths} so far to construct a finite delay constant $N$ and a subgroup $G_m$ that satisfies that $G_m$ can be `seen' from any $\St_G(n)_v$.

\begin{definition}
Let $f$ be a rational function of degree $d \geq 2$ defined over a field $K$ and $p \in P_f$. The \textit{exponent} of $p$ is defined as
\begin{align*}
e_p :=  \left \{ \begin{matrix} 
\gcd(\mathcal{E}(P): P \in \Gamma_{f,p}) & \text{if $\Gamma_{f,p} \neq \emptyset$} \\ 
1 & \text{otherwise.}
\end{matrix}\right.
\end{align*}
\end{definition}

The exponent of $p$ is well-defined as $\mathcal{E}(P)$ is finite for every good pair. Moreover, if $p$ is a post-critical point such that $\Gamma_{f,p} \neq \emptyset$, by \cref{lemma: optimal path} and \cref{lemma: minimized pairs} we have
\begin{align}
e_p = \gcd(\mathcal{E}(P): P \in \Gamma_{f,p} \text{ is a minimized pair}).
\label{equation: ep minimized pairs}
\end{align}

The following proposition relates the energy of paths worked out in \cref{section: ramification portrait and energy of paths} with the set $\Delta_f$ defined in \cref{definition: Deltaf}: 

\begin{proposition}
\label{proposition: ep for orbit not in Deltaf}
Let $f$ be a rational function of degree $d \geq 2$ defined over a field $K$. Let $p \in P_f$ such that $\set{f^n(p): n \geq 0} \setminus \Delta_f \neq \emptyset$. Then, there exists a minimized pair $P = (\gamma_1, \gamma_2)$ such that
\begin{align*}
\mathcal{E}(P) = 1 \quad \text{ and } \quad \ell(\gamma_1) \leq 1+ \# (\set{f^n(p): n \geq 0} \cap \Delta_f) \leq 5.    
\end{align*}
\end{proposition}

\begin{proof}
Suppose first that $\set{f^n(p): n \geq 1} \setminus \Delta_f \neq \emptyset$. Let $n_0$ be the first natural number such that $q = f^{n_0}(p) \notin \Delta_f$. As $q \notin \Delta_f$, there exist $r \in f^{-1}(q)$ such that $r$ is neither critical nor post-critical. In particular $e_f(r) = 1$.  

Define $P = (\gamma_1,\gamma_2)$ as 
\begin{align*}
\left \{ \begin{array}{r@{\ :\ }l}
\gamma_1 & q = f^{n_0}(p) \rightarrow f^{n_0-1}(p) \rightarrow  \dots \rightarrow p, \\
\gamma_2 & q  = f^{n_0}(p) \rightarrow r.
\end{array}\right.
\end{align*}
It is not hard to check that $P$ is a minimized pair and since $E(\gamma_2) = 1$, then $\mathcal{E}(P) = 1$. The length of $\gamma_1$ is $n_0$, which is less or equal than $$1+ \# (\set{f^n(p): n \geq 0} \cap \Delta_f).$$

Suppose now that $\set{f^n(p): n \geq 1} \setminus \Delta_f = \emptyset$. Then by hypothesis, $p$ is the only one in the forward orbit of $p$ that is not in  $\Delta_f$. Let $r \in f^{-1}(p)$ such that $r$ is neither critical nor post-critical. In particular $e_f(r) = 1$.

As $\Delta_f$ is finite by \cref{theorem: bound Deltaf}, this implies that $p$ must be preperiodic. Thus, we can split the orbit of $p$ in the tail $\mathcal{T}$ and the cycle $\mathscr{C}$. 

Let $n_0$ be the smallest natural number such that $q = f^{n_0}(p) \in \mathscr{C}$ and let $M$ be the size of the cycle. Consider the pair $P = (\gamma_1, \gamma_2)$ where $\gamma_2$ is the straight path from $q$ to $r$ and $\gamma_1$ goes one lap around the cycle and then ends at $p$. Then
\begin{align*}
\ell(\gamma_1) = M + n_0 \quad \text{ and } \quad \ell(\gamma_2) = n_0+1.
\end{align*}
Again, it is not hard to see that $P$ is a minimized pair and as $E(\gamma_2) \mid E(\gamma_1)$ we have $$\mathcal{E}(P) = 1.$$
One more time, 
\begin{align*}
\ell(\gamma_1) = M + n_0 \leq 1+ \# (\set{f^n(p): n \geq 0} \cap \Delta_f).
\end{align*}
By \cref{theorem: bound Deltaf}, we know $\# \Delta_f \leq 4$ and consequently
\begin{equation*}
1+ \# (\set{f^n(p): n \geq 0} \cap \Delta_f) \leq 5. \qedhere
\end{equation*}
\end{proof}

We are now ready to introduce the subgroup $G_m$ mentioned in \cref{theorem: every rational is virtually mixing}:

\begin{definition}
\label{definition: Gm}
Let $f$ be a rational function of degree $d \geq 2$ defined over a field $K$. Let $t$ be transcendental over $\Ksep$ and let $G = G_\infty(\Ksep,f,t)$ be the geometric iterated Galois group of $f$. We define the subgroup $G_m$ of $G$ as
\begin{align*}
G_m := \overline{\<g_p^{e_p}: p \in P_f>}^G.
\end{align*}
\end{definition}

The main theorem of this section is the following:

\begin{theorem}
Let $f$ be a rational function of degree $d \geq 2$ defined over a field $K$ such that $G = G_\infty(K,f,t)$ is a martingale group. Then, there exists $N \geq 0$ only depending on the ramification portrait of $f$ such that $$G_m \leq \St_G(n)_v$$ for all $n \geq 1$ and all $v \in T$ such that $\abs{v} \geq n+N$.
\label{theorem: Gm is a mixing subgroup}
\end{theorem}

\begin{proof}
Let us start assuming that $p$ is a post-critical point of $f$ such that $\Gamma_{f,p} \neq \emptyset$. We split in two cases:

If $\set{f^n(p): n \geq 0} \setminus \Delta_f \neq \emptyset$, by \cref{proposition: ep for orbit not in Deltaf} we have a minimized pair $P = (\gamma_1, \gamma_2)$ such that $\mathcal{E}(P) = 1$ and $\ell(\gamma_1) \leq 5$. Then $e_p = 1$ and by \cref{lemma: see the element with conditional energy}, we conclude $$g_p^{e_p} \in \St_G(n)_v$$ for all $n \geq 1$ and $v \in T$ such that $\abs{v} \geq n+5$.

If $\set{f^n(p): n \geq 0} \subseteq \Delta_f$, then the orbit of $p$ is finite and by \cref{corollary: finitely many minimized pairs finite forward orbit} there are only finitely many minimized pairs in $\Gamma_{f,p}$. Let $P_1,\dots,P_m$ be such pairs. Then, by \cref{lemma: see the element with conditional energy}
\begin{align*}
g_p^{\mathcal{E}(P_i)} \in \St_G(n)_v
\end{align*}
for all $n \geq 1$ and $v \in T$ such that $\abs{v} \geq n+\ell(\gamma_{1i})$, where $P_i = (\gamma_{1i}, \gamma_{2i})$.

By Bezout's theorem, there exist integer numbers $a_1, \dots, a_m$ such that $$e_p = \sum_{i = 1}^m a_i \mathcal{E}(P_i)$$ and therefore
\begin{align*}
g_p^{\sum_{i = 1}^m a_i\mathcal{E}(P_i)} = g_p^{e_p} \in \St_G(n)_v
\end{align*}
for all $n \geq 1$ and $v \in T$ such that $\abs{v} \geq n+\max\set{\ell(\gamma_{1i}): i = 1,\dots,m}$.

Note that if we are in this latter case the post-critical point $p \in \Delta_f$ and hence we deal with only finitely many points $p$. Defining
\begin{align*}
N := \max\set{5, \ell(\gamma_1): (\gamma_1,\gamma_2) \text{ is a minimized pair of $p$ and $p \in \Delta_f$}}, 
\end{align*}
we conclude that
\begin{align*}
g_p^{e_p} \in \St_G(n)_v
\end{align*}
for all $n \geq 1$ and $v \in T$ such that $\abs{v} \geq n+N$.

Finally, we deal with the case where $\Gamma_{f,p} = \emptyset$. By \cref{remark: no good pairs clasification}, every element in the orbit of $p$ is totally ramified and $p$ is in a cycle of length $1$ or $2$. If the cycle has length $1$, by \cref{remark: sections of generators}, there exists a vertex $w \in \mathcal{L}_{n+N}$ such that
\begin{align*}
g_p^{d^{n+N}} \in \St_G(n+N) \quad \text{ and } \quad g_p^{d^{n+N}}|_w \sim g_p.
\end{align*}
As $G$ is fractal, conjugating by an element $g \in G$ such that $g(w) = v$ and $g|_w = 1$, we conclude
\begin{align*}
g_p = g_p^{e_p} \in \St_G(n)_v
\end{align*}
for all $n \geq 1$ and $v \in T$ such that $\abs{v} \geq n+N$.

If the cycle has length $2$, we use $g_p$ or $g_{f(p)}$ depending on the parity of $n+N$. Proceeding as in the case of cycle of length $1$, we also conclude
\begin{align*}
g_p = g_p^{e_p} \in \St_G(n)_v
\end{align*}
for all $n \geq 1$ and $v \in T$ such that $\abs{v} \geq n+N$.

Hence, by \cref{lemma: virtually mixing closed normal closure}, the result follows.
\end{proof}

\section{Iterated Galois groups are virtually mixing}
\label{section: iterated Galois groups are virtually mixing}

In this section we prove that iterated Galois groups are virtually mixing (\cref{theorem: every rational is virtually mixing} and \cref{Corollary: arithmetic IGG is virtually mixing}). We will separate the proof of \cref{theorem: every rational is virtually mixing} in two cases: when $f$ has hyperbolic orbifold and when $f$ has euclidean orbifold. 

\subsection{The hyperbolic orbifold case}
\label{section: The hyperbolic orbifold case}

In this first case, the goal will be to prove that the subgroup $G_m$ constructed in \cref{definition: Gm} has finite index in $G = G_\infty(\Ksep,f,t)$. 

By \cref{proposition: ep for orbit not in Deltaf}, we have
that $G/G_m$ is a subgroup of the group generated by the cosets $g_p \, G_m$, where $p$ satisfies that $\set{f^n(p): n \geq 0} \subseteq \Delta_f$. Since $\# \Delta_f \leq 4$, the quotient is generated by at most four elements.

If 
\begin{align*}
\# \set{p \in P_f: \set{f^n(p): n \geq 0} \subseteq \Delta_f} = 4,
\end{align*}
by \cref{theorem: bound Deltaf} $f$ has signature $(2,2,2,2)$ and thus $f$ has euclidean orbifold. Therefore we may assume that 
\begin{align*}
\# \set{p \in P_f: \set{f^n(p): n \geq 0} \subseteq \Delta_f} \leq 3.
\end{align*}

\begin{proposition}
Let $f$ be a rational function of degree $d \geq 2$ defined over a field $K$ such that $G = G_\infty(K,f,t)$ is a martingale group. If 
\begin{align*}
\# \set{p \in P_f: \set{f^n(p): n \geq 0} \subseteq \Delta_f} \leq 1,
\end{align*}
then $G_m = G$.
\label{proposition: Gm missing at most 1}
\end{proposition}

\begin{proof}
If $$\# \set{p \in P_f: \set{f^n(p): n \geq 0} \subseteq \Delta_f} = 0,$$ the result follows directly from \cref{proposition: ep for orbit not in Deltaf}. If it has exactly one (say $p$), by \cref{proposition: action inertia generators} we can write $p$ in terms of the other generators and therefore $g_p \in G_m$, giving the result.
\end{proof}

\begin{proposition}
Let $f$ be a rational function of degree $d \geq 2$ defined over a field $K$ such that $G = G_\infty(K,f,t)$ is a martingale group. If 
\begin{align*}
\# \set{p \in P_f: \set{f^n(p): n \geq 0} \subseteq \Delta_f}  = 2,
\end{align*}
then $G/G_m$ is a finite cyclic group.
\label{proposition: Gm missing 2}
\end{proposition}

\begin{proof}
Let $p_1$ and $p_2$ be two post-critical points satisfying that its forward orbit is contained in $\Delta_f$. Then, by the definition of $G_m$ and \cref{proposition: action inertia generators} we have
\begin{align*}
G/G_m \leq \<g_{p_1}, g_{p_2} \; | \; g_{p_1}^{e_{p_1}},  g_{p_2}^{e_{p_2}}, g_{p_1}g_{p_2}>,
\end{align*}
where on the right-hand side we are giving the group presentation. Such a presentation is equivalent to
\begin{align*}
\<g_{p_1} \;| \; g_{p_1}^{\gcd(e_{p_1},e_{p_2})}>,
\end{align*}
which corresponds to the presentation of a cyclic group with finite order.
\end{proof}

\subsection{The hyperbolic orbifold case and Von Dyck groups}
\label{section: The hyperbolic orbifold case and Von Dyck groups}

To finish with the hyperbolic case we assume
\begin{align*}
\# \set{p \in P_f: \set{f^n(p): n \geq 0} \subseteq \Delta_f}  = 3.
\end{align*}

If $p_1,p_2,p_3$ are the post-critical points in this set, then 
\begin{align*}
G/G_m \leq \<g_{p_1},g_{p_2},g_{p_3} \; | g_{p_1}^{e_{p_1}}, g_{p_2}^{e_{p_2}}, g_{p_3}^{e_{p_3}}, g_{p_1}g_{p_2}g_{p_3}>.
\end{align*}

The group on the right-hand side is known as a Von Dyck group and its size is classified in terms of the order of the generators:

\begin{theorem}[{\cite[Theorem I.6.2]{MalleMatzat1999}}]
The Von Dyck group
\begin{align*}
\Delta(e_{p_1},e_{p_2},e_{p_3}) := \<g_{p_1},g_{p_2},g_{p_3} \; | g_{p_1}^{e_{p_1}}, g_{p_2}^{e_{p_2}}, g_{p_3}^{e_{p_3}}, g_{p_1}g_{p_2}g_{p_3}>
\end{align*}
is finite if and only 
\begin{align*}
\frac{1}{e_{p_1}} + \frac{1}{e_{p_2}} + \frac{1}{e_{p_3}} > 1
\end{align*}
and moreover when this happens there are finitely many possibilities corresponding to the following groups:

\begin{itemize}
\item $\Delta(1,n,n) \cong C_n$, the cyclic group of order $n$.

\item $\Delta(2,2,n) \cong D_n$, the dihedral group of order $2n$.

\item $\Delta(2,2,3) \cong A_4$, the alternating group over four elements.

\item $\Delta(2,3,4) \cong \Sym(4)$, the symmetric group over four elements.

\item $\Delta(2,3,5) \cong A_5$ the alternating group over five elements.
\end{itemize}
\label{theorem: classification Von Dyck groups}
\end{theorem}

Therefore, our goal is to prove that 
\begin{align*}
\frac{1}{e_{p_1}} + \frac{1}{e_{p_2}} + \frac{1}{e_{p_3}} > 1,
\end{align*}
for the three post-critical points whose forward orbit is contained in $\Delta_f$. 

Let $p_i$ be one of these critical points. We can split its preimage as 
\begin{align}
f^{-1}(p_i) = \set{q_{ij}: j = 1,\dots,m_i} \sqcup (f^{-1}(p_i) \cap \set{p_1,p_2,p_3}).
\label{equation: qij}
\end{align}

If $m_i > 0$, define $$d_i := \gcd(e_f(q_{ij}): j = 1,\dots,m_i).$$

By convention and to avoid discussing several cases, we set $d_i = \infty$ when $m_i = 0$. 

\begin{lemma}
\label{lemma: epi divides di}
Suppose $m_i > 0$. With the notation set before, we have
$$e_{p_i} \mid d_i.$$
\end{lemma}

\begin{proof}
We start observing that every $q_{ij} \notin \Delta_f$. Indeed, if $q_{ij} \in \Delta_f$, then $$\set{f^n(q_{ij}): n \geq 0}$$ would be contained in $\Delta_f$ and therefore $$q_{ij} \in f^{-1}(p_i) \cap \set{p_1,p_2,p_3},$$ giving a contradiction with the definition of the elements $q_{ij}$. 

As $q_{ij} \notin \Delta_f$, there exists $r \in f^{-1}(q_{ij})$ that is neither critical nor post-critical. Since the forward orbit of $p_i$ is contained in $\Delta_f$, we know $p_i$ is preperiodic. Let $n_0 \geq 0$ be the smallest natural number such that 
$q = f^{n_0}(p)$ is in the cycle of $p$. Consider $\gamma_2$ the path in $\Gamma_f^{\textrm{opp}}$ that goes directly from $q$ to $r$ and $\gamma_1$ the path in $\Gamma_f^{\textrm{opp}}$ that goes from $q$ to $p$,  going as many times as necessary around the cycle to satisfy that $\ell(\gamma_1) \geq \ell(\gamma_2)$. Then, $P_{ij}:= (\gamma_1,\gamma_2)$ is a good pair and as all the vertices met by $\gamma_2$ before $q_{ij}$ are also part of the path of $\gamma_1$, we have 
\begin{align*}
\mathcal{E}(P_{ij}) \mid e_f(q_{ij}) \, e_f(r) = e_f(q_{ij}).
\end{align*}
Doing this for every $j$ we conclude that
\begin{equation*}
e_p \mid \gcd(\mathcal{E}(P_{ij}): j = 1,\dots,m_i) \mid \gcd(e_f(q_{ij}): j = 1,\dots,m_i) = d_i. \qedhere
\end{equation*}
\end{proof}

\begin{lemma}
\label{lemma: qijs are at least d-1}
Let $f$ be a tamely ramified rational function of degree $d \geq 2$ defined over a field $K$. Suppose that
\begin{align*}
\# \set{p \in P_f: \set{f^n(p): n \geq 0} \subseteq \Delta_f}  = 3.
\end{align*}
With the notation set before, we have
$$m_1 + m_2 + m_3 \geq d-1$$ and the equality holds if and only if 
\begin{align*}
C_f \subseteq \bigcup_{i = 1}^{3} f^{-1}(p_i).
\end{align*}
\end{lemma}

\begin{proof}
Using \cref{equation: qij} we obtain
\begin{align*}
d = \sum_{j = 1}^{m_i} e_f(q_{ij}) + \sum_{p \in f^{-1}(p_i) \cap \set{p_1,p_2,p_3}} e_f(p), 
\end{align*}
which implies 
\begin{multline*}
d -m_i - \#(f^{-1}(p_i) \cap \set{p_1,p_2,p_3}) \\ = \sum_{j = 1}^{m_i} (e_f(q_{ij})-1) + \sum_{p \in f^{-1}(p_i) \cap \set{p_1,p_2,p_3}} (e_f(p)-1). 
\end{multline*}

Adding for each $i = 1,2,3$ we obtain
\begin{align*}
3d-\sum_{i = 1}^3 m_i -3 &= \sum_{i = 1}^3 \sum_{j = 1}^{m_i} (e_f(q_{ij})-1) + \sum_{i = 1}^3 (e_f(p_i)-1) \\
&\leq \sum_{p \in \PP^1(\K-)} (e_f(p)-1) \\
&= 2d-2,
\end{align*}
where in the last equality we use Riemann-Hurwitz formula.

Solving, we conclude
\begin{align*}
m_1 + m_2 + m_3 \geq d-1.
\end{align*}

From the calculation, it is not hard to see that the equality holds if and only if all the critical points of $f$ are in
\begin{equation*}
\bigcup_{i = 0}^{3} f^{-1}(p_i). \qedhere
\end{equation*}
\end{proof}

To ease the notation more let us define the following constants:
\begin{align*}
E_i &:= \sum_{f^{-1}(p_i) \cap \set{p_1,p_2,p_3}} e_f(p), \\
f_i &:= \#(f^{-1}(p_i) \cap \set{p_1,p_2,p_3}), \\
A_i &:= \frac{1}{d_i} \left(\sum_{j = 1}^{m_i} e_f(q_{ij}) \right).
\end{align*}
If $m_i = 0$, we set $A_i = 0$ and convene that $A_i d_i = 0$.

From \cref{equation: qij} we deduce
\begin{align}
d = E_i + d_i \, A_i.
\label{equation: local degrees each pi}
\end{align}

Note also that $e_f(q_{ij}) \geq d_i$ for each $j = 1, \dots, m_i$, giving the inequality
\begin{align}
A_i \geq m_i.
\label{equation: Ai more mi}
\end{align}

\begin{proposition}
Let $f$ be a tamely ramified rational function of degree $d \geq 2$ defined over a field $K$. Suppose that
\begin{align*}
\# \set{p \in P_f: \set{f^n(p): n \geq 0} \subseteq \Delta_f}  = 3,
\end{align*}
and at least one of the following inequalities happen:
\begin{enumerate}[\normalfont(1)]
\item $m_1+m_2+m_3 > d-1$,
\item $A_i > m_i$ for some $i = 1,2,3$,
\item $\frac{1}{d_1} + \frac{1}{d_2} + \frac{1}{d_3} > 1$.
\end{enumerate}
Then $G/G_m$ is finite.
\label{proposition: assume sum is d-1}
\end{proposition}

\begin{proof}
If the third inequality happens, by \cref{lemma: epi divides di} we have
\begin{align*}
\frac{1}{e_{p_1}} + \frac{1}{e_{p_2}} + \frac{1}{e_{p_3}} \geq \frac{1}{d_1} + \frac{1}{d_2} + \frac{1}{d_3} > 1.
\end{align*}

As the $m_i$'s and the $A_i$'s are integer numbers, if one of the first two inequalities hold we have $$d \leq A_1 + A_2 + A_3.$$
If one of them (say $m_1$) equals zero, then $A_1 = 0$ and 
\begin{align*}
d \leq A_2 + A_3 = \frac{d-E_2}{d_2} + \frac{d-E_3}{d_3} \leq \frac{d}{d_2} + \frac{d}{d_3}.  
\end{align*}
Using \cref{lemma: epi divides di} we conclude
\begin{align*}
\frac{1}{e_{p_1}} + \frac{1}{e_{p_2}} + \frac{1}{e_{p_3}} > \frac{1}{d_2} + \frac{1}{d_3} \geq 1.
\end{align*}
In the general case,
\begin{align*}
d \leq A_1 + A_2 + A_3 = \frac{d-E_1}{d_1} + \frac{d-E_2}{d_2} + \frac{d-E_3}{d_3}.  
\end{align*}
Since $E_1 + E_2 + E_3 \geq 3$, we have
\begin{align*}
d < \frac{d}{d_1} + \frac{d}{d_2} + \frac{d}{d_3}  
\end{align*}
and again using \cref{lemma: epi divides di},
\begin{equation*}
\frac{1}{e_{p_1}} + \frac{1}{e_{p_2}} + \frac{1}{e_{p_3}} \geq \frac{1}{d_1} + \frac{1}{d_2} + \frac{1}{d_3} > 1. \qedhere
\end{equation*}
\end{proof}

Therefore, from the rest of this subsection, we may assume 
\begin{align}
d-1 = m_1+m_2+m_3, \,\,  \, m_i = A_i \text{ for all $i$} \,\, \text{ and } \,\, \frac{1}{d_1} + \frac{1}{d_2} + \frac{1}{d_3} \leq 1.
\label{equation: setting for cases}
\end{align}
This, combined with \cref{lemma: qijs are at least d-1},  implies that every $q_{ij}$ is a critical point and $e_f(q_{ij}) = d_i$ for every $j = 1, \dots, m_i$. Moreover, all the critical points of $f$ are preimages of $f^{-1}(p_i)$ for some $i$.

The following lemma will be used many times in the following subsections:

\begin{lemma}
\label{lemma: bound with d0}
Let $f$ be a tamely ramified rational function of degree $d \geq d_0$ defined over a field $K$. Suppose that
\begin{align*}
\# \set{p \in P_f: \set{f^n(p): n \geq 0} \subseteq \Delta_f}  = 3
\end{align*}
and $f$ satisfies the conditions of \cref{equation: setting for cases}. Keeping the convention that $d_1,d_2,d_3$ may be $\infty$, we have
\begin{align*}
d_0-1 \leq \frac{d_0-f_1}{d_1} + \frac{d_0-f_2}{d_2} + \frac{d_0-f_3}{d_3} \quad \text{ and } \quad 1 \geq \frac{E_1}{d_1} + \frac{E_2}{d_2} + \frac{E_3}{d_3}.
\end{align*}
\end{lemma}

\begin{proof}
Using \cref{equation: local degrees each pi} and \cref{equation: setting for cases} we have
\begin{align*}
d-1 = \frac{d-E_1}{d_1} + \frac{d-E_2}{d_2} + \frac{d-E_3}{d_3}.
\end{align*}
Rearranging, 
\begin{align*}
d \left(1- \frac{1}{d_1} - \frac{1}{d_2} - \frac{1}{d_3}    \right) = 1-\frac{E_1}{d_1} - \frac{E_2}{d_2} - \frac{E_3}{d_3}.
\end{align*}
By \cref{equation: setting for cases} the left-hand side is positive, so the right-hand side too. 

Finally, since $E_i \geq f_i$ for all $i = 1,2,3$ we conclude
\begin{align*}
d_0 \left(1- \frac{1}{d_1} - \frac{1}{d_2} - \frac{1}{d_3}    \right) &\leq 
d \left(1- \frac{1}{d_1} - \frac{1}{d_2} - \frac{1}{d_3}    \right) \\ 
&= 1-\frac{E_1}{d_1} - \frac{E_2}{d_2} - \frac{E_3}{d_3} \\
& \leq 1-\frac{f_1}{d_1} - \frac{f_2}{d_2} - \frac{f_3}{d_3}. \qedhere
\end{align*}
\end{proof}

Since we are dealing with the case of three points, we have the obvious equality $f_1 + f_2 + f_3 = 3$. Although it is tedious, we will not escape from a case by case analysis to solve the whole problem. By symmetry, we may assume $f_1 \geq f_2 \geq f_3.$ Then, we have three cases to analyze:
\begin{align*}
(f_1,f_2,f_3) \in \set{(1,1,1), (3,0,0), (2,1,0)}.
\end{align*}

The following tools will be used in the cases where $f_3 = 0$. 

\begin{lemma}
\label{lemma: bound for E2}
Let $f$ be a tamely ramified rational function of degree $d \geq 2$ defined over a field $K$ satisfying \cref{equation: setting for cases}, with $f_3 = 0$ and $E_2 < d_2 < \infty$. If $E_1 \geq d_3$ then 
\begin{align*}
E_2 \leq 2d_2 + d_3 -d_2d_3.
\end{align*}
Moreover, if $E_2 = 2d_2 + d_3 -d_2d_3$ then $E_1 = d_3$.
\end{lemma}

\begin{proof}
By hypothesis $f_3 = 0$, so $d = d_3 m_3$. Thus, by \cref{equation: setting for cases} and \cref{equation: local degrees each pi}
\begin{align*}
d_2(d_3m_3-1) = d_2(m_1 + m_2 + m_3) = d_2m_1 + d-E_2 + d_2m_3
\end{align*}
which implies
\begin{align*}
m_1d_2+d_2-E_2 = m_3(d_2d_3-d_3-d_2).
\end{align*}

Since in this case $E_3 = 0$, using \cref{lemma: bound with d0}, we have 
\begin{align*}
1 \geq \frac{E_1}{d_1} + \frac{E_2}{d_2}.
\end{align*}
Since by hypothesis $d_2 > E_2$, solving $d_1$ we obtain
\begin{align*}
d_1 \geq \frac{d_2}{d_2-E_2} E_1.
\end{align*}

Then, using that $E_1 \geq d_3$, 
\begin{align}
\label{equation: bounds for E1m3}
E_1m_3 \geq d_3m_3 = d = E_1 + d_1m_1 \geq E_1 \left(1 + \frac{d_2}{d_2-E_2}m_1 \right) = \\
\frac{E_1}{d_2-E_2}(d_2-E_2+d_2m_1) = \frac{E_1 m_3}{d_2-E_2}(d_2d_3-d_2-d_3) \notag
\end{align}
and consequently 
\begin{align*}
1 \geq \frac{d_2d_3-d_2-d_3}{d_2-E_2},
\end{align*}
which is equivalent to 
\begin{align*}
E_2 \leq 2d_2+d_3-d_2d_3.    
\end{align*}

In the case we have equality in the previous formula, all the inequalities in \cref{equation: bounds for E1m3} become equalities and therefore $E_1 = d_3$. 
\end{proof}

\begin{lemma}
\label{lemma: if E2 is d2}
Let $f$ be a tamely ramified rational function of degree $d \geq 2$ defined over a field $K$ satisfying \cref{equation: setting for cases}, with $f_3 = 0$. If $E_2 = d_2$ then 
\begin{align*}
d_1 = \infty \quad \text{ and } \quad d_2d_3 = d_2+d_3.
\end{align*}
\end{lemma}

\begin{proof}
By \cref{lemma: bound with d0} we have
\begin{align*}
1 \geq \frac{E_1}{d_1} + \frac{E_2}{d_2}.
\end{align*}
Therefore, if $E_2 = d_2$ then $d_1 = \infty$ and $m_1 = 0$. Also, by \cref{equation: local degrees each pi}
\begin{align*}
d = E_2 + d_2m_2 = d_2(1+m_2).
\end{align*}

As $f_3 = 0$ we have $d = d_3m_3$ so by \cref{equation: setting for cases}
\begin{align*}
d_2d_3m_3 = d_2d = d_2(1 + m_2 + m_3) = d_2\left(\frac{d_3m_3}{d_2} + m_3 \right)
\end{align*}
which gives
\begin{equation*}
d_2d_3 = d_2 + d_3. \qedhere
\end{equation*}

\end{proof}

\subsubsection{Case $(f_1,f_2,f_3) = (1,1,1)$}

\begin{figure}
\centering
\fbox{%
\begin{minipage}[c][3.5cm][c]{0.3\textwidth}
\centering
\begin{tikzpicture}[
>=Stealth,
every node/.style={inner sep=0pt},
arr/.style={
->,
line width=0.9pt,
shorten <=7pt,
shorten >=7pt},
scale=0.65]

\coordinate (p1) at (0,1.5);
\coordinate (p2) at (0,-1.5);
\coordinate (p3) at (2,0);

\fill (p1) circle (0.12);
\fill (p2) circle (0.12);
\fill (p3) circle (0.12);

\draw[arr]
(p1) .. controls +(1.2,0.7) and +(1.2,-0.7) ..
node[midway,right=5pt] {$e_f(p_1)$} (p1);

\draw[arr]
(p2) .. controls +(1.2,0.7) and +(1.2,-0.7) ..
node[midway,right=5pt] {$e_f(p_2)$} (p2);

\draw[arr]
(p3) .. controls +(1.2,0.7) and +(1.2,-0.7) ..
node[midway,right=5pt] {$e_f(p_3)$} (p3);
\end{tikzpicture}
\end{minipage}}
\fbox{%
\begin{minipage}[c][3.5cm][c]{0.32\textwidth}
\centering
\begin{tikzpicture}[
>=Stealth,
every node/.style={inner sep=0pt},
arr/.style={
->,
line width=0.9pt,
shorten <=7pt,
shorten >=7pt},
scale=0.65]

\coordinate (p1) at (0,1.5);
\coordinate (p2) at (0,-1.5);
\coordinate (p3) at (2.5,0);

\fill (p1) circle (0.12);
\fill (p2) circle (0.12);
\fill (p3) circle (0.12);

\draw[arr,bend left=25]
(p1) to node[left=20pt] {$e_f(p_1)$} (p2);

\draw[arr,bend left=25]
(p2) to node[right=20pt] {$e_f(p_2)$} (p1);

\draw[arr]
(p3) .. controls +(1.2,0.7) and +(1.2,-0.7) ..
node[midway,above=10pt] {$e_f(p_3)$} (p3);

\end{tikzpicture}
\end{minipage}}
\fbox{%
\begin{minipage}[c][3.5cm][c]{0.3\textwidth}
\centering
\begin{tikzpicture}[
>=Stealth,
every node/.style={inner sep=0pt},
arr/.style={
->,
line width=0.9pt,
shorten <=7pt,
shorten >=7pt},
scale=0.65]

\coordinate (p1) at (0,1.5);
\coordinate (p2) at (0,-1.5);
\coordinate (p3) at (2.5,0);

\fill (p1) circle (0.12);
\fill (p2) circle (0.12);
\fill (p3) circle (0.12);

\draw[arr]
(p1) to node[left=7pt] {$e_f(p_1)$} (p2);

\draw[arr]
(p2) to node[below=7pt] {$e_f(p_2)$} (p3);

\draw[arr]
(p3) to node[above=7pt] {$e_f(p_3)$} (p1);

\end{tikzpicture}
\end{minipage}}

\caption{Possible ramification portraits for $(f_1,f_2,f_3) = (1,1,1)$.}
\label{figure: Ramification portraits case 1 1 1}
\end{figure}

\cref{figure: Ramification portraits case 1 1 1} shows the possible diagrams that fit in this case. Using \cref{lemma: bound with d0} with $d = d_0$, we have
\begin{align*}
d-1 \leq  \frac{d-1}{d_1} + \frac{d-1}{d_2} + \frac{d-1}{d_3},
\end{align*}
which implies that
\begin{align*}
1 \leq \frac{1}{d_1} + \frac{1}{d_2} + \frac{1}{d_3}.
\end{align*}
By \cref{equation: setting for cases}, this means that 
\begin{align*}
1 = \frac{1}{d_1} + \frac{1}{d_2} + \frac{1}{d_3}
\end{align*}
and the possible solutions are
\begin{align*}
(2,3,6) \quad (2,4,4) \quad (3,3,3) \quad (2,2,\infty).
\end{align*}

Using again \cref{lemma: bound with d0}, 
\begin{align*}
1 \geq \frac{E_1}{d_1} + \frac{E_2}{d_2} + \frac{E_3}{d_3} \geq \frac{1}{d_1} + \frac{1}{d_2} + \frac{1}{d_3} = 1.
\end{align*}
If $d_1,d_2,d_3 < \infty$ then this implies that $e_f(p_i) = 1$ for all $i = 1,2,3$ and consequently $e_{p_i} = d_i$ and $f$ has an euclidean orbifold. Moreover $v_f(p_i) = e_{p_i}$ which implies that $G_m = 1$ by \cref{proposition: order inertia generators}.

If one of the $d_i$'s is $\infty$, then the other two are $2$ and by \cref{lemma: epi divides di} we obtain that $e_{p_k} \mid 2$ for $k \neq i$, which implies that $G/G_m$ is finite. 

\subsubsection{Case $(f_1,f_2,f_3) = (3,0,0)$}

In this case there is only one possible diagram since $f(p_i) = p_1$ for $i = 1,2,3$. In particular $m_2,m_3 > 0$ and $d_2,d_3 \mid d$. By symmetry in the role of $p_2$ and $p_3$, we may assume that $d_2 \leq d_3$.

Using the good pairs $P = (\gamma_1,\gamma_2)$ defined as
\begin{align*}
\left \{ \begin{array}{r@{\ :\ }l}
\gamma_1 & p_1 \rightarrow p_1 \rightarrow p_k, \\
\gamma_2 & p_1 \rightarrow p_i \rightarrow q_{ij}, 
\end{array}\right.
\end{align*}
for $i,k = 1,2,3$, we obtain
\begin{align}
e_{p_1} &\mid \gcd(d_1,e_f(p_2) d_2, e_f(p_3) d_3), \label{equation: ep1 3 0 0} \\ 
e_{p_2} &\mid \gcd(d_1,d_2, e_f(p_3) d_3), \label{equation: ep2 3 0 0} \\ 
e_{p_3} &\mid \gcd(d_1,e_f(p_2) d_2, d_3). \label{equation: ep3 3 0 0}
\end{align}

In particular $G/G_m$ is finite if $\gcd(d_1,d_3)$ or $\gcd(d_1,d_2)$ is $1$. 

The following proposition will isolate the tuples $(d_1,d_2,d_3)$ that will require special attention:

\begin{proposition}
Let $f$ be a tamely ramified rational function of degree $d \geq 2$ defined over a field $K$ satisfying \cref{equation: setting for cases} and with $(f_1,f_2,f_3) = (3,0,0)$. If $(d_1,d_2,d_3)$ is none of the following tuples:
\begin{align*}
(\infty,3,3), \,\, (\infty,2,d_3) \text{ with $2 \leq d_3 \leq 6$}, \,\, (3,3,3), \\
(6k,2,3) \,\, k \geq 1, \,\, (4,2,4) \,\, \text{ and } \,\, (3,2,6),
\end{align*}
then $G/G_m$ is finite.
\label{proposition: excluded cases 3 0 0}
\end{proposition}

\begin{proof}
Recall we may assume \cref{equation: setting for cases} as otherwise $G/G_m$ is finite. 

Since $E_1 \geq 3$, then $d \geq 3$, so applying \cref{lemma: bound with d0} with $d_0 = 3$ we obtain
\begin{align*}
\frac{2}{3} \leq \frac{1}{d_2} + \frac{1}{d_3}
\end{align*}
whose solutions are $(3,3)$ and $(2,d_3)$ with $d_3 \leq 6$. If $d_1 = \infty$, we isolate these cases.

If $d_1 < \infty$, then by \cref{equation: local degrees each pi} we have $d \geq 3 + d_1$ and since $d_1 \geq 2$ we have $d \geq 5$. 

Consider $(d_2,d_3) = (3,3)$. Then $3 \mid d$ and therefore $d \geq 6$. Applying \cref{lemma: bound with d0} with $d_0 = 6$, we obtain $d_1 \leq 3$. On the other hand, by \cref{equation: setting for cases}
\begin{align*}
1 \geq \frac{1}{d_1} + \frac{1}{d_2} + \frac{1}{d_3}, 
\end{align*}
so applied to this case $d_1 \geq 3$ and therefore $d_1 = 3$.

Consider now $d_2 = 2$ and $2 \leq d_3 \leq 6$. From \cref{equation: setting for cases}, we find a lower bound for $d_1$ in terms of $d_3$:
\begin{align}
d_1 \geq \frac{2 d_3}{d_3-2}.
\label{equation: lower bound d1 3 0 0}
\end{align}
To find an upper bound we use \cref{lemma: bound with d0}.

If $d_3 = 2$, as $d_1 < \infty$ then 
\begin{align*}
\frac{1}{d_1} + \frac{1}{d_2} + \frac{1}{d_3} > 1
\end{align*}
contradicting \cref{equation: setting for cases}.

If $d_3 = 3$, then $6 \mid d$ and as $d \geq 5$ then $d \geq d_0 = 6$. Notice that $G/G_m$ is finite if $\gcd(d_1,d_2) = 1$ or $\gcd(d_1,d_3) = 1$, so $6 \mid d_1$.

If $d_3 = 4$, then $4 \mid d$ and as $d \geq 5$ then $d \geq d_0 = 8$. Then by \cref{equation: lower bound d1 3 0 0} and \cref{lemma: bound with d0} $4 \leq d_1 \leq 5$. As $G/G_m$ is finite if $\gcd(d_1,d_2) = 1$, we can discard the case $d_1 = 5$. 

If $d_3 = 5$, then $10 \mid d$ and so $d \geq d_0 = 10$. By \cref{equation: lower bound d1 3 0 0} and \cref{lemma: bound with d0} $10/3 \leq d_1 \leq 7/2$ which does not give any solution as there are no integer numbers in that interval.

Finally, if $d_3 = 6$, then $6 \mid d$ and so $d \geq d_0 = 6$. Then, as before, we obtain $d_1 = 3$.
\end{proof}

\begin{lemma}
Let $f$ be a tamely ramified rational function of degree $d \geq 2$ defined over a field $K$ satisfying \cref{equation: setting for cases} and with $(f_1,f_2,f_3) = (3,0,0)$. If $d_1 \mid d_i$ for some $i = 2,3$ then $d_1 = E_1$.
\label{lemma: d1 is E1 3 0 0}
\end{lemma}

\begin{proof}
Since $d_i \mid d$, if $d_1 \mid d_i$ then $d_1 \mid d$. By \cref{equation: local degrees each pi} then $d_1 \mid E_1$. In particular $E_1 \geq d_1$. By \cref{lemma: bound with d0} we have 
\begin{align*}
1 \geq \frac{E_1}{d_1},
\end{align*}
and consequently $E_1 = d_1$.
\end{proof}

We are ready to prove the cases excluded in \cref{proposition: excluded cases 3 0 0}. Note that as $E_2 = 0$ and in all the cases of \cref{proposition: excluded cases 3 0 0} we have $d_2 < \infty$, we just need to check $E_1 \geq d_3$ to apply \cref{lemma: bound for E2}. Since $f_1 = 3$, we also have $E_1 \geq 3$.

If \underline{$(d_2,d_3) = (3,3)$} then $E_1 \geq d_3$ and $2d_2+d_3-d_2d_3 = E_2 = 0$. Therefore, by \cref{lemma: bound for E2} $E_1 = d_3 = 3$ which implies that $e_f(p_i) = 1$ for all $i= 1,2,3$. In this case $f$ has an euclidean orbifold $(3,3,3)$. Moreover, \cref{equation: ep1 3 0 0}, \cref{equation: ep2 3 0 0} and \cref{equation: ep3 3 0 0} are equalities giving $v_f(p_i) = e_{p_i} = 3$ for all $i = 1,2,3$. By \cref{proposition: order inertia generators} this implies that $G_m = 1$.

If \underline{$(d_1,d_2) = (\infty,2)$ and $2 \leq d_3 \leq 6$} then \cref{equation: local degrees each pi} and \cref{equation: setting for cases} give
\begin{align*}
d-1 = \frac{d}{2} + \frac{d}{d_3}.    
\end{align*}
Solving for $d$ in each case of $d_3$ we find that only $(d_3,d) = (3,6)$ and $(4,4)$ can happen. These cases can be solved together with the next excluded cases that we have from the list of \cref{proposition: excluded cases 3 0 0}.

If \underline{$(d_2,d_3) = (2,3)$}, observe that if $2 \nmid e_f(p_3)$ then by \cref{equation: ep2 3 0 0} we have $e_{p_2} = 1$ and $G/G_m$ is finite. Similarly, if $3 \nmid e_f(p_2)$. So, we may assume that $2 \mid e_f(p_3)$ and $3 \mid e_f(p_2)$. By \cref{lemma: bound with d0} we have $d_1 \geq E_1$. 

If $d_1 = \infty$, then $d = E_1 = 6$, so in particular $E_1 \geq 6$. If $d_1 < \infty$, we have $6 \mid d_1$. Note also that $d_2,d_3 \mid d$, so $6 \mid d$. By \cref{equation: local degrees each pi} then $6 \mid E_1$ and again $E_1 \geq 6$. Then we apply the same strategy as in \cref{lemma: bound for E2}. Since $6 \mid d$, we can write $d = 6\ell$. By \cref{equation: setting for cases}
\begin{align*}
6\ell - 1 = d-1 = m_1+m_2+m_3 = m_1 + 3\ell + 2\ell \Longrightarrow m_1 + 1 = \ell.
\end{align*} 

Therefore
\begin{align*}
E_1 \ell \geq 6 \ell = d = E_1 + d_1 m_1 \geq E_1(m_1+1) = E_1 \ell
\end{align*}
which implies that $E_1 = 6$.

The divisibility conditions over $e_f(p_2)$ and $e_f(p_3)$ force 
\begin{align*}
(e_f(p_1), e_f(p_2), e_f(p_3)) = (1,3,2).
\end{align*}

This rational portrait corresponds to the portrait of a euclidean orbifold $(2,3,6)$. Again $v_f(p_i) = e_{p_i}$ for $i = 1,2,3$ and therefore $G_m = 1$.

If \underline{$(d_2,d_3) = (2,4)$}, we have two cases. If $d_1 = \infty$, then $d = E_1 = 4$. If $d_1 = 4$, then $d_1 \mid d_3$ and by \cref{lemma: d1 is E1 3 0 0} $d_1 = E_1 = 4$. As $E_1$ is the sum of the local degrees of each $p_i$ and each local degree is at least $1$, then only one of the $p_i$ has $e_f(p_i) = 2$.

\begin{itemize}
\item If $e_f(p_1) = 2$, then the good pair $P = (\gamma_1,\gamma_2)$ defined as 
\begin{align*}
\left \{ \begin{array}{r@{\ :\ }l}
\gamma_1 & p_1 \rightarrow p_1 \rightarrow p_3 \\
\gamma_2 & p_1 \rightarrow p_2 \rightarrow q_{2j} 
\end{array}\right.
\end{align*}
has $\mathcal{E}(P) = 1$ and consequently $e_{p_3} = 1$ and $G/G_m$ is finite. 

\item If $e_f(p_2) = 2$, then $f$ has an euclidean orbifold $(2,4,4)$ and $e_{p_i} = v_f(p_i)$ for $i = 1,2,3$. Therefore $G_m = 1$. 

\item If $e_f(p_3) = 2$, by \cref{equation: ep1 3 0 0} we have $e_{p_1} \mid 2$ and by \cref{equation: ep3 3 0 0} we have $e_{p_3} \mid 2$. Therefore $G/G_m$ is finite. 
\end{itemize}

If \underline{$(d_1,d_2,d_3) = (3,2,6)$}, then $d_1 \mid d_3$ and by \cref{lemma: d1 is E1 3 0 0} we have $d_1 = E_1 = 3$. Therefore $e_f(p_i) = 1$ for $i = 1,2,3$. Then by \textcolor{teal}{Equations} \ref{equation: ep1 3 0 0}, \ref{equation: ep2 3 0 0} and \ref{equation: ep3 3 0 0}, we conclude $e_{p_i} \mid 2$ for all $i = 1,2,3$ and consequently $G/G_m$ is finite.

\subsubsection{Case $(f_1,f_2,f_3) = (2,1,0)$}

\begin{figure}
\centering
\fbox{%
\begin{minipage}[c][3.5cm][c]{0.3\textwidth}
\centering
\begin{tikzpicture}[
>=Stealth,
every node/.style={inner sep=0pt},
arr/.style={
->,
line width=0.9pt,
shorten <=7pt,
shorten >=7pt},
scale=0.65]

\coordinate (p1) at (0,0);
\coordinate (p2) at (2.2,0);
\coordinate (p3) at (-2.2,0);

\fill (p1) circle (0.12);
\fill (p2) circle (0.12);
\fill (p3) circle (0.12);

\draw[arr]
(p3) to node[above=5pt] {$e_f(p_3)$} (p1);

\draw[arr,bend left=25]
(p1) to node[above=5pt] {$e_f(p_1)$} (p2);

\draw[arr,bend left=25]
(p2) to node[below=5pt] {$e_f(p_2)$} (p1);
\end{tikzpicture}
\end{minipage}}
\fbox{%
\begin{minipage}[c][3.5cm][c]{0.32\textwidth}
\centering
\begin{tikzpicture}[
>=Stealth,
every node/.style={inner sep=0pt},
arr/.style={
->,
line width=0.9pt,
shorten <=7pt,
shorten >=7pt},
scale=0.65]

\coordinate (p1) at (0,0);
\coordinate (p2) at (2.7,0);
\coordinate (p3) at (-2.2,0);

\fill (p1) circle (0.12);
\fill (p2) circle (0.12);
\fill (p3) circle (0.12);

\draw[arr]
(p3) to node[above=5pt] {$e_f(p_3)$} (p1);

\draw[arr]
(p1) .. controls +(1.2,0.7) and +(1.2,-0.7) ..
node[midway,above=10pt] {$e_f(p_1)$} (p1);

\draw[arr]
(p2) .. controls +(-1.2,0.7) and +(-1.2,-0.7) ..
node[midway,below=10pt] {$e_f(p_2)$} (p2);
\end{tikzpicture}
\end{minipage}}
\fbox{%
\begin{minipage}[c][3.5cm][c]{0.3\textwidth}
\centering
\begin{tikzpicture}[
>=Stealth,
every node/.style={inner sep=0pt},
arr/.style={
->,
line width=0.9pt,
shorten <=7pt,
shorten >=7pt},
scale=0.65]

\coordinate (p1) at (2,0);
\coordinate (p2) at (0,0);
\coordinate (p3) at (-2,0);

\fill (p1) circle (0.12);
\fill (p2) circle (0.12);
\fill (p3) circle (0.12);

\draw[arr]
(p1) .. controls +(1.2,0.7) and +(1.2,-0.7) ..
node[midway,below=10pt] {$e_f(p_1)$} (p1);

\draw[arr]
(p3) to node[above=7pt] {$e_f(p_3)$} (p2);

\draw[arr]
(p2) to node[above=7pt] {$e_f(p_2)$} (p1);
\end{tikzpicture}
\end{minipage}}
\caption{Possible ramification portraits for $(f_1,f_2,f_3) = (2,1,0)$.}
\label{figure: Ramification portraits case 2 1 0}
\end{figure}
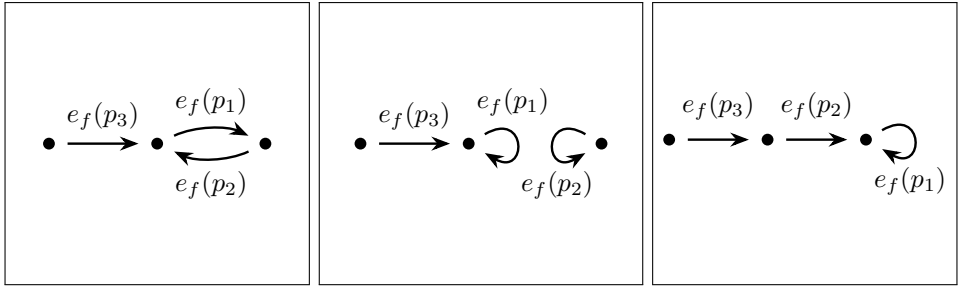

\cref{figure: Ramification portraits case 2 1 0} shows the three possible diagrams of this case. Notice $p_1$ is always in a cycle and also in the forward orbit of $p_3$. As $f_3 = 0$, then $m_3 > 0$ and $d_3 \mid d$. 
Using the good pair $P = (\gamma_1,\gamma_2)$ defined as 
\begin{align*}
\left \{ \begin{array}{r@{\ :\ }l}
\gamma_1 & p_1 \rightarrow \dots \rightarrow p_1 \rightarrow \dots \rightarrow p_3, \\
\gamma_2 & p_1 \rightarrow \dots \rightarrow p_i \rightarrow q_{ij}, 
\end{array}\right.
\end{align*}
with $i = 1,3$, we conclude $e_{p_3} \mid \gcd(d_1,d_3)$. So in particular $G/G_m$ is finite if $\gcd(d_1,d_3) = 1$.

\begin{lemma}
\label{lemma: di is Ei 2 1 0}
Let $f$ be a tamely ramified rational function of degree $d \geq 2$ defined over a field $K$ satisfying \cref{equation: setting for cases} and with $(f_1,f_2,f_3) = (2,1,0)$. Let $\set{i,k} = \set{1,2}$. If $d_i \mid d_3$ then 
\begin{align*}
E_i = d_i \quad \text{ and } \quad d_k = \infty.
\end{align*}
\end{lemma}

\begin{proof}
As in \cref{lemma: d1 is E1 3 0 0} if $d_i \mid d_3$, since $d_3 \mid d$ then $d_i \mid d$ and consequently $d_i \mid E_i$. In particular $E_i \geq d_i$. Then, by \cref{lemma: bound with d0}
\begin{align*}
1 \geq \frac{E_1}{d_1} + \frac{E_2}{d_2} \geq 1 + \frac{E_k}{d_k} 
\end{align*}
which forces $E_i = d_i$ and $d_k = \infty$.
\end{proof}

As in the case $(f_1,f_2,f_3) = (3,0,0)$, the following proposition will isolate the tuples $(d_1,d_2,d_3)$ that will require special attention:

\begin{proposition}
Let $f$ be a tamely ramified rational function of degree $d \geq 2$ defined over a field $K$ satisfying \cref{equation: setting for cases} and with $(f_1,f_2,f_3) = (2,1,0)$. If $(d_1,d_2,d_3)$ is none of the following tuples:
\begin{align*}
(\infty,2,4), \,\, (\infty,2,3), \,\, (\infty,3,3), \,\, (\infty,\infty, 2), \,\, (\infty,2,2), \,\, (\infty,3,2), \,\, (\infty,4,2), \\
(6,2,3) \,\, (9,2,3), \,\, (2,\infty,2), \,\, (4,4,2), \,\, (4,5,2), \,\, (6,3,2) \,\, \text{ and } \,\, (8,3,2),
\end{align*}
then $G/G_m$ is finite.
\label{proposition: excluded cases 2 1 0}
\end{proposition}

\begin{proof}
By \cref{lemma: bound with d0} with $d_0 = 2$ we obtain
\begin{align*}
1 \leq \frac{1}{d_2} + \frac{2}{d_3}
\end{align*}
whose solutions are $(2,4), (2,3), (3,3)$ and $(d_2,2)$ for $d_2 \leq \infty$.

By \cref{lemma: di is Ei 2 1 0} in the cases $(2,4)$ and $(3,3)$, we are forced to put $d_1 = \infty$. For the other cases, if $d_1 = \infty$, we isolate the tuples $(\infty,2,3)$ and $(\infty,\infty,2)$. If $d_2 < \infty$, then by \cref{equation: local degrees each pi} we have $d \geq 1 + d_2 \geq 3$ and so applying \cref{lemma: bound with d0} with $d_0 = 3$ and $(d_1,d_3) = (\infty, 2)$ we obtain the upper bound $d_2 \leq 4$.  

If $d_1 < \infty$ by \cref{equation: local degrees each pi} we have $d \geq 2 + d_1$. 

If $(d_2,d_3) = (2,3)$ then we may assume $3 \mid d_1$ (otherwise $\gcd(d_1,d_3) = 1$ and $G/G_m$ is finite) and so $d \geq 5$. From \cref{equation: setting for cases} and \cref{lemma: bound with d0} with $d_0 = 5$ we obtain $6 \leq d_1 \leq 9$. If $\gcd(d_1,d_3) = 1$ then $G/G_m$ is finite so we isolate the cases of $d_1 = 6$ and $9$.

If $d_3 = 2$ then we may assume $d_1$ is even. From \cref{equation: setting for cases} we deduce
\begin{align*}
\frac{2d_1}{d_1-2} \leq d_2.
\end{align*}
Now we use \cref{lemma: bound with d0} with $d_0 = 2 + d_1$:
\begin{align*}
d_1 + 1 \leq 1 + \frac{d_1+1}{d_2} + \frac{d_1+2}{2} \Longrightarrow d_2 \leq \frac{2(d_1+1)}{d_1-2},
\end{align*}
namely, 
\begin{align*}
\frac{2d_1}{d_1-2} \leq d_2 \leq \frac{2(d_1+1)}{d_1-2}.
\end{align*}
Note that $d_2 \geq 3$ and the upper bound converges to $2$ when $d_1 \rightarrow +\infty$. Thus, we only need to worry of finitely many values of $d_1$. Concretely, imposing that the upper bound must be at least $3$, we conclude that $d_1 \leq 8$. As we assume $d_1$ is even, we only need to consider the cases of $d_1 = 2,4,6$ and $8$.

If $d_1 = 2$, then $d_2 = \infty$. If $d_1 = 4$, then $4 \leq d_2 \leq 5$. If $d_1 = 6$ then $3 \leq d_2 \leq 7/2$, which implies that $d_2 = 3$ and if $d_2 = 8$ then $d_2 = 3$.
\end{proof}

Although the graphs in \cref{figure: Ramification portraits case 2 1 0} correspond to $(f_1,f_2,f_3) = (2,1,0)$, the arguments can vary in some cases. We call the case $i$ to the one that satisfies that $f(p_i) = p_2$. 

Except when the case is $2$ and $k = 2$, we use the good pair $P = (\gamma_1,\gamma_2)$ defined as
\begin{align*}
\left \{ \begin{array}{r@{\ :\ }l}
\gamma_1 & p_1 \rightarrow \dots \rightarrow p_1 \rightarrow \dots \rightarrow p_k, \\
\gamma_2 & p_1 \rightarrow \dots \rightarrow p_i \rightarrow q_{ij}, 
\end{array}\right.
\end{align*}
for $i,k = 1,2,3$.

If the case is $2$ and $k = 2$, we use the good pair $P = (\gamma_1,\gamma_2)$ defined as
\begin{align*}
\left \{ \begin{array}{r@{\ :\ }l}
\gamma_1 & p_2 \rightarrow p_2 \\
\gamma_2 & p_2 \rightarrow q_{2j}, 
\end{array}\right.
\end{align*}

Combining, we obtain the following equations:

For the case 1:
\begin{align*}
e_{p_1} &\mid \gcd(d_1,d_2,e_f(p_3) d_3) \\ 
e_{p_2} &\mid \gcd(d_1,d_2,e_f(p_3)d_3), \\ 
e_{p_3} &\mid \gcd(d_1,d_2,d_3).
\end{align*}

For the case 2:
\begin{align*}
e_{p_1} &\mid \gcd(d_1,e_f(p_3) d_3), \\ 
e_{p_2} &\mid \gcd(d_2), \\ 
e_{p_3} &\mid \gcd(d_1,d_3).
\end{align*}

For the case 3:
\begin{align*}
e_{p_1} &\mid \gcd(d_1,e_f(p_2)d_2,e_f(p_2)e_f(p_3) d_3), \\ 
e_{p_2} &\mid \gcd(d_1,d_2,e_f(p_3)d_3), \\ 
e_{p_3} &\mid \gcd(d_1,d_2,d_3).
\end{align*}

Notice that in cases $1$ and $3$, if $\gcd(d_1,d_2,d_3) = 1$ then $e_{p_3} = 1$ and $G/G_m$ is finite. So, many cases in the list of \cref{proposition: excluded cases 2 1 0} will only be treated for case 2.

If \underline{$(d_1,d_2,d_3) = (\infty,2,4)$} by \cref{lemma: bound with d0}
\begin{align*}
1 \geq \frac{E_1}{d_1} + \frac{E_2}{d_2} \geq \frac{E_2}{d_2}.
\end{align*}
In particular $d_2 = 2 \geq E_2$. If we have the equality, by \cref{lemma: if E2 is d2} we should have $d_2 d_3 = d_2 + d_3$, which is false in this case. therefore $E_2 = 1$. But then from \cref{equation: local degrees each pi}
\begin{align*}
d &= E_2 + d_2m_2 \equiv 1 \pmod{2} \\
d &= d_3m_3 \equiv 0 \pmod{2}
\end{align*}
giving a contradiction. Hence, this case is impossible.

If \underline{$(d_2,d_3) = (2,3)$} by the same argument used before $E_2 \leq 2$. By \cref{lemma: if E2 is d2} the equality only holds if $d_1 = \infty$ and $d_2d_3 = d_2+d_3$ which is again false. Therefore $E_2 = 1$. On the other hand we know $E_1 \geq 2$. Suppose $E_1 = 2$. If $d_1 = \infty$, then $d = E_1 = 2$ and consequently $2 \mid E_2$ by \cref{equation: local degrees each pi}, which gives a contradiction. If $d_1 < \infty$, then by \cref{proposition: excluded cases 2 1 0} we have $3 \mid d_1$ and again by \cref{equation: local degrees each pi} we have $3 \mid E_1$, which contradicts that $E_1 = 2$.

Therefore, we have $E_1 \geq 3 = d_3$. So, by \cref{lemma: bound for E2}
\begin{align*}
1 = E_2 \leq 2d_2+d_3-d_2d_3 = 1.
\end{align*}
Since the equality holds, by the second part of \cref{lemma: bound for E2} we conclude $E_1 = d_3 = 3$.

As $\gcd(d_2,d_3) = 1$, we just need to worry about case 2. 
In case 2 we have $E_1 = e_f(p_1) + e_f(p_3) = 3$, and therefore we have two subcases. If $e_f(p_1) = 2$ then $(e_{p_1},e_{p_2}, e_{p_3}) = (3,2,3)$ and consequently $G/G_m$ is finite. If $e_f(p_3) = 2$ then $(e_{p_1},e_{p_2}, e_{p_3}) = (3,2,3)$ if $d_1 = 9$ which gives again that $G/G_m$ is finite or $(e_{p_1},e_{p_2}, e_{p_3}) = (6,2,3)$ if $d_1 \in \set{6,\infty}$. These last two cases correspond to an euclidean orbifold of signature $(2,3,6)$ where $e_{p_i} = v_f(p_i)$ for $i = 1,2,3$. Thus $G_m = 1$.

If \underline{$(d_1,d_2,d_3) = (\infty,3,3)$} \cref{lemma: di is Ei 2 1 0} gives $E_2 = d_2 = 3$. Then by \cref{lemma: if E2 is d2} we must have $d_2d_3 = d_2+d_3$. However, this does not hold and therefore this case is impossible.

If \underline{$(d_1,d_2,d_3) = (\infty,\infty,2)$} then by \cref{equation: setting for cases}
\begin{align*}
d-1 = \frac{d}{2}
\end{align*}
which implies that $d = E_1 = E_2 = 2$. We now split by the three cases in \cref{figure: Ramification portraits case 2 1 0}. 

In case 1 (namely $f(p_1) = p_2$), using the good pair $P = (\gamma_1,\gamma_2)$ defined as  
\begin{align*}
\left \{ \begin{array}{r@{\ :\ }l}
\gamma_1 & p_1 \rightarrow p_2 \rightarrow p_1 \\
\gamma_2 & p_1 \rightarrow p_3 \rightarrow q_{3j}, 
\end{array}\right.
\end{align*}
we obtain $\mathcal{E}(P) = 1$ and $e_{p_1} = 1$. Therefore $G/G_m$ is finite. 

In case 2, the point $p_2$ is a totally ramified fixed-point and therefore $\Gamma_{f,p_2} = \emptyset$. By definition $e_{p_2} = 1$ and $G/G_m$ is finite.

In case 3, the function $f$ has an euclidean orbifold $(2,4,4)$ and $(e_{p_1},e_{p_2},e_{p_3}) = (4,4,2)$. By \cref{proposition: order inertia generators} we have $G_m = 1$. \\

All the remaining cases in the list of \cref{proposition: excluded cases 2 1 0} have $d_3 = 2$. 

If $d_2 = \infty$, then $d_1 = 2$ and by \cref{lemma: epi divides di} we have $e_{p_i} \mid 2$ for $i = 1,3$. Since the exponents $e_p$ are always finite we conclude that 
\begin{align*}
\frac{1}{e_{p_1}} + \frac{1}{e_{p_2}} + \frac{1}{e_{p_3}} > 1
\end{align*}
and consequently $G/G_m$ is finite.

All the remaining cases have $d_3 = 2$ and $d_2 < \infty$. If $E_2 = d_2$, by \cref{lemma: if E2 is d2} then $d_1 = \infty$ and $d_2 = 2$. This in particular implies that $(d_1,d_2,d_3) = (\infty,4,2)$ is impossible. By \cref{lemma: epi divides di} we have $e_{p_i} \mid d_i = 2$ for $i = 2,3$ and since $e_{p_1}$ is always finite we conclude that $G/G_m$ is finite.

Therefore, all the remaining cases have now $d_3 = 2$ and $E_2 < d_2 < \infty$. As $E_1 \geq 2 = d_3$ by \cref{lemma: bound for E2} 
\begin{align*}
E_2 \leq 2d_2+d_3-d_2d_3 = 2.
\end{align*}
In particular we conclude that if $d_3 = 2$ and $E_2 < d_2$ then 
\begin{align}
E_2 = 2 \Longrightarrow E_1 = 2.
\label{equation: E2 = 2 then E1 = 2}
\end{align}

If \underline{$(d_2,d_3) = (3,2)$} then $\gcd(d_2,d_3) = 1$ and therefore $G/G_m$ is finite for cases 1 and 3. So, we may assume that we are in the case 2. If $E_2 = 2$, then $E_1 = 2$ by \cref{equation: E2 = 2 then E1 = 2}, which implies that $e_f(p_3) = 1$. Thus $e_{p_1} \mid \gcd(d_1,d_3) \mid 2$ and also $e_{p_3} \mid 2$, giving that $G/G_m$ is finite. So, we may assume that $E_2 = 1$.

If $d_1 = 6$ and $3 \nmid e_f(p_3)$ then $e_{p_1}, e_{p_3} \mid 2$ and $G/G_m$ is finite. If $3 \mid e_f(p_3)$ then $E_1 \geq 4$. On the other hand by \cref{lemma: bound with d0}
\begin{align*}
\frac{2 d_1}{3} = 4 \geq E_1,
\end{align*}
which gives $e_f(p_1) = 1$. Then $f$ has an euclidean orbifold of signature $(2,3,6)$ with $(e_{p_1},e_{p_2},e_{p_3}) = (6,3,2)$. Moreover $v_f(p_i) = e_{p_i}$ for $i = 1,2,3$ which implies that $G_m = 1$ by \cref{proposition: order inertia generators}.

If $d_1 = 8$ and $2 \nmid e_f(p_3)$ then $e_{p_1}, e_{p_3} \mid 2$ and $G/G_m$ is finite. If $2 \mid e_f(p_3)$ then $E_1 \geq 3$. By \cref{lemma: bound with d0}
\begin{align*}
\frac{2 d_1}{3} = \frac{16}{3} \geq 5 \geq E_1,
\end{align*} 
and so $e_f(p_3)$ can either be $2$ or $4$. If $e_f(p_3) = 2$ then $e_{p_1} \mid 4$, $e_{p_2} \mid 3$ and $e_{p_3} \mid 2$ and therefore $G/G_m$ is finite. If $e_f(p_3) = 4$ then $E_1 = 5$ but from \cref{equation: local degrees each pi}
\begin{align*}
d &= 5 + 8m_1 \equiv 1 \pmod{2} \\
d &= 2m_3 \equiv 0 \pmod{2}.
\end{align*}

Finally, if $d_1 = \infty$ by \cref{equation: setting for cases}
\begin{align*}
d-1 = \frac{d-1}{3} + \frac{d}{2},
\end{align*}
which implies that $d = E_1 = 4$. If $e_f(p_1) = 1$, then $e_f(p_3) = 3$ and $f$ has an euclidan orbifold of signature $(2,3,6)$. Again $G_m = 1$ by \cref{proposition: order inertia generators}. If $e_f(p_1) = 2$, then $e_f(p_3) = 2$ and the good pair $P = (\gamma_1, \gamma_2)$ defined as 
\begin{align*}
\left \{ \begin{array}{r@{\ :\ }l}
\gamma_1 & p_1 \rightarrow p_1 \rightarrow p_3 \\
\gamma_2 & p_1 \rightarrow p_3 \rightarrow q_{3j}, 
\end{array}\right.
\end{align*}
has conditional energy $\mathcal{E}(P) = 1$. Therefore $G/G_m$ is finite. If $e_f(p_1) = 3$, then $e_f(p_3) = 1$ and $e_{p_1}, e_{p_3} \mid 2$. Therefore $G/G_m$ is finite.

If \underline{$(d_2,d_3) = (4,2)$} we always have $e_{p_3} \mid 2$. If $2 \nmid e_f(p_3)$ then $e_{p_1} \mid 2$ in cases 1 and 2 and $e_{p_2} \mid 2$ in case 3. Therefore $G/G_m$ is finite. Thus, we may assume that $2 \mid e_f(p_3)$. In cases 1 and 2 this means that $E_1 \geq 3$ and in case 3 that $E_2 = 2$. For cases 1 and 2 we also have that $E_2 = 1$ as otherwise we contradict \cref{equation: E2 = 2 then E1 = 2}.

If $d_1 = 4$ by \cref{lemma: bound with d0}
\begin{align*}
1 \geq \frac{E_1}{4} + \frac{E_2}{4}.
\end{align*}
In cases 1 and 2 this means that $E_1 = 3$. Then, all the cases have $$(e_f(p_1), e_f(p_2), e_f(p_3)) = (1,1,2)$$ and $f$ has an euclidean orbifold with signature $(2,4,4)$. Furthermore $(e_{p_1},e_{p_2},e_{p_3}) = (4,4,2)$, the generators $g_{p_i}$ have order $e_{p_i}$ for $i = 1,2,3$ and $G_m = 1$. 

If $d_1 = \infty$ by \cref{equation: setting for cases}
\begin{align*}
d-1 = \frac{d-E_2}{4} + \frac{d}{2}
\end{align*}
and so $d = 3$ in cases 1 and 2 and $d = 2$ in case 3. However, this is impossible as $d_2 = 4$ which implies that $d \geq 4$.

If \underline{$(d_1,d_2,d_3) = (4,5,2)$} as $\gcd(d_2,d_3) = 1$ then $G/G_m$ is finite in cases 1 and 3, so we may assume we are in case 2. If $2 \nmid e_f(p_3)$ then $e_{p_1},e_{p_3} \mid 2$ and $G/G_m$ is finite. Therefore $E_1 \geq 3$. By \cref{equation: E2 = 2 then E1 = 2} this implies that $E_2 = 1$. By \cref{lemma: bound with d0} $E_1 \leq 3$ and therefore $E_1 = 3$. But now by \cref{equation: local degrees each pi}
\begin{align*}
d &= E_1 + 4m_1 \equiv 1 \pmod{2} \\
d &= 2m_3 \equiv 0  \pmod{2},
\end{align*}
giving a contradiction. \\

All the analysis made in \textcolor{teal}{Sections} \ref{section: The hyperbolic orbifold case} and \ref{section: The hyperbolic orbifold case and Von Dyck groups} allows us to conclude:

\begin{theorem}
Let $f$ be a tamely ramified rational function of degree $d \geq 2$ defined over a field $K$ with hyperbolic orbifold. If $G = G_\infty(K^{\textrm{sep}},f,t)$, with $t$ transcendental over $K$, is a martingale group, then $G$ is virtually mixing with a delay constant $N$ that only depends on the ramification portrait of $f$ and its mixing subgroup $G_m$ is the one defined in \cref{definition: Gm}. Moreover $G/G_m$ is a finite Von Dyck group.
\label{theorem: virtually mixing hyperbolic case}
\end{theorem}

\subsection{The euclidean orbifold case}
\label{section: The euclidean orbifold case}

From the analysis of the previous section we deduce the following result:

\begin{proposition}
Let $f$ be a tamely ramified rational function of degree $d \geq 2$ defined over a field $K$. Let $G = G_\infty(K^{\textrm{sep}},f,t)$, with $t$ transcendental over $K$ and $G_m$ the subgroup defined in \cref{definition: Gm}. Then $G_m = 1$ if and only if $f$ has an euclidean orbifold where $v_f(p) < \infty$ for every $p \in P_f$.
\end{proposition}

Therefore, for these cases we need to consider another subgroup different than the $G_m$ defined in \cref{definition: Gm} to prove that the geometric iterated Galois group is virtually mixing.

\subsubsection{Iterated monodromy groups}

Let $X$ be a connected and locally path-connected topological space and $F: X \rightarrow X$ a finite $d$-cover with post-critical set $P_F$. Then $F^n:X \rightarrow X$ is also a finite cover. Choose $x_0 \in X \setminus P_F$ a base point and consider the \textit{monodromy action} of $\pi_1(X \setminus P_F,x_0) \curvearrowright f^{-n}(x_0)$. These different actions are compatible, which allows us to merge them in one single action $\rho_{\IMG}$ over the $d$-regular rooted tree of preimages of $x_0$ under $F$. The tree of preimages $(T_{x_0},x_0)$ is a $d$-regular rooted as $F^n: X \setminus F^{-n}(P_F) \rightarrow X \setminus P_F$ is an unramified cover. 

The \textit{iterated monodromy group} of $F$ is defined as
\begin{align*}
\IMG(F) = \IIm(\rho_{\IMG}).
\end{align*}
We denote $\overline{\IMG(F)}$ to the closure of $\IMG(F)$ in $\Aut(T_{x_0})$, using the congruence topology.

We are interested in some results of iterated monodromy groups that will be used in this section. For a more general description of iterated monodromy groups and their properties, see \cite{Self_similar_groups}.

\begin{proposition}[{\cite[Proposition 5.6.1]{Self_similar_groups}}]
Let $X,Y$ two connected and locally path-connected topological spaces and three continuous maps $F: X \rightarrow X$, $f: Y \rightarrow Y$ and $\pi: X \rightarrow Y$ such that $\pi$ is an open map and $\pi \circ F = f \circ \pi$. Then $\pi$ induces an injective map $\pi_*: \IMG(F) \rightarrow \IMG(f)$ that agrees with the action on the $d$-regular rooted tree $T$.
\label{proposition: injection between IMGs}
\end{proposition}

\begin{theorem}[{\cite[Proposition 6.4.2]{Self_similar_groups}}]
\label{theorem: iterated Galois group over C and IMG}
Let $f \in \CC(x)$ be a post-critically finite rational function of degree $d \geq 2$ and $t$ transcendental over $\CC$. Then 
\begin{align*}
G_\infty(\CC,f,t) \cong \overline{\IMG(f)}.
\end{align*}
\end{theorem}

\subsubsection{From iterated Galois groups to iterated monodromy groups over $\CC$}

We start proving that geometric iterated Galois groups are invariant under field extensions:

\begin{proposition}
\label{proposition: geometric Galois group invarint under field extensions}
Let $F$ be a rational function of degree $d \geq 2$ defined over a field $K$ which is a subfield of a field $L$. Let $t$ be transcendental over $L$. Then
\begin{align*}
G_\infty(\Ksep,F,t) \cong G_\infty(L^{\textrm{sep}},F,t).
\end{align*}
In particular if $F$ is post-critically finite and $K$ is a subfield of $\CC$ then 
\begin{align*}
G_\infty(\Ksep,F,t) \cong \overline{\IMG(F)}.
\end{align*}
\end{proposition}

\begin{proof}
As $L^{\textrm{sep}}_\infty(F,t)$ contains $\Ksep_\infty(F,t)$, the restriction map induces a group homomorphism
\begin{align*}
G_\infty(L^{\textrm{sep}},F,t) &\rightarrow G_\infty(\Ksep,F,t) \\
g &\mapsto g|_{\Ksep_\infty(F,t)}.
\end{align*}

As both fields $\Ksep$ and $L^\textrm{sep}$ are separably closed we have $L^{\textrm{sep}}(t) \cap \Ksep_\infty(F,t) = \Ksep(t)$, so the map is injective. Surjectivity follows from the fact that preimages of $t$ under $F^n$ determine the elements in the Galois group, and this is independent of the base field. The second part follows from \cref{theorem: iterated Galois group over C and IMG}.
\end{proof}

\begin{theorem}
Let $f$ be a post-critically finite tamely ramified rational function of degree $d \geq 2$ defined over a field $K$ of characteristic zero and with cardinality at most $2^{\aleph_0}$. Let $t$ be transcendental over $K$. Then, 
\begin{align*}
G_\infty(\Ksep,f,t) \cong \overline{\IMG(f)}.
\end{align*}
\label{theorem: lifting to complex numbers}
\end{theorem}

\begin{proof}
By the hypothesis over the cardinality of $K$, we can embed $\K-$ inside $\CC$ and hence $G_\infty(K,f,t) \cong \overline{\IMG(f)}$ by \cref{proposition: geometric Galois group invarint under field extensions}.
\end{proof}

\subsubsection{Iterated monodromy groups of exceptional rational functions}

By the previous two sections, we just need to worry about post-critically finite maps over $\PP^1(\CC)$ whose orbifold is euclidean. Fortunately, this was already classified by Douady and Hubbard in \cite{DoaudyHubbard1993}.

A \textit{lattice} $\Lambda \subseteq \CC$ is an abelian group isomorphic to $\ZZ^2$ whose generators are $\RR$-linearly independent. The resulting quotient $\CC/\Lambda$ is a \textit{complex torus}. Holomorphic maps $L: \CC/\Lambda \rightarrow \CC/\Lambda$ between complex torii are induced by linear maps $F(z) = az+b$ such that $a\Lambda \subseteq \Lambda$ (see \cite[Proposition 1.3.2]{DiamondShurman2005}).

\begin{theorem}[{\cite[Page 289 and Propositions 9.3]{DoaudyHubbard1993}}]
\label{theorem: classification euclidean orbifolds over C}
If $f: \PP^1(\CC) \rightarrow \PP^1(\CC)$ is a post-critically finite map with euclidean orbifold and $v_f(z) < \infty$ for all $z \in \PP^1(\CC)$, then there exists a lattice $\Lambda$, a linear map $F$ and a finite cyclic subgroup $H$ of $\Aut(\CC/\Lambda)$ such that the following diagram commute:
\begin{equation}
\centering
\begin{tikzcd}
\CC/\Lambda \arrow{rr}{F} \arrow{dd}{\pi} & & \CC/\Lambda \arrow{dd}{\pi} \\
 & & \\
\PP^1(\CC) \arrow{rr}{f} & & \PP^1(\CC)
\end{tikzcd}
\label{diagram: lattes maps}
\end{equation}
Here $\pi$ corresponds to the quotient map by $H$. Moreover, if we write $F(z) = az+b$ and $H = \<z \mapsto \epsilon z>$, we have:
\begin{enumerate}[(1)]
\item If the signature of the orbifold of $f$ is $(2,2,2,2)$ then there exists $\tau$ in the upper half plane such that $\Lambda = \ZZ \oplus \tau\ZZ$ and $\epsilon = -1$.

\item If the signature of the orbifold of $f$ is $(2,4,4)$ then $\Lambda = \ZZ \oplus i\ZZ$ and $\epsilon = i$.

\item If the signature of the orbifold of $f$ is $(2,3,6)$ then $\Lambda = \ZZ \oplus e^{\frac{\pi i}{3}}\ZZ$ and $\epsilon = e^{\frac{\pi i}{3}}$.

\item If the signature of the orbifold of $f$ is $(3,3,3)$ then $\Lambda = \ZZ \oplus e^{\frac{\pi i}{3}}\ZZ$ and $\epsilon = e^{2\frac{\pi i}{3}}$.
\end{enumerate}
The degree of $f$ is $\abs{a}^2$ and in every case it is satisfied that $\epsilon \Lambda = \Lambda$ and $(\epsilon^k-1)b \in \Lambda$ for all $k \in \ZZ$.
\end{theorem}

To calculate $\IMG(f)$ in the cases of \cref{theorem: classification euclidean orbifolds over C}, let $H$ be the group generated by $\epsilon$ in $S^1$. Since $\PP^1(\CC)$ is obtained by a quotient under $H$, by \cref{proposition: injection between IMGs} 
\begin{align*}
\IMG(f) \cong \IMG(F) \rtimes H. 
\end{align*}

To calculate $\IMG(F)$ we use the following result: 

\begin{theorem}[{\cite[Theorem 6.1.3]{Self_similar_groups}}]
\label{theorem: IMG in compact manifolds}
If $X$ is a compact Riemannian manifold and $F: X \rightarrow X$ is an expanding endomorphism, then
\begin{align*}
\IMG(F) \cong \pi_1(X).
\end{align*}
\end{theorem}

In our case $X = \CC/\Lambda$ and $F(z) = az+b$. Notice that $F$ is expanding as $\abs{a} = \sqrt{\deg(f)} \geq \sqrt{2}$. Then $\IMG(F) \cong \Lambda$ and 
\begin{align*}
\IMG(f) \cong \Lambda \rtimes H.
\end{align*}

Our next goal is to understand the action of $\IMG(f)$ over the $d$-regular rooted tree. $As \abs{a} > 1$, we can tesselate the complex torus by $a^{-n}\Lambda$. Then, we have a bijection between the points in the complex torus and $\varprojlim (\Lambda +b)/(a^n \Lambda + b)$. 

Let $g \in \IMG(f)$. Observe that if $g \in \IMG(f)$, then $$g(\Lambda +b) = \Lambda + b.$$

Indeed, the action of $g$ is given by $g(x) = \epsilon^k x+\lambda$ for some $k \in \ZZ$ and $\lambda \in \Lambda$. Then, by \cref{theorem: classification euclidean orbifolds over C} 
\begin{align*}
g(\Lambda + b) = \epsilon^k \Lambda + \epsilon^k b + \lambda = \Lambda + (\epsilon^k-1)b + b = \Lambda + b.
\end{align*}

Therefore, in order to understand the action of $g$ over $\CC/\Lambda$, we can understand the action of $g$ over any coset of $(\Lambda + b)/(a \Lambda + b)$ and then use the recursion to construct the action on the tree.

Let $\mathcal{R} := \set{r_1, \dots, r_d}$ be a set of representatives of $\Lambda/a \Lambda$. Then, any element of $\Lambda+b$ can be written as $au+r_i+b$ where $u \in \Lambda$. The action of $g$ over this element is
\begin{align*}
g(au+r_i+b) = a\epsilon^ku + \epsilon^k r_i + \epsilon^k b + \lambda.
\end{align*}
As $g(\Lambda + b) = \Lambda + b$ and $a\epsilon^ku \in a\Lambda$, then $\epsilon^k r_i + \epsilon^k b + \lambda \in \Lambda + b$, so there exist unique $u_i \in \Lambda$ and $r_{\sigma(i)} \in \mathcal{R}$ such that 
\begin{align*}
g(au+r_i+b) = a(\epsilon^ku + u_i) + r_{\sigma(i)} + b.
\end{align*}

Thus, the action on the first level is given by the permutation $\sigma$ and the section at the vertex $i$ by $(g|_i)(x) = \epsilon^k x + u_i$.

\begin{theorem}
Let $f$ be a tamely ramified rational function of degree $d \geq 2$ defined over a field $K$ of characteristic zero and with cardinality at most $2^{\aleph_0}$, and $f$ has euclidean orbifold with $v_f(z) < \infty$ for all $z \in \PP^1(\CC)$. If $G = G_\infty(K^{\textrm{sep}},f,t)$, with $t$ transcendental over $K$, then $G$ is a martingale group and it is virtually mixing with delay constant $N = 0$. The mixing subgroup is $\overline{\IMG(F)}$ where $F$ is the map from \cref{diagram: lattes maps}. Then $G/\overline{\IMG(F)}$ is a finite cyclic group of order dividing $\abs{H}$, in particular a Von Dyck group.
\label{theorem: virtually mixing euclidean case}
\end{theorem}

\begin{proof}
By \cref{theorem: lifting to complex numbers}, we may assume $f$ is defined over $\CC$ and it fits in \cref{diagram: lattes maps}. Then $G \cong \overline{\IMG(f)}$.

Let $\lambda \in \Lambda$ and $g \in \IMG(f)$ defined as $g(x) = x + \lambda$. Consider $\widetilde{g}(x) = x + a^n \lambda$. Then
\begin{align*}
\widetilde{g}(au+r_i + b) = a(u + a^{n-1} \lambda) + r_i + b,
\end{align*}
which implies that $\pi_1(\widetilde{g}) = \id$ and $(\widetilde{g}|_i)(x) = x + a^{n-1} \lambda$ for all $i = 1,\dots,d$. Continuing in this fashion, we obtain that $\widetilde{g} \in \St_G(n)$ and $\widetilde{g}|_v = g$ for all $v \in \mathcal{L}_n$. In particular $\St_G(n)_v \geq \IMG(F)$. Then, by \cref{lemma: virtually mixing closed normal closure}, we conclude $\St_G(n)_v \geq \overline{\IMG(F)}$. 

To prove that $G$ is a martingale group, let $n \geq 0$ and $v \in \mathcal{L}_n$. Given $i,j \in \set{1,\dots,d}$, choose $\lambda = r_j - r_i$. Then $g(r_i) = r_j$ and $\widetilde{g} \in \St_G(n)$ moves vertices below $v$ transitively. 
\end{proof}

We finally conclude \cref{theorem: every rational is virtually mixing}:

\begin{proof}[Proof of \cref{theorem: every rational is virtually mixing}]
The hyperbolic case follows from \cref{theorem: virtually mixing hyperbolic case} and the euclidean case from \cref{theorem: virtually mixing euclidean case}.
\end{proof}

\subsection{Arithmetic iterated Galois groups}

Let $f$ be a tamely ramified rational function of degree $d \geq 2$ defined over a field $K$ and let $t$ be transcedental over $K$. We called \textit{field of constants} of $f$ to the field $\Ksep \cap K_\infty(f,t)$. We finish this section by proving that arithmetic iterated Galois groups are also virtually mixing when the field of constants of $f$ has finite index over $K$. We start recalling the connection between geometric and arithmetic iterated Galois groups:

\begin{proposition}
\label{proposition: short exact sequence arithmetic and geometric Galois groups}
Let $f$ be a rational function of degree $d \geq 2$ defined over a field $K$ and $t$ transcendental over $K$. Then, we have the exact sequence $$1 \rightarrow G_\infty(\Ksep, f, t) \rightarrow G_\infty(K, f, t) \rightarrow \Gal(\Ksep \cap K_\infty(f,t)/K) \rightarrow 1.$$ 
\end{proposition}

\begin{proof}
As the projections $\pi_n$ are surjective, the functor $\varprojlim$ is exact (see \cite[Proposition 2.2.4]{RibesZalesskii2000}), so it is enough to prove
$$1 \rightarrow G_n(\Ksep, f, t) \rightarrow G_n(K, f, t) \rightarrow \Gal(\Ksep \cap K_n(f,t)/K) \rightarrow 1$$ for every $n \geq 1$, where 
$$K_n(f,t) := K(f^{-n}(t)).$$ It is not hard to see that $K_n(f,t)/K(t)$ is Galois and $K_n(f,t) = K_1(f^n,t)$. 
Moreover, since $G_n(K,f,t) = G_1(K,f^n,t)$, it is enough to prove it for $n = 1$. 

For the first arrow, the restriction map 
\begin{align*}
G_1(\Ksep, f, t) &\rightarrow G_1(K, f, t) \\
\sigma & \mapsto \sigma|_{K_1(f,t)}
\end{align*}
is well-defined because $K_1(f,t)/K(t)$ is Galois. Injectivity follows from the fact that $\Ksep_1(f,t)$ is the field generated by $\Ksep(t)$ and $K_1(f,t)$. 

If we call $L_1 := \Ksep(t) \cap K_1(f,t)$, then the image of the restriction map is $$\Gal(K_1(f,t)/L_1).$$ As $L_1$ is a Galois extension, by Galois theory
\begin{align*}
G_1(K, f, t)/\Gal(K_1(f,t)/L_1) \cong \Gal(L_1/K(t)).
\end{align*}
Then, one can see that the map
\begin{align*}
\varphi_1: \Gal(L_1/K(t)) &\rightarrow \Gal(\Ksep \cap K_1(f,t)/K) \\
\sigma & \mapsto \sigma|_{\Ksep \cap K_1(f,t)},
\end{align*}
that forgets the action over $t$, is an isomorphism. 
\end{proof}

We now prove \cref{Corollary: arithmetic IGG is virtually mixing}:

{
\renewcommand{\thecorollary}{3}
\begin{corollary}
Let $f$ be a tamely ramified rational function of degree $d \geq 2$ defined over a field $K$. Furthermore, if $f$ is exceptional, assume that $K$ has characteristic zero and its cardinality is at most $2^{\aleph_0}$. Let $G = G_\infty(K,f,t)$ be its arithmetic iterated Galois group. Assume that $G$ is a martingale group and the extension $\Ksep \cap K_\infty(f,t)/K$ is finite. Then

\begin{enumerate}[\normalfont(1)]
\item $G$ is virtually mixing and the delay constant $N$ only depends on the ramification portrait of $f$ (same constant as in \cref{theorem: every rational is virtually mixing}). 

\item There exists a mixing subgroup $G_m$ such that $G_m$ is closed and normal in $G$ and $G/G_m$ satisfies the exact sequence
\begin{align*}
1 \rightarrow V \rightarrow G/G_m \rightarrow \Gal(\Ksep \cap K_\infty(f,t)/K) \rightarrow 1,    
\end{align*}
where $V$ is a finite Von Dyck group.
\end{enumerate}
\end{corollary}
}

\begin{proof}
Let $G_m$ be the mixing subgroup of $G_\infty(\Ksep,f,t)$. As by hypothesis we assume $G_\infty(\Ksep,f,t) \leq_f G$, then $G_m$ is also a mixing subgroup of $G$. Therefore $G$ is virtually mixing with the same delay constant as the one used by $G_\infty(\Ksep,f,t)$. This proves (1).

Let $\widehat{G_m}$ be the normal closure of $G_m$ in $G$. As $G_m$ is closed and of finite index, then $\widehat{G_m}$ is also closed. As $G_\infty(\Ksep,f,t) \lhd G$, we have $\widehat{G_m} \leq G_\infty(\Ksep,f,t)$. Therefore, $G_\infty(\Ksep,f,t)/\widehat{G_m}$ is a quotient of $G_\infty(\Ksep,f,t)/G_m$ and since the quotient of a Von Dyck group is still a Von Dyck group, we have that $V := G_\infty(\Ksep,f,t)/\widehat{G_m}$ is a finite Von Dyck group.

By the third isomorphism theorem we have 
\begin{align*}
(G/\widehat{G_m})/(G_\infty(\Ksep,f,t)/\widehat{G_m}) \cong G/G_\infty(\Ksep,f,t),
\end{align*}
which together to \cref{proposition: short exact sequence arithmetic and geometric Galois groups} justify the exact sequence in the statement of \cref{Corollary: arithmetic IGG is virtually mixing}.
\end{proof}

\section{Dynamical pullbacks}
\label{section: dynamical pullbacks}

Let $G$ be geometric iterated Galois group of a rational function $f$ defined over a field $K$. As it was proved in \cref{section: iterated Galois groups are virtually mixing}, we know that $G$ is virtually mixing. Let $G_M$ be its maximal mixing subgroup defined in \cref{definition: maximal mixing subgroup}. In this section we prove the connection between the quotient $G/G_M$ and the geometric properties of $f$, proving \cref{theorem: dynamical pullback maximal mixing subgroup} and \cref{theorem: Gm is GM}.

\subsection{Dynamical pullbacks and $G_M$}

\begin{definition}
A rational function $f$ is a \textit{dynamical pullback} if there exists an irreducible smooth projective curve $C$, a Galois cover $\pi: C \rightarrow \PP^1(\K-)$ of degree greater than $1$ and an endomorphism $F: C \rightarrow C$ such that the diagram
\begin{equation*}
\centering
\begin{tikzcd}
C \arrow{rr}{F} \arrow{dd}{\pi} & & C \arrow{dd}{\pi} \\
 & & \\
\PP^1(\K-) \arrow{rr}{f} & & \PP^1(\K-)
\end{tikzcd}
\end{equation*}
is a pullback and $K(\pi^{-1}(t))$ is contained in $K_\infty(f,t)$. 
\end{definition}

A pullback is in particular a semiconjugacy but the converse is not true (see \cite[Remark 13]{Adams2025}). The connection between dynamical pullbacks and special subgroups of the geometric iterated Galois group was first given by Adams in \cite{Adams2025}.

\begin{theorem}[{\cite[Theorem 8]{Adams2025}}]
Let $f$ be a tamely ramified rational function of degree $d \geq 2$ defined over a field $K$. Let $G = G_\infty(\Ksep,f,t)$ be its geometric iterated Galois group. Then $f$ is a dynamical pullback if and only if there exists a project-invariant open subgroup $H \neq G$.
\label{theorem: dynamical pullback project invariant}
\end{theorem}

We now prove \cref{theorem: dynamical pullback maximal mixing subgroup}. The direct implication follows the same idea as the one used in \cite[Proposition 2.5]{FariñaRadi2026FPP}.

{
\renewcommand{\thetheorem}{4}
\begin{theorem}
Let $f$ be a tamely ramified rational function of degree $d \geq 2$ defined over a field $K$. Then, $f$ is a dynamical pullback if and only if the maximal mixing subgroup of $G = G_\infty(\Ksep,f,t)$ is not $G_\infty(\Ksep,f,t)$.
\end{theorem}
}

\begin{proof}
$(\Rightarrow)$ Suppose that $G_M = G$ (in other words $G$ is mixing) and let $N$ be a project invariant normal open subgroup of $G_\infty(\Ksep,f,t)$. As it is open, there exists $m \geq 1$ such that $N \geq \St_G(m)$. As $G$ is mixing, there exists $v \in T$ such that 
\begin{align*}
G \geq N \geq N_v \geq \St_G(m)_v = G,
\end{align*}
which implies that $G = N$. Therefore $G$ does not have non-trivial normal open project invariantariant subgroups, implying that $f$ is not a dynamical pullback by \cref{theorem: dynamical pullback project invariant}.

$(\Leftarrow)$ The subgroup $G_M$ is closed by definition and of finite index by \cref{theorem: every rational is virtually mixing}. Therefore $G_M$ is open. Moreover, $G_M$ is project invariant by \cref{proposition: properties GM}. As $G_M \neq G$ then $f$ is a dynamical pullback by \cref{theorem: dynamical pullback project invariant}.
\end{proof}

\subsection{When $G_m$ equals $G_M$}

The construction of $G_M$ is theoretical, whereas the construction of $G_m$ is explicit (see \cref{definition: Gm}). As $G_M$ is the maximal mixing subgroup, we clearly have $G_M \geq G_m$. We study in this subsection when they coincide.

\begin{proposition}
Let $T$ be a $d$-regular rooted tree and $G \leq \Aut(T)$ be a virtually mixing closed group with delay constant $N$, mixing closed subgroup $H$ and maximal mixing subgroup $G_M$. If $H$ is project invariant, then $H = G_M$.
\label{proposition: project invariant mixing subgroup is maximal}
\end{proposition}

\begin{proof}
By hypothesis $H$ is closed and of finite index. Therefore it is open. Let $n_0 \in \NN$ such that $H \geq \St_G(n)$ for all $n \geq n_0$. Let $\set{v_n}_{n \geq n_0}$ be a set of vertices of $T$ such that $\abs{v_n} = n+N$. Then
\begin{align*}
H \geq H_{v_n} \geq \St_G(n)_{v_n}
\end{align*}
for all $n \geq n_0$. Taking intersection and using \cref{equation: simplification definition GM}
\begin{equation*}
H \geq \bigcap_{n \geq n_0} H_{v_n} \geq \bigcap_{n \geq n_0} \St_G(n)_{v_n} = G_M. \qedhere
\end{equation*}
\end{proof}

{
\renewcommand{\thetheorem}{5}
\begin{theorem}
Let $f$ be a tamely ramified rational function of degree $d \geq 2$ defined over a field $K$ such that $G = G_\infty(\Ksep,f,t)$ is a martingale group. Let $G_M$ be the maximal mixing subgroup of $G$ of delay constant $N$ and $G_m$ the subgroup from \cref{definition: Gm}. If $G_m$ is project invariant, then $G_m = G_M$.
\end{theorem}
}

\begin{proof}
By \cref{theorem: every rational is virtually mixing} $G$ is virtually mixing with mixing subgroup $G_m$. Then, apply \cref{proposition: project invariant mixing subgroup is maximal} with $H = G_m$.
\end{proof}

\section{Fixed-point proportion of virtually mixing groups}
\label{section: Fixed-point proportion of virtually mixing groups}

We finish this article proving \cref{theorem: FPP virtually mixing groups} and \cref{theorem: not dynamical pullbacks and FPP}. The ideas used will be a combination of the arguments in \cite{FariñaJonesRadi2026corrigendum} and \cite{FariñaRadi2026FPP}.

Let $\mathscr{T}$ be a $d$-regular rooted tree and let $G \leq \Aut(\mathscr{T})$. The \textit{fixed-point proportion} of $G$ is defined as
\begin{align}
\FPP(G) := \lim_{n \rightarrow +\infty} \frac{\# \set{g \in \pi_n(G): g \text{ fixes at least one vertex in } \mathcal{L}_n}}{\abs{\pi_n(G)}}.
\label{equation: fixed-point proportion}
\end{align}

It can be shown that the limit in \cref{equation: fixed-point proportion} always exists (see \cite[Lemma 2.9]{Radi2025FPP} for a proof).

As it was detailed in the introduction, the fixed-point proportion of iterated Galois groups is particularly important for problems in number theory and arithmetic dynamics, and it is for this reason that we devote one section of this article to study the fixed-point proportion of virtually mixing groups. 

\begin{remark}
If $G \leq \Aut(\mathscr{T})$ and $\overline{G}$ is its (topological) closure, then we have $\FPP(G) = \FPP(\overline{G})$, which allows to focus only on the case where $G$ is closed; see \cite[Lemma 2.10]{Radi2025FPP} for a proof of this fact.
\label{remark: FPP of group and closure}
\end{remark}

As $\Aut(T)$ is compact, closed subgroups are also compact and consequently they can be equipped with a unique normalized Haar measure $\mu$. The Borel $\sigma$-algebra of $G$ is generated by the sets
\begin{align*}
C_A := \set{g \in G: \pi_n(g) \in A}
\end{align*}
where $A$ is a subset of $\pi_n(G)$. We will call these sets the \textit{cone sets}. If $a \in \pi_n(G)$, to ease the notation, we will write $C_a$ instead of $C_{\set{a}}$.

Notice that if $a \in \pi_n(G)$ and $g_a$ is an element of $G$ such that $\pi_n(g_a) = a$, then we can write $C_a = g_a \St_G(n)$. Therefore, by the invariance of the Haar measure we have 
\begin{align*}
\mu_G(C_A)=\sum_{a\in A} \mu_G(g_a\mathrm{St}_G(n))=\#A\cdot \mu_G(\mathrm{St}_G(n))=\frac{\#A}{\abs{\pi_n(G)}}.    
\end{align*}

Given $G \leq \Aut(\mathscr{T})$, we define the maps $\set{X_n}_{n \geq 1}$ as 
\begin{align*}
X_n : G &\rightarrow \NN \\
g &\mapsto \# \set{v \in \mathcal{L}_n: g(v) = v}.
\end{align*}

In other words, the maps $X_n$ count how many vertices are fixed by $g$ at level $n$. It is not hard to see that if $G$ is closed, then every map $X_n$ is measurable. We call the family $\set{X_n}_{n \geq 1}$ the \textit{fixed-point process} of $G$. Then, one can prove (see \cite[page 9]{Radi2025FPP}) that in the case that $G$ is closed, its fixed-point proportion can be calculated as
\begin{align*}
\FPP(G) = \mu(\set{g \in G: X_n(g) > 0 \text{ for all } n \geq 1}.
\end{align*}

\begin{remark}
\label{remark: FPP invariant under relabeling}
If $T$ is the canonical $d$-regular rooted tree and $\phi: \mathscr{T} \rightarrow T$ is any labeling for $\mathscr{T}$, it is not hard to see that $\FPP(G) = \FPP(\rho_\phi(G))$. Therefore, for the remainder of \cref{section: Fixed-point proportion of virtually mixing groups}, we may assume that $G \leq \Aut(T)$.
\end{remark}

The clever idea of Jones to consider the fixed-point process (for the first time in \cite{Jones2007}) allows to become the calculation of the fixed-point proportion a probabilistic problem. The most important detail of his idea was to realize that under certain mild conditions of $G$, its fixed-point process is a martingale process:

\begin{proposition}[{see {\cite[Theorem 5.3]{BridyJones2022}} and \cite{Jones2008}}]
\label{proposition: martingale characterization}
Let $G \leq \Aut(T)$ be a closed subgroup. Then, its fixed-point process is a martingale process if and only if for every $n \geq 1$, the subgroup $\St_G(n)$ acts level transitively on $T_v$, for every vertex $v \in \mathcal{L}_n$.
\end{proposition}

\cref{proposition: martingale characterization} justifies our definition of martingale group used throughout the article. Basically, a group $G$ is a martingale group if and only if its fixed-point process is a martingale process. The interest in martingales is due to the following well-known theorem in probability:

\begin{theorem}[{\cite[Theorem 7, Chapter 12]{Grimmett2020}}]
Let $\{Y_n\}_{n\ge 1}$ be a non-negative martingale over a probability space $(\Omega,\PP)$ with $\mathrm{E}(Y_1)<\infty$. Then
$$\lim_{n\to\infty} Y_n(x)$$
exists for $x\in X$ almost surely and with finite value.
\end{theorem}

As in our case the fixed-point process ranges over the non-negative integer numbers, the existence of $\lim_{n \to \infty} X_n(g)$ implies that $\set{X_n(g)}_{n \geq 1}$ is eventually constant. It is worth mentioning that this constant will depend on the element $g \in G$ (see \cite{Radi2025FPP} for examples).

Our main goal in this section is to find a sufficient condition for a virtually mixing group to have zero fixed-point proportion. This will be based in the following martingale strategy first developed by Jones in \cite{Jones2007} and then generalized by Fariña-Asategui and the author in \cite[Lemma 3.2]{FariñaRadi2025FPP}. If $S_1,S_2$ are measurable subsets of $G$ such that $\mu(S_2) > 0$, recall that $\mu(S_1|S_2)$ represents the conditional probability, namely,
\begin{align*}
\mu(S_1|S_2) := \frac{\mu(S_1 \cap S_2)}{\mu(S_2)}.
\end{align*}

\begin{lemma}[{\cite[Lemma 3.2]{FariñaRadi2025FPP}}]
\label{lemma: martingale strategy}
Let $G\le \mathrm{Aut}(T)$ be a closed subgroup. Assume that:
\begin{enumerate}[\normalfont(i)]
\item the fixed-point process $\{X_n\}_{n\ge 1}$ of $G$ is a martingale and
\item for any $r>0$ there exists $\epsilon:=\epsilon(r)$ and $m:=m(r)$ such that for infinitely many $n\ge 1$ we have $\mu(X_{n+m} = r \mid  X_n = r) \leq 1 - \epsilon$.
\end{enumerate}
Then, the fixed-point process of $G$ is eventually zero almost surely, i.e. $\FPP(G) = 0$.
\end{lemma}

In order to apply \cref{lemma: martingale strategy}, we rely on ergodic properties of the following operator, first studied by Fariña-Asategui in \cite{Fariña2025ergodic}. Let $G$ be a self-similar closed subgroup of $\Aut(T)$.
Consider the right free monoid action $\mathcal{T}: T \curvearrowright G$ given by
\begin{align*}
G \times T & \rightarrow G \\
(g,v) &\mapsto g|_v.
\end{align*}

Note that this action is well-defined as we have $(g|_v)|_w = g|_{vw}$ by \cref{equation: concatenation of sections}.
This action makes the tree act on the group by taking sections of the elements. As $G$ comes equipped with a normalized Haar measure, we can study the operator $\mathcal{T}$ from the point of view of its ergodic properties. When we specify a vertex $v$ in $T$, we have the map $\mathcal{T}_v: G \rightarrow G$ given by 
\begin{align*}
\mathcal{T}_v(g) = g|_v.
\end{align*}
Therefore, if $C_A$ is a cone set, then $\mathcal{T}_v^{-1}(C_A)$ represents all the elements in $G$ whose section at $v$ is in $C_A$.

The following lemma was originally proven in \cite{FariñaJonesRadi2026corrigendum}, but we add the main ideas of the proof for completeness:

\begin{lemma}[{Subindependence lemma}]
\label{lemma: subindependence lemma}
Let $T$ be the canonical $d$-regular rooted tree and $G \leq \Aut(T)$ a closed self-similar group with normalized Haar measure $\mu$. Then, for every $n,m \geq 1$, $a \in \pi_n(G)$, $b \in \pi_m(G)$ and $v \in \mathcal{L}_n$ satisfying $C_a \cap \mathcal{T}_v^{-1}(C_b) \neq \emptyset,$ we have $$\mu(C_a \cap \mathcal{T}_v^{-1}(C_b)) \geq \mu(C_a) \cdot \mu(C_b).$$
\end{lemma}

\begin{proof}
Denote $\mathrm{id}_n = \pi_n(\mathrm{id})$ and $\mathrm{id}_m = \pi_m(\mathrm{id})$. As $C_a \cap \mathcal{T}_v^{-1}(C_b) \neq \emptyset$ and $\mu$ is translation invariant, we get 
$$\mu(C_a \cap \mathcal{T}_v^{-1}(C_b)) = \mu(C_{\mathrm{id}_n} \cap \mathcal{T}_v^{-1}(C_{\mathrm{id}_m})).$$

Let $\varphi: \pi_{n+m}(\mathrm{St}_G(n)) \rightarrow \pi_m(G)$ be the group homomorphism induced by the composition $\pi_m \circ \varphi_v \mid_{\mathrm{St}_G(n)}$. Then 
\begin{align*}
\mu(C_{\mathrm{id}_n} \cap \mathcal{T}_v^{-1}(C_{\mathrm{id}_m})) &= \frac{\abs{\ker(\varphi)}}{[G: \mathrm{St}_G(n+m)]} = \frac{1}{\abs{\IIm(\varphi)}}\cdot  \frac{\abs{\pi_{n+m}(\mathrm{St}_G(n))}}{[G: \mathrm{St}_G(n+m)]} \\
&\geq \frac{1}{\abs{\pi_m(G)}} \cdot\frac{\abs{\pi_{n+m}(\mathrm{St}_G(n))}}{[G: \mathrm{St}_G(n+m)]} \\
&= \frac{1}{|\pi_m(G)|}\cdot  \frac{1}{|\pi_n(G)|}= \mu(C_a) \cdot \mu(C_b). \qedhere
\end{align*}
\end{proof}

We are now ready to prove \cref{theorem: FPP virtually mixing groups}:

{
\renewcommand{\thetheorem}{6}
\begin{theorem}
Let $G \leq \Aut(T)$ be a martingale virtually mixing group with mixing subgroup $H$. If at every coset $S$ of $G/H$ there exists an element $s \in S$ such that $s$ fixes infinitely many ends, then $$\FPP(G) = 0.$$
\end{theorem}
}

\begin{proof}
By \cref{remark: FPP of group and closure}, we may assume that $G$ is closed. Let $\mathcal{S} := \set{s_1, \cdots, s_{[G:H]}}$ be the set of representatives of $G/H$ where each $s_i$ fixes infinitely ends. Let $r \geq 1$. As $\mathcal{S}$ is finite, there exists a constant $m := m(r)$ such that $X_m(s_i) > r$ for every $i = 1,\dots, [G:H]$.

Define
\begin{align*}
A_{n,r} := \set{a \in \pi_n(G): X_n(a) = r}.
\end{align*}
Note that $X_n(a)$ is well-defined as the amount of fixed vertices on level $n$ only depends on the action on the first $n$ levels. Assume $A_{n,r} \neq \emptyset$ as otherwise there is nothing to prove. If $a \in A_{n,r}$ with $r \geq 1$, then there exists a vertex $u_a \in \mathcal{L}_n$ such that $u_a$ is fixed by $a$. Choose any vertex $v_a \in \mathcal{L}_{n+N}$ below $u_a$ and define 
\begin{align*}
B_{a,v_a} := \set{a' \in \pi_{n+N}(G): \pi_{(n+N),n}(a') = a \text{ and } a'(v_a) = v_a}.
\end{align*}
Then, one has that
\begin{align*}
C_{B_{a,v_a}} = C_a \cap \St_G(v_a).
\end{align*}

We claim $C_{B_{a,v_a}} \neq \emptyset$. Indeed, let $g \in \pi_n^{-1}(a)$. Since $G$ is a martingale group, the stabilizer $\St_G(n)$ acts transitively below $u_a$, so we may multiply $g$ by an element $h_1 \in \St_G(n)$ such that $g h_1$ fixes $v_a$. 

If $g \in C_{B_{a,v_a}}$, as $G$ is self-similar we have $g|_{v_a} \in G$, so in particular $g|_{v_a}$ is in one coset $S$ of $G/H$. Say $s_i \in \mathcal{S} \cap S$. As $G$ is virtually mixing with delay constant $N$ and mixing subgroup $H$, there exists an element $h_2 \in \St_G(n) \cap \st_G(v_a)$ such that $h_2|_{v_a} = (g|_{v_a})^{-1} \cdot s_i$. Therefore $g h_2 \in C_{B_{a,v_a}}$ and $(g h_2)|_{v_a} = s_i$, which implies that 
\begin{align*}
C_{B_{a,v_a}} \cap \mathcal{T}_{v_a}^{-1}(C_{\pi_m(s_i)}) \neq \emptyset.
\end{align*}

By the subindependence lemma, we conclude 
\begin{align}
\mu(C_{B_{a,v_a}} \cap \mathcal{T}_{v_a}^{-1}(C_{\pi_m(s_i)})) \geq \mu(C_{B_{a,v_a}}) \cdot \mu(C_{\pi_m(s_i)}) = \mu(C_{B_{a,v_a}}) \cdot \frac{1}{\abs{\pi_m(G)}}.
\label{equation: measure Bava cap si}
\end{align}

We now estimate $\mu(C_{B_{a,v_a}})$. As $v_a$ is below $u_a$, we can write it as the concatenation $v_a = u_a w_a$ where $\abs{w_a} = N$. Then
\begin{align*}
C_{B_{a,v_a}} = C_a \cap \st_G(v_a) = C_a \cap \mathcal{T}_{u_a}^{-1}(\st_G(w_a)).
\end{align*}

As $C_{B_{a,v_a}}$ is non-empty, again by the subindependence lemma
\begin{align}
\mu(C_{B_{a,v_a}}) = \mu(C_a \cap \mathcal{T}_{u_a}^{-1}(\st_G(w_a))) \geq \mu(C_a) \cdot \mu(\st_G(w_a)).
\label{equation: measure Bava}
\end{align}

As $G$ is level-transitive, by the orbit stabilizer theorem
\begin{align*}
\mu(\st_G(w_a)) = \frac{1}{[G:\st_G(w_a)]} = \frac{1}{d^N}.
\end{align*}

Combining \cref{equation: measure Bava cap si} and \cref{equation: measure Bava}, we conclude
\begin{align*}
\mu(C_{B_{a,v_a}} \cap \mathcal{T}_{v_a}^{-1}(C_{\pi_m(s_i)})) \geq \mu(C_a) \cdot \frac{1}{d^N \abs{\pi_m(G)}}.
\end{align*}

Observe that 
\begin{align*}
C_{B_{a,v_a}} \cap \mathcal{T}_{v_a}^{-1}(C_{\pi_m(s_i)}) \subseteq \set{g \in G: X_{n+N+m}(g) > r} \cap C_a.
\end{align*}
Indeed, if $g \in C_{B_{a,v_a}} \cap \mathcal{T}_{v_a}^{-1}(C_{\pi_m(s_i)})$, then $X_{n+N+m}(g) \geq X_m(s_i) > r$.

Thus,
\begin{align*}
\mu(X_{n+N+m} > r | X_n = r) &= \frac{\mu(X_{n+N+m} > r \cap X_n = r)}{\mu(C_{A_{n,r}})} \\
& = \frac{ \sum_{a \in A_{n,r}} \mu(X_{n+N+m} > r \cap C_a)}{\mu(C_{A_{n,r}})} \\
& \geq \frac{ \sum_{a \in A_{n,r}} \mu(C_{B_{a,v_a}} \cap \mathcal{T}_{v_a}^{-1}(C_{\pi_m(s_i)}))}{\mu(C_{A_{n,r}})} \\
& \geq \frac{ \sum_{a \in A_{n,r}} \mu(C_a)} {\mu(C_{A_{n,r}})} \cdot \frac{1}{d^N \abs{\pi_m(G)}} \\ 
& = \frac{1}{d^N \abs{\pi_m(G)}}.
\end{align*}

Define $\epsilon := (d^N \abs{\pi_m(G)})^{-1}$. Notice that $\epsilon$ only depends on $r$ (because $m$ does) and the group $G$. Therefore 
\begin{align*}
\mu(X_{n+N+m} = r|X_n = r) \leq 1 - \mu(X_{n+N+m} > r|X_n = r) \leq 1-\epsilon.
\end{align*}
By \cref{lemma: martingale strategy}, the fixed-point proportion is zero.
\end{proof}

As a corollary of \textcolor{teal}{Theorems} \ref{theorem: every rational is virtually mixing}, \ref{theorem: dynamical pullback maximal mixing subgroup} and \ref{theorem: FPP virtually mixing groups} we conclude: 

{
\renewcommand{\thetheorem}{7}
\begin{theorem}
Let $f$ be a tamely ramified rational function of degree $d \geq 2$ defined over a field $K$ such that $G = G_\infty(\Ksep,f,t)$ is a martingale group. If $f$ is not a dynamical pullback, then $$\FPP(G_\infty(\Ksep,f,t)) = 0.$$
\end{theorem}
}

\begin{proof}
By \cref{proposition: iterated galois groups are fractal} and \cref{theorem: every rational is virtually mixing}, there exists a labeling such that $G$ embeds in $\Aut(T)$ as a virtually mixing group. By \cref{remark: FPP invariant under relabeling}, we may assume that $G$ is virtually mixing in $\Aut(T)$ as the labeling does not change the fixed-point proportion. If $f$ is not a dynamical pullback, then the maximal mixing subgroup of $G$ is itself. If we take the identity element $\id$ as a representative of $G/G$, then clearly $\id$ fixes infinitely many ends and by \cref{theorem: FPP virtually mixing groups}, we conclude $\FPP(G) = 0$. 
\end{proof}

\begin{example}
\label{example: group virtually mixing but not mixing}
It is worth mentioning that there are virtually mixing groups that are not mixing and whose fixed-point proportion is still zero. 

Consider $G = \overline{\<g_i: i = 1, \dots,4>}$ where 
\begin{align*}
g_1 = \sigma(g_2,g_2), \quad g_2 = (g_1,g_2), \quad g_3 = (g_4,1), \quad g_4 = \sigma(1,g_3).
\end{align*}

This group is clearly self-similar. Moreover, using induction on the level, one can prove that $g_1^2 = g_2^2 = g_1 \, g_3 \, g_2 \, g_4 = 1$. To prove that the group is fractal observe that the action on the first level is transitive and if $x$ represents the leftmost vertex, then the elements $g_2, g_4^2, g_3 \in \st_G(x)$ and 
\begin{align*}
g_2|_x = g_1, \quad (g_4)^2|_x = g_3 \quad \text{ and } \quad g_3|_x = g_4.
\end{align*}

As $G$ acts on the binary tree and it is fractal, then it is a martingale group by \cite[Corollary 5.11]{BridyJones2022}. 

To prove that it is virtually mixing, we use the commutator trick (\cref{lemma: commutator trick}). As $g_3 = (g_4,1)$, then it has a section $1$ and a section $g_4$ both fixed, which implies that we can see $g_4$. If we take $g_4^4$ then one can check that there are two vertices $u,w \in \mathcal{L}_3$ such that $g_4^4$ fixes them, the section $(g_4)^4|_u = 1$ and $(g_4)^4|_w = g_3$. If we define $H = \overline{\<g_3,g_4>}$, then $H$ is a normal closed mixing subgroup with delay constant $N = 3$. 

We claim that $[G:H] = 2$. First of all, $[G:H] \leq 2$ as $G/H$ is generated by $g_2$. To prove that $[G:H] = 2$, we show that $\pi_3(g_2) \notin \pi_3(H)$. 

\begin{figure}[t]
\centering

\begin{subfigure}{0.32\textwidth}
\centering
\begin{tikzpicture}[
    level 1/.style={sibling distance=18mm, level distance=8mm},
    level 2/.style={sibling distance=10mm, level distance=8mm}
]
\node {$\id$}
    child {
        node {$\sigma$}
        child {node {$\id$}}
        child {node {$\id$}}
    }
    child {
        node {$\id$}
        child {node {$\sigma$}}
        child {node {$\id$}}
    };
\end{tikzpicture}
\caption{Portrait of $g_2$.}
\label{fig:alpha}
\end{subfigure}
\hfill
\begin{subfigure}{0.32\textwidth}
\centering
\begin{tikzpicture}[
    level 1/.style={sibling distance=18mm, level distance=8mm},
    level 2/.style={sibling distance=10mm, level distance=8mm}
]
\node {$\id$}
    child {
        node {$\sigma$}
        child {node {$\id$}}
        child {node {$\id$}}
    }
    child {
        node {$\id$}
        child {node {$\id$}}
        child {node {$\id$}}
    };
\end{tikzpicture}
\caption{Portrait of $g_3$.}
\label{fig:beta}
\end{subfigure}
\hfill
\begin{subfigure}{0.32\textwidth}
\centering
\begin{tikzpicture}[
    level 1/.style={sibling distance=18mm, level distance=8mm},
    level 2/.style={sibling distance=10mm, level distance=8mm}
]
\node {$\sigma$}
    child {
        node {$\id$}
        child {node {$\id$}}
        child {node {$\id$}}
    }
    child {
        node {$\id$}
        child {node {$\sigma$}}
        child {node {$\id$}}
    };
\end{tikzpicture}
\caption{Portrait of $g_4$.}
\label{fig:gamma}
\end{subfigure}
\caption{Portraits of $g_2$, $g_3$ and $g_4$.}
\label{figure: portraits of g2 g3 g4}
\end{figure}
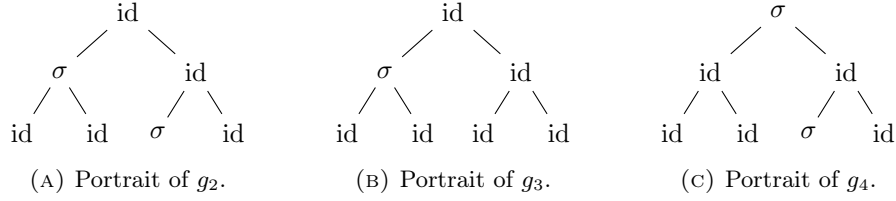

The portraits of the elements $g_2,g_3,g_4$ can be found in \cref{figure: portraits of g2 g3 g4}. If we define 
\begin{align*}
\Par_n: G & \rightarrow \Sym(2) \\ 
\Par_n(g) &:= \prod_{v \in \mathcal{L}_n} g|_v^1,
\end{align*}
then $\Par$ is a group homomorphism and from \cref{figure: portraits of g2 g3 g4}, we clearly have that 
\begin{align*}
\pi_3(H) \leq K_3 := \set{g \in \pi_3(G): \Par(g,1) = \Par(g,3)}
\end{align*}
(in fact one can prove that they are the same). However $\pi_3(g_2) \notin K_3$. Finally, as $H$ is project invariant, we conclude $H = G_M$ by \cref{proposition: project invariant mixing subgroup is maximal}.

This does not prevent us of proving that $\FPP(G) = 0$. To apply \cref{theorem: FPP virtually mixing groups}, we need to find one element in $H$ and one element in $G \setminus H$ fixing infinitely many ends. In the case of $H$, we just take $\id$. In the case of $G \setminus H$, consider $h = g_2 g_3$. Then, one can check that
\begin{align*}
h = (g_2,h,g_1,g_2)_2.
\end{align*}
As $h$ repeats itself and $g_2$ fixes one end, this shows that $X_n(h) \xrightarrow[n \rightarrow +\infty]{} \infty$.

In a discussion with Rafe Jones, we believe that the group $G$ corresponds to the iterated monodromy group of a rational function of the form
\begin{align*}
f(x) = 1 + C \frac{x^2}{1-x},
\end{align*}
where $C$ is either $3/4$ or $-1/4$. The ramification portrait of these two cases is the same and it can be seen in \cref{figure: ramification portrait Jones anlyis}.

\begin{figure}[h!]
\begin{tikzpicture}[
>=Stealth,
every node/.style={inner sep=0pt},
arr/.style={
->,
line width=0.9pt,
shorten <=7pt,
shorten >=7pt},
scale=0.65]

\coordinate (c1) at (0,0);
\coordinate (p1) at (2.4,0);
\coordinate (p2) at (4.8,0);
\coordinate (p3) at (7.2,0);
\coordinate (p4) at (9.6,0);

\draw (c1) circle (0.22);
\fill (p1) circle (0.12);
\fill (p2) circle (0.12);
\draw (p3) circle (0.22);
\fill (p3) circle (0.12);
\fill (p4) circle (0.12);

\draw[arr](c1) to node[above=5pt] {$2$} (p1);
\draw[arr](p1) to node[above=5pt] {} (p2);

\draw[arr]
(p2) .. controls +(1.2,0.7) and +(1.2,-0.7)
.. node[midway,right=5pt] {} (p2);

\draw[arr,bend left=25](p3) to node[above=5pt] {$2$} (p4);
\draw[arr,bend left=25](p4) to node[below=5pt] {} (p3);
\end{tikzpicture}
\caption{Ramification portrait of the map in \cref{example: group virtually mixing but not mixing}}
\label{figure: ramification portrait Jones anlyis}
\end{figure}
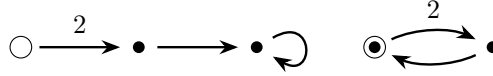
\end{example}

\subsection*{Acknowledgements}

The author would like to thank Rafe Jones for discussions and the invitation to Carleton College, where the discussion on \cref{example: group virtually mixing but not mixing} and the idea of virtually mixing groups arose.


\bibliographystyle{unsrt}

\end{document}